\documentclass[11pt]{amsart}

\usepackage{url}

\usepackage[matrix,arrow,tips,curve]{xy}
\usepackage{pb-diagram,pb-xy}
\usepackage[leqno]{amsmath}
\usepackage{amssymb,amsthm}
\usepackage{amscd}
\usepackage{enumerate}
\usepackage{verbatim}
\usepackage{mathrsfs}
\usepackage{color}
\usepackage{graphicx}
\usepackage{epsfig}

\newenvironment{notation}[0]{%
  \begin{list}%
    {}%
    {\setlength{\itemindent}{0pt}
     \setlength{\labelwidth}{4\parindent}
     \setlength{\labelsep}{\parindent}
     \setlength{\leftmargin}{5\parindent}
     \setlength{\itemsep}{0pt}
     }%
   }%
  {\end{list}}

\input{xy}
\xyoption{all}

\newtheorem{thm}{Theorem}[section]

\newtheorem{lemma}[thm]{Lemma}

\newtheorem{prop}[thm]{Proposition}

\newtheorem{claim}[thm]{Claim}

\newtheorem{cor}[thm]{Corollary}

\theoremstyle{remark}
\newtheorem{remark}[thm]{Remark}

\theoremstyle{definition}

\theoremstyle{remark}
\newtheorem{example}[thm]{Example}

\def\L{\mathcal L}

\def\X{\mathcal X}

\newcommand{\ZZ}{\mathbb{Z}}
\newcommand{\R}{\mathcal{R}}
\newcommand{\RR}{\mathbb{R}}
\newcommand{\PP}{\mathbb{P}}

\newcommand{\C}{\mathfrak{C}}

\newcommand{\ord}{\mathrm{ord}}
\newcommand{\Red}{\mathrm{Red}^{\C}}

\newcommand{\Div}{\operatorname{Div}}
\newcommand\fX{\mathfrak X}
\newcommand{\Prin}{\operatorname{Prin}}
\renewcommand{\div}{\mathrm{div}}

\newcommand{\Pic}{\operatorname{Pic}}
\renewcommand{\k}{\kappa}
\newcommand{\K}{\mathbb K}

\newcommand{\Spec}{\operatorname{Spec}}
\newcommand{\g}{\mathfrak g}

\newcommand{\f}{\mathfrak f}

\newcommand{\mr}{\mathcal M}

\newcommand{\an}{\operatorname{an}}
\newcommand{\ad}{\operatorname{ad}}
\newcommand{\onto}{\twoheadrightarrow}  
     
\newcommand{\red}{\mathrm{red}}
\newcommand{\D}{\mathcal D}
\newcommand{\E}{\mathcal E}
\newcommand{\F}{\mathcal F}
\renewcommand{\H}{\mathcal H}

\title[Linear series on metrized complexes of algebraic curves]{Linear series on metrized complexes of algebraic curves}

\author{Omid Amini}
\address{CNRS - DMA, \'Ecole Normale Sup\'erieure, Paris}
\email{oamini@math.ens.fr}
\author{Matthew Baker}
\address{School of Mathematics, Georgia Institute of Technology}
\email{mbaker@math.gatech.edu}

\begin{document}

\begin{abstract}
A {\em metrized complex of algebraic curves} over an algebraically closed field $\kappa$ is, roughly speaking, a finite metric graph $\Gamma$ together with a collection of marked complete nonsingular algebraic curves $C_v$ over $\kappa$, one for each vertex $v$ of $\Gamma$; the marked points on $C_v$ are in bijection with the edges of $\Gamma$ incident to $v$.
We define linear equivalence of divisors and establish a Riemann-Roch theorem for metrized complexes of curves which combines the classical Riemann-Roch theorem over $\kappa$ with its graph-theoretic and tropical analogues
from \cite{AC,BN,GK,MZ}, providing a common generalization of all of these results.
For a complete nonsingular curve $X$ defined over a non-Archimedean field $\K$, together with a strongly semistable model $\fX$ for $X$ over the valuation ring $R$ of $\K$, we define a corresponding metrized complex $\C\fX$ of curves over the residue
field $\kappa$ of $\K$ and a canonical specialization map $\tau^{\C\fX}_*$
from divisors on $X$ to divisors on $\C\fX$ which preserves degrees and linear equivalence.
We then establish generalizations of the specialization lemma from \cite{bakersp} and its weighted graph analogue 
from \cite{AC}, showing that the rank of a divisor cannot go down under specialization from $X$ to $\C\fX$.
As an application, we establish a concrete link between specialization of divisors from curves to metrized complexes
and the Eisenbud-Harris theory \cite{EH86} of limit linear series.  Using this link, we formulate a generalization of the notion of limit linear series to curves which are not necessarily of compact type and prove, among other things, that any degeneration of a $\g^r_d$ in a regular family of semistable curves is a limit $\g^r_d$ on the special fiber.
\end{abstract}

\date{\today}

\maketitle

\setcounter{tocdepth}{1}
\tableofcontents

\thanks{{\it Acknowledgments}: 
 The authors would like to thank  Vladimir Berkovich, Erwan Brugall\'e, 
Lucia Caporaso, Ethan Cotterill, Eric Katz, Johannes Nicaise, Joe Rabinoff, Frank-Olaf Schreyer, David Zureick-Brown, and the referees for helpful discussions and remarks. The second author was supported in part by NSF grant DMS-0901487. }

\section{Introduction}
We begin by quoting the introduction from the groundbreaking paper \cite{EH86} of Eisenbud and Harris:

\medskip

\begin{itemize}
\item[]
``One of the most potent methods in the theory of (complex projective algebraic)
curves and their linear systems since the work of Castelnuovo has been
degeneration to singular curves...
Most problems of interest about curves are, or can be, formulated in terms of
(families of) linear series. Thus, in order to use degenerations to reducible curves for
studying smooth curves, it is necessary to understand what happens to linear series
in the course of such a degeneration, and in particular, to understand what structure
on a reducible curve plays the part of a linear series.''
\end{itemize}

This has been very successful in practice and has led to important advances in the study of 
algebraic curves (see \cite{HM} for an overview); we mention for example Brill-Noether theory~\cite{EH86, Bi}, 
the geometry of the moduli space of curves~\cite{HM82, EH87}, and Weierstrass points and their 
monodromy~\cite{EH87bis, EM, EH4}.
In view of the Deligne-Mumford compactification of $\mathcal M_g$ and
the stable reduction theorem~\cite{DM}, it is 
particularly natural to consider degeneration of smooth curves to stable (or more generally semistable) curves.
Eisenbud and Harris were able to settle a number of longstanding open problems with their theory of limit linear series on such curves,
and the theory has undergone numerous developments and improvements in the last 25 years.
However, the theory of limit linear series only applies, for the most part, to a rather restricted class of
reducible curves, namely those of {\em compact type} (i.e., nodal curves whose dual graph is a tree).

\medskip

In \cite{bakersp}, the second author introduced a new framework for degenerating linear series on curves, degenerating linear series on a regular semistable family of curves to a linear series on the {\em dual graph} of the special fiber.
This theory, which is closely related to tropical geometry and also to the theory of Berkovich analytic spaces, is more or less orthogonal to the Eisenbud-Harris theory, in that it works best for special fibers which are 
{\em maximally degenerate}, meaning that the dual graph has first Betti number equal to the genus of 
the generic fiber.  Specializing linear series to the dual graph provides no information whatsoever when the special fiber is of compact type.  Intriguingly, both the Eisenbud-Harris theory and the second author's theory from \cite{bakersp} lead to
simple proofs of the celebrated Brill-Noether theorem of Griffiths-Harris (see \cite{EH86,CDPR}).

\medskip

The aim of the present paper is to introduce a theoretical framework suitable for generalizing both the Eisenbud-Harris theory of limit linear series and the second author's theory of specialization from curves to graphs.   
The main new object of study is what we call a {\em metrized complex
of algebraic curves} over an algebraically closed field $\kappa$; this is, roughly speaking, a finite metric graph $\Gamma$ together with a collection of marked complete nonsingular algebraic curves $C_v$ over $\kappa$, one for each vertex $v$ (with a marked point on $C_v$ for each edge of $\Gamma$ incident to $v$).
We define divisors, linear equivalence of divisors, and rank of divisors on metrized complexes of curves and establish a
Riemann-Roch theorem for metrized complexes of curves which combines the classical Riemann-Roch theorem over 
$\kappa$ with its graph-theoretic and tropical analogues proved in \cite{AC,BN, GK, MZ}.
For a curve $X$ defined over a non-Archimedean field $\K$, together with a strongly semistable model $\fX$ for $X$
over the valuation ring $R$ of $\K$, we define a corresponding metrized complex $\C\fX$ of curves over the residue
field $\kappa$ of $\K$ and a specialization map $\tau^{\C\fX}_*$
from divisors on $X$ to divisors on $\C\fX$ preserving degrees and linear equivalence.
We then establish generalizations of the specialization lemma from \cite{bakersp} and its weighted graph analogue 
from \cite{AC}, showing that the rank of a divisor cannot go down under specialization from $X$ to $\C\fX$.
As an application, we establish a concrete link between specialization of divisors from curves to metrized complexes
and the Eisenbud-Harris theory of limit linear series.  Using this link, we formulate a generalization of the notion of
limit linear series to curves which are not necessarily of compact type and prove, among other things, that any degeneration of a $\g^r_d$ in a regular family of semistable curves is a limit $\g^r_d$ on the special fiber.

\medskip

As mentioned above, the specialization theory from \cite{bakersp} works best in the case where the 
family is maximally degenerate; in other cases, the dual graph of the family forgets too much information.
This loss of information was partially remedied in \cite{AC} by looking at vertex-weighted graphs, where the
weight attached to a vertex is the genus of the corresponding irreducible component of the special fiber.
The theory of metrized complexes of curves offers a richer solution in which one keeps track of the (normalization 
of the) irreducible components $C_v$ of the special fiber themselves, and not just their genera.  From the point of view of non-Archimedean geometry, this corresponds to retracting divisors on $X$ to the skeleton
$\Gamma_{\fX}$, a metric graph canonically embedded in the Berkovich analytic space $X^{\an}$, but also keeping track, for points retracting to a vertex of $\Gamma_{\fX}$, of the tangent direction in which the vertex is
approached.  This is quite natural from the valuation-theoretic point of view: the metric graph $\Gamma_{\fX}$ associated to a strongly semistable model $\fX$ is a canonical subset of the Berkovich analytic space $X^{\an}$, and the curve 
$C_v$ over $\kappa$ is naturally identified (for each $v \in V$) with the fiber over $v$ of the canonical retraction map $X^{\ad} \to X^{\an}$ (which sends a continuous valuation of arbitrary rank in the Huber adic space $X^{\ad}$ to its canonical rank-1 generalization), c.f. \cite[Remark 2.6]{Te}.
Though we could possibly (with quite a bit of additional effort) have formulated many of our theorems and proofs without 
mentioning Berkovich spaces, it would have resulted in a significant loss of elegance and clarity, and in any case
Berkovich's theory seems ideally suited for the point of view taken here.  We do not, however, require any non-trivial
facts from Huber's theory of adic spaces, so we will not mention them again even though they are certainly lurking in the
background.  

\medskip

There are a number of connections between the ideas in the present paper and tropical geometry; these will be explored
in more detail in future work.  For example, the theory of morphisms between metrized complexes of curves
studied in \cite{ABBR} sheds interesting new light on the question of which morphisms
between tropical curves are liftable.  The theory developed in \cite{BPR} shows that there is a close connection between 
metrized complexes of curves and ``exploded tropicalizations'' in the sense of \cite[Definition 2.9]{Pay}; we plan to say more about this in the future as well.  As noted by Payne in \cite{Pay}, exploded tropicalizations can be thought of as an 
algebraic analogue of the ``exploded torus fibrations'' studied by Parker from a symplectic viewpoint \cite{Par}.
It could be interesting to explore connections between {\it loc.~cit.} and the present work.

\medskip

The ideas in the present paper also have Diophantine applications to the study of rational points on curves over number fields (specifically, to the method of Coleman-Chabauty).
Indeed, Eric Katz and David Zureick-Brown have recently proved a result
similar to Theorem~\ref{thm:MCClifford} below (Clifford's theorem for metrized complexes), and they use this result
to answer a question of M.~Stoll.  A special case of the main result in \cite{KZB} is the following.  Let $X$ be a smooth projective geometrically irreducible curve of genus $g \geq 2$ over ${\mathbf Q}$ and assume that the Mordell-Weil rank $r$ of the Jacobian of $X$ is less than $g$.  
Fix a prime number $p > 2r+2$ and let $\fX$ be a proper (not necessarily semistable) regular model for $X$ over 
${\mathbf Z}_p$.  Then (letting $\bar{\fX}^{{\rm sm}}$ denote the smooth locus of $\bar{\fX}$)
\[
\# X({\mathbf Q}) \leq \bar{\fX}^{{\rm sm}}({\mathbf F}_p) + 2r.
\]

We discuss the theorem of Katz and Zureick-Brown in more detail in Section~\ref{ChabautySection} below,
explaining how limit linear series on metrized complexes of curves can be used to illuminate their proof and put it into a broader context.

\subsection{Notation}
\label{section:Notation}

We set the following notation, which will be used
throughout this paper.

\begin{notation}
\item[$G$]
a finite edge-weighted graph with vertex set $V$ and edge set $E$
\item[$\Gamma$]
a compact metric graph (geometric realization of an edge-weighted graph $G$)
\item[$\K$]
a complete and algebraically closed non-trivially valued non-Archimedean field
\item[$R$]
the valuation ring of $\K$
\item[$\k$]
the residue field of $\K$
\item[$X$]
a smooth, proper, connected algebraic curve defined over $\K$ 
\item[$X^{\mathrm{an}}$]
the Berkovich analytification of $X/\K$
\item[$\mathfrak X$]
a proper and flat semistable $R$-model for $X$, usually assumed to be strongly semistable
\item[$C_v$]
a smooth, proper, connected algebraic curve over $\k$ (one for each vertex $v\in V$)
\item[$\C$]
a metrized complex of algebraic curves over $\k$
\item[$|\C|$]
the geometric realization of $\C$
\item[$\C \mathfrak X$]
the metrized complex associated to a strongly semistable $R$-model $\mathfrak X$ of $X$
\item[$X_0$]
a nodal curve defined over $\k$
\item[$X_v$] 
an irreducible component of $X_0$
\item[$\C X_0$]
the regularization of the nodal curve $X_0$ 
\item[$\Div(\C)$]
the group of divisors on the metrized complex $\C$
\item[$\mathcal D$]
a divisor on $\C$
\item[$\mathcal E$]
an effective divisor on $\C$
\item[$\f$]
a rational function on $\C$
\item[$D_v$]
a divisor on $C_v$
\item[$E_v$]
an effective  divisor on $C_v$
\item[$f_\Gamma$]
a rational function on $\Gamma$, usually the $\Gamma$-part of a rational function $\f$ on $\C$
\item[$f_v$]
a rational function on $C_v$, usually the $C_v$-part of a rational function $\f$ on $\C$
\item[$\mathrm{slp}_e(f_\Gamma)$]
the outgoing slope of $f_\Gamma$ along an edge (or tangent direction) $e$
\item[$r_\C$]
the rank function on $\Div(\C)$
\item[$r_{\C,\F}$]
the restricted rank function with respect to a collection $\F = \{ F_v \}_{v\in V}$ of subspaces of $\k(C_v)$
\item[$\tau$]
the retraction map from $X^{\rm an}$ to the skeleton $\Gamma_{\mathfrak X}$ of $\mathfrak X$
\item[$\tau^{\C\fX}_*$]
the specialization map from $\Div(X)$ to $\Div(\C\fX)$
\item[$\mathcal A_v$] 
the marked points of $C_v$ in the metrized complex $\C$
\item[$x^e_v$]
the point of $\mathcal A_v$ corresponding to the edge (or tangent direction) $e$ 
\item[$A_v$]
the sum in $\Div(C_v)$ of the points of $\mathcal A_v$ 
\end{notation}

\subsection{Overview}
\label{section:overview}

We now discuss the contents of this paper in more detail.   

\medskip

An {\em edge-weighted graph} $G$ is a connected multigraph, possibly with loop edges, having vertex set $V$ and
edge set $E$, and endowed with a weight (or length) function $\ell : E \to \RR_{> 0}$.
A {\em metric graph} $\Gamma$ is the geometric realization of an edge-weighted graph $G$ in which each edge $e$ of
$G$ is identified with a line segment of length $\ell(e)$.  We call $G$ a {\em model} for $\Gamma$.
Subdividing an edge of $G$ in a length-preserving fashion changes the model but not the underlying metric graph.

\medskip

Let $\kappa$ be an algebraically closed field.\footnote{To simplify the presentation, we restrict to algebraically 
closed fields. 
In \S{2.3} of the January 2013 arXiv version of this paper, 
we indicate how the theory can be developed over an arbitrary field $\kappa$.}
A {\em metrized complex $\C$ of $\kappa$-curves} consists of the
following data:

\begin{itemize}
\item A connected finite graph $G$ with vertex set $V$ and edge set $E$.  
\item A metric graph $\Gamma$ having $G$ as a model (or equivalently, a length function $\ell : E \to \RR_{>0}$).  
\item For each vertex $v$ of $G$, a complete, nonsingular, irreducible curve $C_v$ over $\kappa$.
\item For each vertex $v$ of $G$, a bijection $e \mapsto x^e_v$ between the edges of $G$ incident to $v$ (with loop edges counted twice) and a subset 
$\mathcal A_v = \{ x^e_v \}_{e \ni v}$ of $C_v(\kappa)$.
\end{itemize}

\medskip

 \begin{figure}[!tb]
\includegraphics[width=6cm]{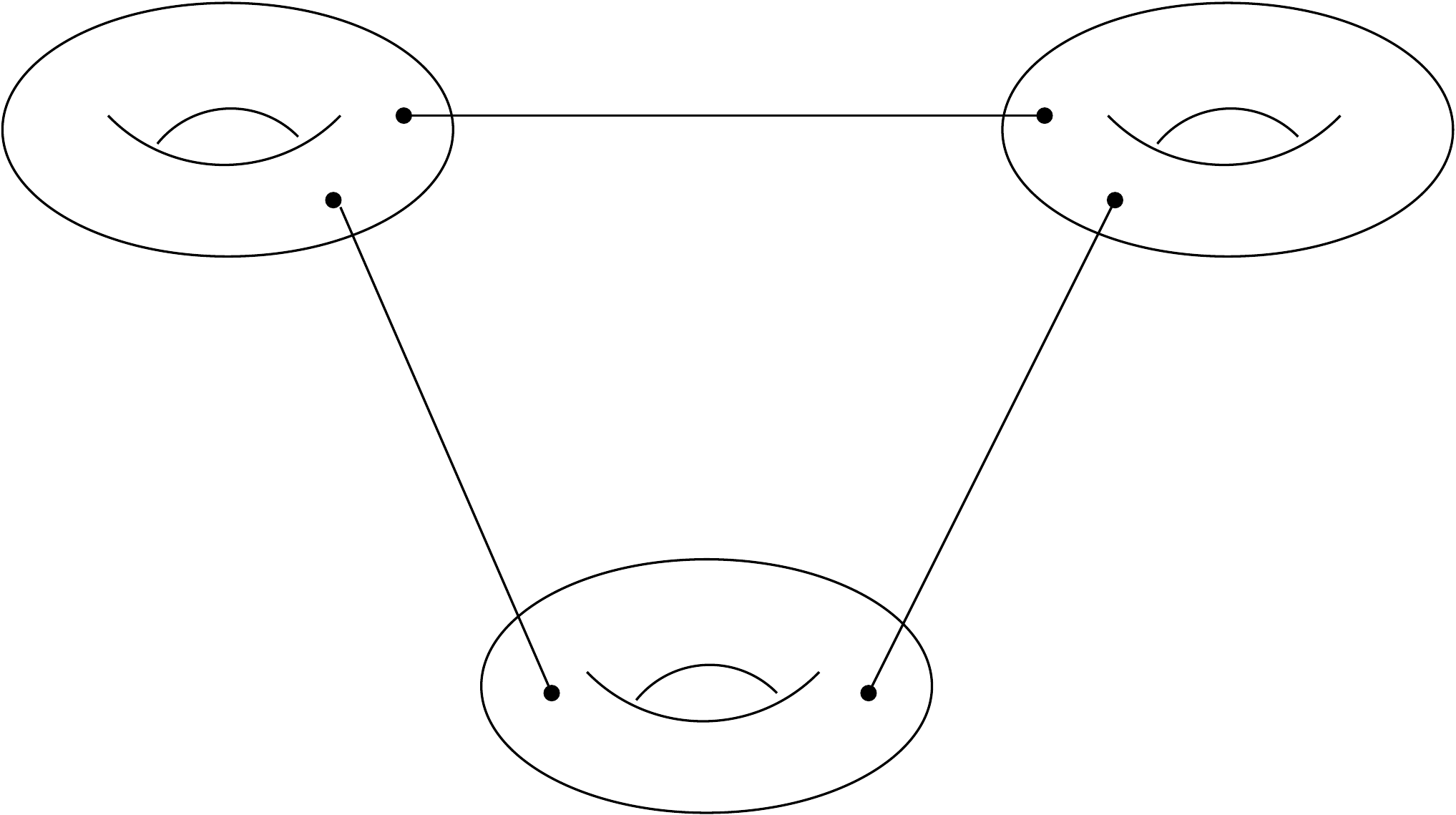}
 \caption{The geometric realization of a metrized complex of genus four.}
 \label{fig:mcsmall}
 \end{figure}

The {\em geometric realization} $|\C|$ of $\C$ is defined to be the union of the edges of $G$
and the collection of curves $C_v$, with each endpoint $v$ of an edge $e$ identified with the corresponding marked point $x^e_v$.
See Figure~\ref{fig:mcsmall}.
When we think of $|\C|$ as a set, we identify it with the disjoint union of $\Gamma \backslash V$ and $\bigcup_{v \in V} C_v(\kappa)$.  
Thus, when we write $x \in |\C|$, we mean that $x$ is either a non-vertex point of $\Gamma$ (a {\em graphical point} of $\C$) 
or a point of $C_v(\kappa)$ for some $v \in V$ (a {\em geometric point} of $\C$).

\medskip

The {\em genus} of a metrized complex of curves $\C$, denoted $g(\C)$, is by definition 
$g(\C)=g(\Gamma)+\sum_{v\in V}g_v$, where $g_v$ is the genus of $C_v$ and $g(\Gamma)$ is the first Betti number of $\Gamma$.

\medskip

Given a metrized complex $\C$ of $\kappa$-curves, there is an associated semistable curve
$X_0$ over $\kappa$ obtained by
gluing the curves $C_v$ along the points $x^e_v$ (one intersection for each edge $e$ of $G$) and
forgetting the metric structure on $\Gamma$.  Conversely, given a semistable curve $X_0$ over $\kappa$ together with a positive real number for each node (which we call a ``length function''), one obtains an associated metrized complex of $\kappa$-curves by letting $G$ be the dual graph of $X_0$, $\Gamma$ the metric graph associated to $G$ and the given length function, $C_v$ the normalization of the irreducible component $X_v$ of $X_0$ corresponding to $v$,
and $\mathcal A_v$ the preimage in $C_v$ of the set of nodes of $X_0$ belonging to $X_v$.

\medskip

Let $\K$ be a complete and algebraically closed non-Archimedean field with residue field $\kappa$ and
let $X/\K$ be a smooth, proper, connected algebraic curve.
There is a metrized complex $\C = \C\fX$ canonically associated to any strongly semistable model $\fX$ of $X$ over the valuation ring $R$ of $\K$: the special fiber $\bar{\fX}$ of $\fX$ is a semistable curve over 
$\kappa$, and one defines the length of an edge $e$ of the dual graph of $\bar{\fX}$ to be the {\em modulus} of the open annulus ${\rm red}^{-1}(z)$, where $z$ is the singular point of $\bar{\fX}$ corresponding to $e$ and ${\rm red} : X(\K) \to \bar{\fX}(\kappa)$ is the canonical reduction map.
Equivalently, there is a local analytic equation for $z$ of the form $xy=\varpi$ with $\varpi$ in the maximal ideal of $R$, and the length of $e$ is ${\rm val}(\varpi)$.
The corresponding metric graph $\Gamma = \Gamma_{\fX}$ is called the {\em skeleton} of $\fX$,
and there is a canonical retraction map $\tau = \tau_{\fX}: X^{\an} \onto \Gamma_{\fX}$, where $X^{\an}$ denotes the Berkovich analytification of $X$.
By restricting to $X(\K)$ and then extending by linearity, we obtain a specialization map
$\tau_* : \Div(X) \to \Div(\Gamma_{\fX})$.

\medskip

If $R$ is a discrete valuation ring and the fibered surface $\fX/R$ is regular, then one can use intersection-theoretic
methods to define and study specialization of divisors from curves to graphs. 
For example, in \cite{bakersp}, using intersection theory, a specialization homomorphism $\rho : \Div(X) \to \Div(G)$ was defined, where $G$ is the dual graph of $\bar{\fX}$, with the property that $\rho(D)$ and $\tau_*(D)$ are linearly equivalent as divisors on $\Gamma_{\fX}$ for all $D \in \Div(X)$.  In the present paper, we focus exclusively on $\tau_*$ and its
generalization to metrized complexes of curves, though in principle all of our results can be translated (when $R$
is a discrete valuation ring and $\fX$ is regular) into results about this alternative specialization map.

\medskip

A {\em divisor} on a metrized complex of curves $\C$ is an element $\mathcal D$ of the free abelian group on $|\C|$.
Thus a divisor on $\C$ can be written uniquely as $\mathcal D = \sum_{x \in |\C|} a_x (x)$ where $a_x \in \ZZ$, all but finitely many of the $a_x$ are zero,
and the sum is over all points of $\Gamma \backslash V$ as well as $C_v(\kappa)$ for $v \in V$.
The {\em degree} of $\mathcal D$ is defined to be $\sum a_x$.

\medskip

To a divisor on $\C$, we can naturally associate a divisor $D_\Gamma$ on $\Gamma$, called the {\em $\Gamma$-part} of $\D$, as well as, for each $v \in V$, a 
divisor $D_v$ on $C_v$ (called the {\em $C_v$-part} of $\D$).
The divisor $D_v$ is simply the restriction of $\D$ to $C_v$, i.e. $D_v = \sum_{x \in C_v(\kappa)} \mathcal D(x)(x)$, and $D_\Gamma$ is defined as
\[
D_\Gamma = \sum_{x \in \Gamma \backslash V} \D(x) (x) + \sum_{v \in V} \deg(D_v) (v),
\]
where $\mathcal D(x)$ denotes the coefficient of $x$ in $\mathcal D$.

In particular, the degree of $\mathcal D$ equals the degree of $D_\Gamma$.
One could equivalently define a divisor on $\C$ to be an element of the form
$\mathcal D = D_\Gamma \oplus \sum_v D_v$ of $\Div(\Gamma)\oplus (\oplus_v \Div(C_v))$  
such that $\deg(D_v) = D_\Gamma(v)$ for all $v$ in $V$.

\medskip

The specialization map $\tau_*$ can be enhanced in a canonical way to a map from divisors on $X$ to divisors on $\C\fX$. 
The fiber $\tau^{-1}(v)$ of the retraction map $\tau : X(\K) \to \Gamma$ over a vertex $v \in V$ can be
canonically identified with ${\rm red}^{-1}(C_v^{\rm sm}(\kappa))$, the fiber of the reduction map $\red : X(\K) \to
\bar{\fX}(\kappa)$ over the smooth locus $C_v(\kappa) \setminus \mathcal A_v$ of $C_v$.
We define $\tau^{\C\fX} : X(\K) \to \C\fX$ by setting
\[
\tau^{\C\fX} (P) = \left\{
\begin{array}{ll}
\tau(P) & \tau(P) \not\in V \\
{\rm red}(P) & \tau(P) \in V \\
\end{array}
\right.
\]
and we extend this by linearity to a map $\tau_*^{\C\fX} : \Div(X) \to \Div(\C\fX)$.

\medskip

Intuitively speaking, for points $P$ of $X(\K)$ which reduce to smooth points of $\bar{\fX}$, 
the map $\tau^{\C\fX}$ 
keeps track of the reduced point ${\rm red}({P}) \in \bar{\fX}^{\rm sm}(\kappa)$,
 while if $P$ reduces to a singular point of $\bar{\fX}$ then $\tau^{\C\fX}$ instead keeps track of the retraction of $P$ to the skeleton of the open annulus $\red^{-1}(P)$, which is canonically identified with the relative interior of the corresponding edge of $\Gamma_{\fX}$.

\medskip

A {\em nonzero rational function} $\f$ on a metrized complex of curves $\C$ is the data of a rational function $f_\Gamma$ on $\Gamma$ and nonzero rational functions $f_v$ on $C_v$ for each $v\in V$.  (We do {\em not} impose any compatibility
conditions on the rational functions $f_\Gamma$ and $f_v$.)
We call $f_\Gamma$ the {\em $\Gamma$-part} of $\f$ and $f_v$ the {\em $C_v$-part} of $\f$.

\medskip

The {\em divisor} of a nonzero rational function $\f$ on $\C$ is defined to be
\[
\div(\f) := \sum_{x \in |\C|} \ord_x(\f) (x),
\]
where $\ord_x(\f)$ is defined as follows:
\begin{itemize}
\item If $x \in \Gamma \backslash V$, then $\ord_x(\f) = \ord_x(f_\Gamma)$, where 
$\ord_x(f_\Gamma)$ is the sum of the slopes of $f_\Gamma$ in all tangent directions emanating from $x$.
\item If $x \in C_v(\kappa) \backslash \mathcal A_v$, then $\ord_x(\f) = \ord_x(f_v)$.
\item If $x = x^e_v \in \mathcal A_v$, then $\ord_x(\f) = \ord_x(f_v) + \mathrm{slp}_e(f_\Gamma)$, where
$\mathrm{slp}_e(f_\Gamma)$ is the outgoing slope of $f_\Gamma$ at $v$ in the direction of $e$.
\end{itemize}

If $D_\Gamma$ (resp. $D_v$) denotes the $\Gamma$-part (resp. the $C_v$-part) of $\div(\f)$, then
\[ D_\Gamma = \div(f_{\Gamma}) = \sum_{u\in \Gamma} \mathrm{ord}_u(f_\Gamma)\,(u)\qquad \textrm{and} 
\qquad D_v = \div(f_v) + \div_v(f_\Gamma), \qquad \textrm{where}
\]
\begin{equation}
\label{eq:divf2}
\div_v(f_\Gamma) := \sum_{e \ni v} \mathrm{slp}_e(f_\Gamma)(x^e_v).
\end{equation}

Divisors of the form $\div(\f)$ are called {\em principal}, and the principal divisors form a subgroup of 
$\Div^0(\C)$, the group of divisors of degree zero on $\C$.
Two divisors in $\Div(\C)$ are called {\em linearly equivalent} if they differ by a principal divisor.

\medskip

Linear equivalence of divisors on $\C$ can be understood quite intuitively in terms of ``chip-firing moves''. To state this formally, for each vertex $v$ of $G$, let $A_v$ be the sum of the $\deg_G(v)$ points in $\mathcal A_v$ and let
$\ell_v$ be the minimum length of an edge $e$ incident to $v$.  Also, for an edge $e$ having $v$ as an endpoint and a positive real number $\varepsilon < \ell(e)$, let $p^e_{v, \varepsilon}$ be the unique point of $e$ at distance $\varepsilon$ from $v$.  
Then it is not hard to show that linear equivalence of divisors on $\C$ is the equivalence relation generated by the following three kinds of ``moves'':

\begin{enumerate}
\item (Firing on $C_v$) Choose a vertex $v$ and replace the divisor $D_v$ with a linearly equivalent divisor $D'_v$ on $C_v$.
\item (Firing a vertex) Choose a vertex $v$ and a positive real number $\varepsilon < \ell_v$ and replace $D_v$ with 
$D_v - A_v$ and $D_\Gamma$ with $D_\Gamma - \deg_G(v)(v) + \sum_{e \ni v} (p^e_{\varepsilon})$.
\item (Firing a non-vertex) Choose a non-vertex $p \in \Gamma$ and a positive real number $\varepsilon$ less than the distance from $p$ to the
nearest vertex and replace $D_\Gamma$ by the divisor $D_\Gamma - 2({p}) + (p^+_\varepsilon) + (p^-_\varepsilon)$, where $p^{\pm}_\varepsilon$ are the two points at distance $\varepsilon$ from $p$.
\end{enumerate}

\medskip

The motivation for our definitions of $\tau^{\C\fX}_*$ and $\div(\f)$ come from the following fundamental relation,
a consequence of the non-Archimedean Poincar{\'e}-Lelong formula due to Thuillier (see \S{5} of \cite{BPR}).  
Let $f$ be a nonzero rational function
on $X$ and let $\f$ be the corresponding nonzero rational function on $\C\fX$, where $f_{\Gamma}$ is the restriction
to $\Gamma$ of the piecewise linear function $\log|f|$ on $X^{\an}$ and $f_v \in \kappa(C_v)$ for $v \in V$ is the 
normalized reduction of $f$ to $C_v$ (cf.~\S\ref{section:ReductionOfRationalFunctions}).
Then
\[
\tau_*^{\C\fX}(\div(f)) =\div(\f).
\]

In particular, we have $\tau_*^{\C\fX}(\Prin(X)) \subseteq \Prin(\C\fX)$.

\medskip

A divisor $\mathcal E = \sum_{x \in |\C|} {a_x (x)}$ on $\C$ is called {\em effective} if $a_x \geq 0$ for all $x$
(or, equivalently, if the $\Gamma$-part $E_\Gamma$ and the $C_v$-parts $E_v$ of $\mathcal E$ are effective for every $v\in V$). The {\em rank} $r_\C$ of a divisor $\mathcal D \in \Div(\C)$ is defined to 
be the largest integer $k$ such that $\mathcal D - \mathcal E$ is linearly equivalent to an effective divisor
for all effective divisors $\mathcal E$ of degree $k$ on $\C$ (so in particular
$r_{\C}(\mathcal D) \geq 0$ if and only if $\mathcal D$ is linearly equivalent to
an effective divisor, and otherwise $r_{\C}(\mathcal D)=-1$).

\medskip

If $r_X(D)$ denotes the usual rank function $r_X(D) = \dim |D| = h^0(D) - 1$ on $\Div(X)$, 
there is an important semicontinuity result relating $r_X$ to $r_{\C\fX}$:

\begin{thm}[Specialization Theorem]
\label{thm:introSpecialization}
For all $D \in \Div(X)$, we have 
\[
r_X(D) \leq r_{\C\fX}(\tau_*^{\C\fX}(D)).
\]
\end{thm}

Since $r_{\C\fX}(\tau_*^{\C\fX}(D)) \leq r_\Gamma(\tau_*(D))$, Theorem~\ref{thm:introSpecialization}
is a strengthening of the analogous specialization result from \cite{bakersp}.
In conjunction with a simple combinatorial argument, Theorem~\ref{thm:introSpecialization} also refines the specialization lemma for vertex-weighted graphs from \cite{AC}. 

\medskip

We also state and prove a version of Theorem~\ref{thm:introSpecialization} in which one has {\em equality} rather than just an inequality.
One can naturally associate to a rank $r$ divisor $D$ on $X$ not only a divisor $\tau^{\C \fX}_*(D)$ on $\C\fX$, but also a collection 
$\H = \{ H_v \}_{v \in V}$ of $(r+1)$-dimensional subspaces of $\k(C_v)$, where $H_v$ is the reduction to $C_v$ of all functions in $H^0(X, \mathcal L(D))$ (c.f. Section~\ref{section:ReductionOfRationalFunctions} for a discussion of reduction of rational functions).
If $\F = \{ F_v \}_{v \in V}$, where $F_v$ is any $\kappa$-subspace of the function field $\kappa(C_v)$,
then for $\D \in \Div(\C)$ we define the {\em $\F$-restricted rank} of $\D$, denoted $r_{\C,\F}(\D)$, to be the largest integer $k$
such that for any effective divisor $\E$ of degree $k$ on $\C$, there is a rational function $\f$ on $\C$ whose $C_v$-parts
$f_v$ belong to $F_v$ for all $v \in V$, and such that $\D - \E + \div(\f) \geq 0$.
(In terms of the ``chip-firing moves'' described above, one has to be able to make $\D-\E$ effective using only type (1) moves corresponding to functions belonging to
one of the spaces $F_v$, together with moves of type (2) or (3).)
The following variant of Theorem~\ref{thm:introSpecialization} will be proved as Theorem~\ref{thm:limitseriesgeneral} below:

\begin{thm}[Specialization Theorem for Restricted Ranks]
With notation as above, the $\H$-restricted rank of the specialization of $D$ is equal to the rank of $D$, i.e.,
$r_{\C,\H}(\tau^{\C \fX}_*(D)) = r$.
\end{thm}

\medskip

The theory of divisors and linear equivalence on metrized complexes of curves generalizes both the classical theory
for algebraic curves and the corresponding theory for metric graphs and tropical curves found in \cite{GK,MZ}.
The former corresponds to the case where $G$ consists of a single vertex $v$ and no edges and $C=C_v$ is an arbitrary smooth curve. The latter corresponds to the case where the curves $C_v$ have genus zero for all $v \in V$.
Since any two points on a curve of genus zero are linearly equivalent, it is easy to see that the divisor theories and rank functions on $\C$ and $\Gamma$ are essentially equivalent. More precisely, (i) two divisors $\mathcal D$ and  $\mathcal D'$ on $\C$ are linearly equivalent if and only if their $\Gamma$-parts 
$D_\Gamma$ and $D'_\Gamma$ are linearly equivalent on $\Gamma$, and (ii) for every $\mathcal D\in \Div(\C)$ 
we have $r_\C(\mathcal D) = r_\Gamma(D_\Gamma)$. 
In the presence of higher genus curves among the $C_v$, the divisor theories on $\Gamma$ and $\C$ can be very different. In addition, different choices of $\mathcal A_v$ can drastically change both the linear equivalence relation and
the rank function.

\medskip

As with the corresponding specialization theory from \cite{bakersp}, the main utility of 
Theorem~\ref{thm:introSpecialization} is that the rank function $r_{\C}$ on a metrized complex of curves
is surprisingly well-behaved; for example, it satisfies an analogue of the Riemann-Roch formula.
In order to state this result, we need to define the canonical class.  
A {\em canonical divisor} on $\C$, denoted $\mathcal K$, is defined to be any divisor linearly equivalent to 
$\sum_{v \in V} (K_v + A_v)$, where $K_v$ is a canonical divisor on $C_v$ and $A_v$ is the sum of the $\deg_G(v)$ points in $\mathcal A_v$. 
The $\Gamma$-part of $\mathcal K$ is $K^\# = \sum_v \left( \deg_G(v) + 2g_v - 2 \right) (v)$, and the $C_v$-part of $\mathcal K$ is 
$K_v + A_v$.

\medskip

The following result generalizes both the classical Riemann-Roch theorem for algebraic curves and the Riemann-Roch theorem for metric graphs:

\begin{thm}[Riemann-Roch for metrized complexes of algebraic curves]\label{thm:introRR-metrizedcomplexes}
Let $\C$ be a metrized complex of algebraic curves over $\kappa$ and $\mathcal K$ a divisor in the canonical class of $\C$. For any divisor $\mathcal D \in \Div(\C)$, we have 
\[r_\C(\mathcal D) - r_\C(\mathcal K -\mathcal D) = \deg(\mathcal D) - g(\C)+1. \]
\end{thm}

The proof of Theorem~\ref{thm:introRR-metrizedcomplexes} makes use of a suitable notion of {\em reduced divisors}
for metrized complexes of curves.

\medskip

Taken together, Theorems \ref{thm:introSpecialization} and \ref{thm:introRR-metrizedcomplexes} have some 
interesting consequences.  For example, the specialization theorem implies that the specialization to $\C\fX$ of any
canonical divisor on $X$ has degree $2g-2$ and rank at least $g-1$ (where $g = g(X) = g(\C\fX)$), while the Riemann-Roch theorem easily implies that any such divisor on $\C\fX$ belongs to the canonical divisor class.  Thus we have:

\begin{cor}
\label{cor:introcanonspec}
For any canonical divisor $K_X$ on $X$, its specialization $\tau^{\C\fX}_*(K_X)$ belongs to the canonical divisor
class on $\C\fX$.
\end{cor}

In particular, Corollary~\ref{cor:introcanonspec} implies that $\tau_*(K_X)$ is a canonical divisor on $\Gamma$.  
This fact was proved in \cite{bakersp}, for discretely valued $R$, by a completely different argument based on the adjunction formula for arithmetic surfaces.  
It does not seem easy to prove the general (not necessarily discretely valued) case using arithmetic intersection theory, since there does not seem to be a satisfactory theory of relative dualizing sheaves and adjunction in the non-Noetherian setting.

\medskip

Our theory of linear series on metrized complexes of curves has close connections with the Eisenbud-Harris theory
of limit linear series for strongly semistable curves of compact type, and suggests a way to generalize the Eisenbud-Harris theory to more general semistable curves.
A proper nodal curve $X_0$ over $\kappa$ is of {\em compact type} if its dual graph $G$ is a tree.
For such curves, Eisenbud and Harris define a notion of (crude) {\em limit $\g^r_d$}, which we discuss in detail in 
Section~\ref{sec:limitcompact}  below.
A {\it crude limit} $\g^r_d$ $L$ on $X_0$ is the data of a (not necessarily complete) degree $d$ and rank $r$ linear series $L_v$ on $X_v$ for each vertex $v \in V$ such that if two components $X_u$ and $X_v$ of $X_0$ meet at a node $p$, then for any $ 0\leq i \leq r,$
\begin{equation*}
 a^{L_v}_{i}( p ) + a^{L_u}_{r-i}( p ) \, \geq \, d\,,
 \end{equation*}
where $a^{L}_i (p)$ denotes the $i^{\rm th}$ term in the vanishing sequence of a linear series $L$ at $p$.
A crude limit series is {\it refined} if all the inequalities  in the above definition are equalities.  
For simplicity, all limit linear series in the remainder of this introduction will be crude.

\medskip

We can canonically associate to a proper strongly 
 curve $X_0$ a metrized complex $\C X_0$ of $\kappa$-curves, called the {\em regularization} of $X_0$, by assigning a length of $1$ to each edge of $G$, and we write $X_v$ for the irreducible component of $X_0$ corresponding to a vertex $v \in V$. (This is the metrized complex associated to any regular smoothing $\fX$ of $X_0$ over any discrete valuation ring $R$ with residue field $\kappa$.) 

\begin{thm}
\label{thm:introLLS}
Let $\C X_0$ be the metrized complex of curves associated to a proper strongly semistable curve $X_0 / \kappa$ of compact
type.  Then there is a bijective correspondence between the following:

\begin{itemize}
\item Crude limit $\g^r_d$'s on $X_0$.
\item Equivalence classes of pairs $(\H,\D)$, where $\H = \{ H_v \}$, $H_v$ is an $(r+1)$-dimensional subspace of $\kappa(X_v)$ for each $v \in V$, and 
$\D$ is a divisor of degree $d$ supported on the vertices of $\C X_0$ with 
$r_{\C X_0,\H}(\D) = r$.  Here we say that $(\H,\D) \sim (\H', \D')$ if there is a rational function $\f$
on $\C X_0$ such that $D' = D + \div(\f)$ and $H_v = H'_v \cdot f_v$ for all $v \in V$, where $f_v$ denotes the $C_v$-part of $\f$.
\end{itemize}
\end{thm}

Theorem~\ref{thm:introLLS}, combined with our Riemann-Roch theorem for metrized complexes of curves, provides an arguably more conceptual proof of the fact (originally established in \cite{EH86}) that limit linear series satisfy 
analogues of the classical theorems of Riemann and Clifford.
The point is that $r_{\C X_0, \H}(\D) \leq r_{\C X_0}(\D)$ for all $\D \in \Div(\C X_0)$ and therefore upper
bounds on $r_{\C X_0}(\D)$ which follow from Riemann-Roch imply corresponding upper bounds on the restricted rank
$r_{\C X_0, \H}(\D)$.

\medskip

Motivated by Theorem~\ref{thm:introLLS}, we propose the following definition.
Let $X_0$ be a proper strongly semistable (but not necessarily compact type) 
curve over $\kappa$ with regularization $\C X_0$.
A {\em limit $\g^r_d$} on $X_0$ is an equivalence class of pairs $(\H = \{ H_v \},\D)$ as above,
where $H_v$ is an $(r+1)$-dimensional subspace of $\kappa(X_v)$ for each $v \in V$,
and $\D$ is a degree $d$ divisor on $\C X_0$ with $r_{\C X_0,\H}(\D) = r$.

\medskip

As a partial justification for this definition, we prove that if $R$ is a discrete valuation ring with residue field $\kappa$ and 
$\fX/R$ is a regular arithmetic surface whose generic fiber $X$ is smooth and whose special fiber $X_0$ is strongly semistable, then for any divisor $D$ on $X$ with $r_X(D)=r$ and $\deg(D) = d$, our specialization machine 
gives rise in a natural way to a limit $\g^r_d$ on $X_0$ 
(see Theorem~\ref{thm:regularlimitseries} below for a precise statement, and 
Theorem~\ref{thm:limitseriesgeneral} for a more general statement).

\section{Metrized complexes of algebraic curves}
\label{sec:basics}

Let $\kappa$ be an algebraically closed field of arbitrary characteristic and let $\C$ be a metrized complex of
$\kappa$-curves, as defined in Section~\ref{section:overview}.  

\subsection{Divisors and their rank}

If ${\mathcal G}$ is a subgroup of $\RR$ containing the lengths of all edges of $G$,
we write $\mathrm{Div}(\C)_{\mathcal G}$ for the subgroup of $\mathrm{Div}(\C)$
consisting of all divisors $\mathcal D$ such that the support of the $\Gamma$-part of $\D$ 
consists entirely of 
${\mathcal G}$-{\it rational points} of $\Gamma$ (i.e., points of $\Gamma$ whose distances to the vertices of $G$ are in $\mathcal G$).

\medskip

Recall from Section~\ref{section:overview} that the {\em rank} $r_\C$ of a divisor 
$\mathcal D$ on $\C$ is the largest integer $k$ such that $\mathcal D - \mathcal E$ is linearly equivalent to an effective divisor
for all effective divisors $\mathcal E$ of degree $k$ on $\C$. If ${\mathcal G}$ is any subgroup of $\RR$ containing the lengths of all edges of $G$, one can define an analogous rank function $r_{\C,\mathcal G}$ for divisors in $\Div(\C)_\mathcal G$ by restricting the effective divisors $\mathcal E$  in the above definition to lie in $\Div(\C)_\mathcal G$. 
By Corollary~\ref{cor:f-widthG}, $r_{\C,\mathcal G} (\mathcal D)= r_{\C}(\mathcal D)$ for any divisor $\mathcal D \in \Div(\C)_\mathcal G$, so the restriction in the definition of $r_{\C,\mathcal G}$ does not affect the rank of divisors. 
 
 \subsection{Regularization of nodal curves}\label{sec:mcnodal}

Let $X_0$ be a (reduced) connected projective nodal curve over $\k$ with irreducible components $X_{v_1},\dots,X_{v_k}$. We can associate to $X_0$ a metrized complex $\C X_0$, called the {\it regularization} of $X_0$, as follows. The underlying metric graph $\Gamma_0$ has model the dual graph $G_0=(V_0,E_0)$ of $X_0$. Recall that the vertex set $V_0$ of $G_0$ consists of vertices $v_1,\dots, v_k$ (in bijection with the irreducible components of $X_0$) and the edges of $G_0$ are in bijection with singular points of $X_0$. (Note that $G_0$ might have loop edges.) The length of all the edges in $\Gamma_0$ are equal to one. The $\k$-curve $C_{v_i}$ in $\C X_0$  is the normalization of $X_i$ for $i=1,\dots,k$. The collection $\mathcal A_{v_i}$ is the set of all $\k$-points of $C_{v_i}$ which lie over a singular point of $X_0$ in $X_{v_i}$. By the definition of the dual graph, these points are in bijection with the edges adjacent to $v_i$ in $G_0$.  

\medskip

If $X_0$ is {\it strongly semistable}, so that $G_0$ does not have any loop edge and $C_{v_i} = X_{v_i}$, 
the rank function $r_{\C X_0,\mathbb Z}$ on $\Div(\C X_0)_\mathbb Z$ can be reformulated as follows. 
Let $\Pic(X_0)$ be the Picard group of $X_0$, and consider the restriction map $\pi: \Pic(X_0) \rightarrow \oplus_i \Pic(X_{v_i})$. For any line bundle $\mathcal L$ on $X_0$, and $v\in V_0$, denote by $\mathcal L_{v}$ the restriction of $\mathcal L$ to $X_v$.  

\medskip

Two line bundles $\mathcal L$ and $\mathcal L'$ in $\Pic(X_0)$ are said to be {\it combinatorially equivalent} if there exists a function $f: V_0 \rightarrow \mathbb Z$ such that 
$\mathcal L'_v = \mathcal L_v(\div_{v}(f))$ in $\Pic(X_v)$ for any vertex $v$ of $G_0.$ 
(In a regular smoothing of $X_0$, c.f. Section~\ref{sec:limitseries} for the definition, 
the role of $f$ is to specify a particular twist of 
$\mathcal L$ by a divisor supported on the irreducible components of $X_0$.)
In particular, two line bundles $\mathcal L$ and $\mathcal L'$ in $\Pic(X_0)$ with $\pi(\mathcal L) = \pi(\mathcal L')$ are combinatorially equivalent.
In the above definition, $\div_v$ is defined in analogy with (\ref{eq:divf2}) as follows:
$$\div_v(f) = \sum_{u: \{u, v\} \in E_0} (f(u) -f(v))(x^{\{u,v\}}_v)$$ where $x_v^{\{u,v\}}$ is the point of $C_v$ labelled with 
the edge $\{u,v\}$. 
Note that the function $f: V_0 \rightarrow \mathbb Z$ in the definition of combinatorially equivalent line bundles above is unique up to an additive constant. For such an $f$, we denote by $\mathcal L^f$ any line bundle $\mathcal L'$ in $\Pic(X_0)$ such that $\mathcal L'_v = \mathcal L_v(\div_{v}(f))$ for all $v\in V_0$. 

\medskip

For a line bundle $\mathcal L$ on $X_0$, define the {\it combinatorial rank} of $\mathcal L$, denoted $r_{c}(\mathcal L)$, to be the maximum  integer $r$ such that for any effective divisor $E = \sum_{v \in V_0} E(v) (v)$ on $G_0$ of degree $r$, there is a line bundle $\mathcal L'$ combinatorially equivalent to $\mathcal L$ with
$$\dim_\k H^0(X_v, \mathcal L'_v) \geq E(v)+1$$
for all $v \in V$.
In particular, if no non-negative integer $r$ with the above property exists, the combinatorial rank of $\mathcal L$ will be equal to $-1$.  
The combinatorial rank of $\mathcal L$ clearly depends only on the combinatorial equivalence class of $\mathcal L$. 

\medskip
Let $\mathcal D$ be a divisor in $\Div(\C X_0)_\mathbb Z$ with $X_v$-part $D_v$. We denote by $\mathcal L(\mathcal D)$ any line bundle on $X_0$ whose restriction to $X_v$ is $\mathcal L(D_v)$ for each vertex $v$ of $G_0$. 

\begin{prop} \label{prop:combinrank}
Let $X_0$ be a strongly semistable curve over $\k$ and let $\C X_0$ be the corresponding metrized complex.  Then:
\begin{enumerate}
\item Two divisors $\mathcal D$ and $\mathcal D'$  in $\Div(\C X_0)_\mathbb Z$ are linearly equivalent iff $\mathcal L(\mathcal D)$ and $\mathcal L(\mathcal D')$ are combinatorially equivalent. 
\item For any divisor $\mathcal D \in \Div(\C X_0)_\mathbb Z$, the combinatorial rank of $\mathcal L(\mathcal D)$ is equal to $r_{\C X_0}(\mathcal D).$
\end{enumerate}
\end{prop}
\begin{proof}
The first part follows by definition. To see the second part, let $r$ be the combinatorial rank of $\mathcal L(\mathcal D)$. It follows from Corollary~\ref{cor:f-widthZ}
that $r_{\C X_0}(\mathcal D) = r_{\C X_0,\mathbb Z}(\mathcal D)$. Thus, it will be enough to show that  $r_{\C X_0, \mathbb Z}(\mathcal D)=r$.
  
\medskip
  
We first prove that $r_{C X_0,\mathbb Z}(\mathcal D) \geq r$. 
Let $\mathcal E$ be an effective divisor in $\Div(\C X_0)_\mathbb Z$ with $\Gamma$-part $E_\Gamma$ and $X_v$-part $E_v$. 
By the definition of the combinatorial rank, there exists a function $f:V_0 \rightarrow \mathbb Z$ such that 
$$\dim_\k H^0(X_v, \mathcal L^f_v) \geq E_\Gamma(v)+1 \qquad \textrm{for any vertex $v$ in $G_0$}.$$
Since $\deg(E_v) = E_\Gamma(v)$, this implies the existence of a global section $f_v$ of $\mathcal L^f_v$ 
such that $f_v$ has order of vanishing at least $E_v(x)$ at any point $x$ of $X_v$. 

\medskip

Let $f_\Gamma$ be the rational function defined by $f$ on $\Gamma$. 
(Recall that $f_\Gamma|_{V_0} = f$ and $f$ is integer affine with slope $f(u)-f(v)$ on each edge $\{u,v\}$.) 
For the rational function $\f$ on $\C X_0$ defined by $f_\Gamma$ and $\{ f_v \}_{v\in V_0}$, one has $\div(\f) + 
\mathcal D - \mathcal E \geq 0$.  This shows that 
$r_{\C X_0,\mathbb Z}(\mathcal D) \geq r$.

\medskip

We now prove that $r \geq r_{\C X_0,\mathbb Z}(\mathcal D)$. 
Let $\mathcal L = \mathcal L(\mathcal D)$ be a line bundle on $X_0$ defined by $\mathcal D$. Let $E$ be an
effective divisor of degree $r_{\C X_0,\mathbb Z}(\mathcal D)$ on $G_0$, and let 
$\mathcal E$ be any effective divisor in  $\Div(\C X_0)_\mathbb Z$ with $\Gamma$-part equal to $E$. 
Then there exists a rational function $\f$ with $\div(\f) \in \Div(\C X_0)_\mathbb Z$ such that 
$\div(\f) +\mathcal D -\mathcal E \geq 0$. In particular, the $\Gamma$-part $f_\Gamma$ of 
$\f$ is linear on each edge of $G_0$ and $\div(f_\Gamma) + D_\Gamma \geq 0$, 
which shows that up to an additive constant, the set $F$ of all functions $f_\Gamma$ with these properties is finite.
Let $S_\Gamma$ denote the finite set $\{ \div(f_{\Gamma}) \; | \; f_{\Gamma} \in F \}$.
Let $d_v = E(v)$, and for $D_0 \in S_{\Gamma}$, let $S_{D_0}$ be the subset of
$\prod_{v \in V(G_0)} X_v^{(d_v)}$ 
defined by all collections of effective divisors $E_v$ of degree $d_v$
such that there exists a rational function $\f$ on $\C X_0$ with $\Gamma$-component $f_\Gamma$ satisfying 
$\mathcal D -\mathcal E +\div(\f)\geq 0$. 
(Here $X^{(n)}$ denotes the $n^{\rm th}$ symmetric product of a curve $X$.)
Now let $S$ be the union over all $D_0 \in S_{\Gamma}$ of $S_{D_0}$.
Then one sees easily that:
\begin{itemize} 
\item[(i)] Each $S_{D_0}$ is Zariski-closed in $\prod_v X_v^{(d_v)}$.  (Use the fact that $S_{D_0}$ is the projection of a Zariski closed subset of a product of proper curves.)
\item[(ii)] $\bigcup_{D_0 \in S_{\Gamma}} S_{D_0} = \prod_v X_v^{(d_v)}$.
\end{itemize}

Therefore there exists $D_0 \in S_{\Gamma}$ such that $S_{D_0} =\prod_v X_v^{(d_v)}$. 
In other words, we can suppose that the $\Gamma$-part $f_\Gamma$ 
of all the rational functions $\f$ giving $\mathcal D -\mathcal E +\div(f)\geq0 $, when the effective divisors $E_v$ vary, is constant (and we have $D_0=\div(f_\Gamma)$). 
If we fix a corresponding twist $\mathcal L^{f_\Gamma}$, then
\[
\dim_\k H^0(X_v, \mathcal L^{f_\Gamma}_v) \geq E(v)+1.
\]

Since $E$ was an arbitrary effective divisor of degree 
$r_{\C X_0,\mathbb Z}(\mathcal D)$ on $G_0$, we obtain the desired inequality.

\end{proof}

\section{A Riemann-Roch theorem for metrized complexes of curves}\label{sec:R-R}

Our aim in this section is to formulate and prove a Riemann-Roch theorem for metrized complexes of $\kappa$-curves. 

\medskip

Let $\C$ be a metrized complex of algebraic curves over $\k$.  
For each vertex $v \in V$, let $g_v$ be the genus of $C_v$ and let $g(\C) = g(\Gamma) + \sum_v g_v$ 
be the genus of $\C$. For each curve $C_v$, let $K_v$ denote a divisor of degree $2g_v-2$ in the canonical class of 
$C_v$. Denote by $A_v$ the divisor in $C_v$ consisting of the sum of the $\deg_G(v)$ points in $\mathcal A_v$. 
The {\em canonical class of $\C$} is defined to be the linear equivalence class of the divisor 
$\mathcal K = \sum_{v\in V} (K_v + A_v).$

\begin{remark}\label{rem:canonicalsemi-stable} Let $X_0$ be a strongly semistable curve. The definition of the canonical class $\mathcal K$ in $\C X_0$ is compatible 
in a natural sense with the definition of the dualizing sheaf 
(or sheaf of logarithmic $1$-forms) $\omega_{X_0/\kappa}$. Indeed, if $\alpha: X \rightarrow X_0$ denotes the normalization of $X_0$, and $y^e$,$z^e$ denote the points of $X$ above the singular point $x^e$ of $X_0$ for $e\in E(G_0)$, then sections of the sheaf $\omega_{X_0/\kappa}$ on an open set $U$ consist of all the meromorphic 1-forms $\delta$ on $X$ which are regular everywhere on $U$ except for possible simple poles at $y^e,z^e$, with $\mathrm{Res}_{y^e}(\delta)+\mathrm{Res}_{z^e}(\delta) =0$ for all $e$. From this, it is easy to see that the restriction of $\omega_{X_0/\kappa}$ to each component $X_v$ of $X_0$ is the invertible sheaf $\omega_{X_v}(A_v)$ corresponding to $K_v+A_v$ for any canonical divisor $K_v$ on $X_v$.  
\end{remark}

\medskip

For a divisor $\mathcal D$ on $\C$, let $r_\C(\mathcal D)$ denote the rank of $\mathcal D$ as defined 
 in Section~\ref{section:overview}.
 
 \begin{thm}[Riemann-Roch for metrized complexes of algebraic curves]\label{thm:RR-metrizedcomplexes}
 Let $\C$ be a metrized complex of $\kappa$-curves and $\mathcal K$ a divisor in the canonical class of $\C$. For every divisor $\mathcal D \in \Div(\C)$, we have 
 \[r_\C(\mathcal D) - r_\C(\mathcal K -\mathcal D) = \deg(\mathcal D) - g(\C)+1. \]
 \end{thm}

Let $X_0$ be a strongly semistable curve of genus $g$ over $\kappa$. For any $\mathcal L \in \Pic(X_0)$, let $r_{c}(\mathcal L)$ be the combinatorial rank of $\mathcal L$ as defined in Section~\ref{sec:basics}. As a corollary of Theorem~\ref{thm:RR-metrizedcomplexes}, Remark~\ref{rem:canonicalsemi-stable}, and Proposition~\ref{prop:combinrank}, we have:
\begin{cor}[Riemann-Roch for strongly semistable curves]
For any $\mathcal L\in \Pic(X_0)$, 
$$r_{c}(\mathcal L) - r_{c}(\omega_{X_0/\kappa}\otimes \mathcal L^{-1}) = \deg(\mathcal L) - g(X_0)+1.$$
\end{cor}

The Riemann-Roch theorem has a number of well-known formal consequences which can be transported to our situation.  For example, we obtain the following analogue of Clifford's bound for the rank of a special divisor.  (A divisor $\mathcal D$ on a metrized complex $\C$ is called {\em special} if 
$K_{\C} - {\mathcal D}$ is linearly equivalent to an effective divisor.)
\begin{thm}[Clifford's theorem for metrized complexes] 
\label{thm:MCClifford}
For any special divisor $\mathcal D$ on a metrized complex $\C$, $r_\C(\mathcal D) \leq \deg(\mathcal D)/2$.
\end{thm}

\begin{remark} Our Riemann-Roch theorem seems to be rather different from the classical Riemann-Roch theorem on semistable curves. We refer to~\cite{C-LS} for a discussion of Riemann's and Clifford's theorems for linear series on semistable curves, and for examples showing the failure of Clifford's theorem for $h^0(\mathcal L)$ with $\mathcal L \in \Pic(X_0)$ (even for line bundles in the compactified Picard scheme of $X_0$). 
\end{remark}

The rest of this section is devoted to the proof of Theorem~\ref{thm:RR-metrizedcomplexes}. 
Our proof follows and extends  the original arguments of~\cite{BN} in the proof of Riemann-Roch theorem for graphs. 
In particular, we are going to first extend the notion of reduced divisors to the context of metrized complexes of 
algebraic curves, and then study the minimal non-special divisors.

\subsection{Reduced divisors: existence and uniqueness}\label{sec:reduced}
Let $v_0$ be a fixed base point of $\Gamma$. We introduce the notion of $v_0$-reduced divisors and show that each equivalence class of divisors on $\C$ contains a quasi-unique $v_0$-reduced divisor, in a precise sense to be defined immediately preceding Theorem~\ref{thm:reduced.divisors} below.

\medskip

For a closed connected subset $S$ of $\Gamma$ and a point $v \in \partial S$ (the topological boundary of $S$), the number of ``outgoing'' tangent directions at $v$ (i.e., tangent directions emanating outward from $S$) is denoted by $\mathrm{outdeg}_S(v)$; this is also the maximum size of a
collection of internally disjoint segments in $\Gamma \setminus (S -\{v\})$ with one end at $v$.  If in addition we have $v \in V$, we denote by $\div_v(\partial S)$ the divisor in $C_v$ associated to the outgoing edges at $v$; this is by definition the sum of all the points $x^e_v$ of $C_v$ indexed by the edges $e$ leaving $S$ at $v$.  
In what follows we refer to a closed connected subset  $S$ of $\Gamma$ as a {\it cut} in $\Gamma$.
 
 \medskip
 
Let $\mathcal D$ be a divisor on $\C$ with $\Gamma$-part $D_\Gamma$ and $C_v$-part $D_v$, and let $S$ be a cut in $\Gamma$.
A boundary point $x \in \partial S$ is called
{\it saturated} with respect to $\mathcal D$ and $S$  if
\begin{itemize}
\item $x \notin V$ and $\mathrm{outdeg}_S(x) \leq D_\Gamma(x)$; or
\item $x=v$ for some $v\in V$ and $D_v- \div_v(\partial S)$ is equivalent to an effective divisor on $C_v$. 
\end{itemize}

\noindent Otherwise, $x\in \partial S$ is called {\it non-saturated}. A cut $S$ is said to be saturated if all its boundary points are saturated. (When talking about saturated and non-saturated points, we will sometimes omit the divisor $\D$ or the set $S$ if they are clear from the context.)

\medskip

The divisor $\mathcal D$ is said to be {\it $v_0$-reduced} if the following three properties are satisfied: 

\begin{itemize}
\item[(i)] For all points $x \neq v_0$ of $\Gamma$, $D_\Gamma(x) \geq 0$, i.e., all the coefficients of $D_\Gamma$ are non-negative except possibly at the base point $v_0$. 
\item[(ii)] For all points $v \in V \setminus \{v_0\}$ ($=V$ if $v_0$ does not belong to $V$), 
there exists an effective divisor $E_v$ linearly equivalent to $D_v$ on $C_v$.
\item[(iii)] For every cut $S$ of  $\Gamma$ which does not contain $v_0$, there exists a non-saturated point $x \in \partial S$.
\end{itemize}

\medskip

\begin{remark}
\label{rmk:burningalg}
There is an efficient ``burning algorithm'' for checking whether or not a given divisor $\D$ on $\C$ which satisfies (i) is $v_0$-reduced.
The algorithm can be described informally as follows (compare with \cite[\S{2}]{Luo} or \cite[\S{5.1}]{BS}).
Imagine that $\Gamma$ is made of a flammable material, and light a fire at $v_0$.  The fire begins spreading across $\Gamma$ in a continuous manner and can only be blocked at a point $x \in \Gamma$ if one of the following holds:
\begin{itemize}
\item $x \notin V$ and the number of burnt directions coming into $x$ exceeds $D_\Gamma(x)$. (One should imagine $D_\Gamma(x)$ firefighters standing at $x$, each of
whom can block fire in one incoming direction.)
\item $x=v$ for some $v\in V$ and $D_v$ minus the sum of the marked points corresponding to burnt directions is not equivalent to any effective divisor on $C_v$. 
\end{itemize}
It is straightforward to check, following the same ideas as the proof of \cite[Algorithm 2.5]{Luo}, that $\D$ is $v_0$-reduced if and only if the fire eventually burns through all of $\Gamma$.  We omit the details since we will not need this result in the sequel.
\end{remark}

\medskip

We now show that every divisor $\mathcal D$ on $\C$ is linearly equivalent to a quasi-unique $v_0$-reduced divisor $\mathcal D^{v_0}$.
The quasi-uniqueness is understood in the following sense: the $\Gamma$-part $D_\Gamma^{v_0}$ of $\mathcal D^{v_0}$ is unique, 
and for all $v \in V$, the divisor class $[D_v^{v_0}]$ on $C_v$ defined by the $v$-part of $\mathcal D^{v_0}$ is unique. 
 
 \begin{thm} \label{thm:reduced.divisors}
Let $\C$ be a metrized complex of $\k$-curves and $v_0$ a base point of $\Gamma$. For every divisor $\mathcal D \in \Div(\C)$, there exists a quasi-unique $v_0$-reduced divisor $\mathcal D^{v_0}$ such that $\mathcal D^{v_0} \sim \mathcal D$.
\end{thm}

\begin{proof} 
Let $\mathcal T_{\mathcal D}$ be the set of all divisors $\mathcal D'$ linearly equivalent to $\mathcal D$ such that:

\begin{itemize}
\item All the coefficients of the $\Gamma$-part $D'_\Gamma$ of $\mathcal D'$ are non-negative at every point of $\Gamma$ except possibly at $v_0$.
\item For each point $v \in V \setminus \{v_0\}$, the $v$-part $D'_v$ of $\mathcal D'$ has non-negative rank on $C_v$.
\item The coefficient of $v_0$ in $D'_\Gamma$ is maximal with respect to the two properties above.
\end{itemize}

One shows as in~\cite[Proof of Theorem 2]{amini} or \cite[Proof of Proposition 3.1]{BN}
that there is a divisor $\mathcal D'$ linearly equivalent to $\mathcal D$ on $\C$  with the property that the coefficient of every point $x \in \Gamma \setminus (V \cup \{v_0\})$ is non-negative and the coefficient of $D'_\Gamma$ at each point $v \in V$ (which coincides with the degree of $D'_v$) is at least $g_v$. 
(The intuitive idea is to repeatedly fire $v_0$ and then use further chip-firing moves to spread chips wherever needed on $\Gamma$, assuring that each $v \in V$ gets at
least $g_v$ chips and that no $p \in \Gamma$ except for $v_0$ remains in debt.)
By the (classical) Riemann-Roch theorem, each $D'_v$ has non-negative rank on $C_v$, which shows that $\mathcal T_{\mathcal D}$ is non-empty. 

\medskip

Define $T_{\mathcal D}$ as the set of all $D'_\Gamma \in \Div(\Gamma)$ for which there exists
${\mathcal D}' \in \mathcal T_{\mathcal D}$ whose $\Gamma$-part is $D'_\Gamma$.
Each divisor $D'_\Gamma \in T_{\mathcal D}$ has degree $\deg(\mathcal D)$, and
all divisors in $T_{\mathcal D}$ have the same coefficient at $v_0$.
Thus the number
$N := \sum_{x \in \Gamma \setminus \{v_0\}} D'_\Gamma(x)$
is independent of the choice of a divisor $D'_\Gamma$ in $T_{\mathcal D}$.
As a consequence, $T_{\mathcal D}$ inherits a natural topology from the topology of $\Gamma$, since it can be embedded as a subset of $\mathrm{Sym}^N\Gamma$, and is compact.
 Let $\mathcal A$ be the subset of $\mathbb R^{N}$ defined by
$$\mathcal A := \Bigl\{(x_1,\dots,x_N)\,|\,0\leq x_1\leq x_2\leq \dots\leq x_N\Bigr\},$$
equipped with the total order defined by the lexicographical rule $(x_i)\leq (y_i)$ iff 
$x_1=y_1,\dots,x_i =y_i$ and $x_{i+1} < y_{i+1}$. One considers a map $F:  T_{\mathcal D} \rightarrow \mathcal A$ defined as follows.  For each divisor $D'_\Gamma \in T_{ \mathcal  D}$, consider the multiset $A(D'_\Gamma)$ of points in $\Gamma \setminus \{v_0\}$ where each point $v\neq v_0$ appears in this multiset exactly $D'_\Gamma(v)$ times. Define $F(D'_\Gamma)$ to be the point of $\mathcal A$ defined by the multiset of distances $\mathrm{dist}_\Gamma(v,v_0)$ for $v\in A(D'_\Gamma)$, ordered in an increasing way.
It is straightforward to verify that the map $F$ is continuous.
Since $T_{\mathcal D}$ is compact and $F$ is continuous, there exists a divisor $D^{v_0}_\Gamma$ in $T_{\mathcal D}$ such that $F$ takes its minimum value at $D^{v_0}_\Gamma$, i.e., $F(D_\Gamma^{v_0}) = \min_{D'_\Gamma \in T_{\mathcal D}} F(D'_\Gamma)$. Let $\mathcal D^{v_0}$ be an element of $\mathcal T_{\mathcal D}$ 
whose $\Gamma$-part is $D^{v_0}_\Gamma$.

\medskip

{\bf Claim:} The divisor $\mathcal D^{v_0}\sim \mathcal D$ is $v_0$-reduced.

\medskip

Properties (i) and (ii) above are clearly satisfied. The only point one needs to check is that every cut $S$ which  does not contain $v_0$ has a non-saturated point on its boundary. For the sake of contradiction, suppose this is not the case, and let $S$ be a closed connected set violating this condition. This means that the following hold:
 \begin{itemize}
 \item For all $x \in \partial S \setminus V$, $\mathrm{outdeg}_S(x) \leq D^{v_0}_\Gamma(v)$.
 \item For each $v \in \partial S\cap V$, $D^{v_0}_v - \div_v(\partial S)$ has non-negative rank on $C_v$, i.e., there exists a rational function $f_v$ on $C_v$ such that $D^{v_0}_v - \div_v(\partial S) + \div(f_v) \geq 0$.
 \end{itemize}
 By the definition of $\mathrm{outdeg}_S$, there exists an $\epsilon>0$ such that for each vertex $x \in \partial S$, there are closed segments $I_1^x,\dots,I^x_{\mathrm{outdeg}_S(x)}$ emanating from $x$ with the following properties:
\begin{itemize}
\item For $x \in \partial S$ and $1\leq  j\leq \mathrm{outdeg}_S(x)$, the half-open segments $I^x_j \setminus x$ are 
disjoint from $S$ and from each other and do not contain $v_0$.
\item Each segment $I_j^x$ has length $\epsilon$, for $x \in \partial S$ and $1\leq  j\leq \mathrm{outdeg}_S(x)$.
\end{itemize}
These data give rise to a rational function $f_\Gamma: \Gamma \rightarrow \mathbb R$ which is identically zero on $S$, is linear of slope $-1$ on each interval $I_j^x$ for $x\in \partial S$ and $1\leq j \leq \mathrm{outdeg}_S(x)$, and takes the constant value $-\epsilon$ at all points of $\Gamma \setminus \Bigl(S\cup \bigcup_{x,j } I_j^x\Bigr)$. 
Consider the divisor $\mathcal D^*= \mathcal D^{v_0} + \div(\f)$ on $\C$, where $\f$ is the rational function on $\C$ consisting of the rational function $f_\Gamma$ on $\Gamma$ and the rational functions $f_v$ on $C_v$. 
Letting $D_\Gamma^*$ be the $\Gamma$-part of $\D^*$,
one verifies that $\mathcal D^*$ and $D_\Gamma^*$ lie in  $\mathcal T_{\mathcal D}$ and $T_{\mathcal D}$, respectively, and that $F(D^*_\Gamma) < F(D^{v_0}_\Gamma)$, contradicting the choice of $D^{v_0}_\Gamma$. 
This proves the claim and hence the existence part of the theorem. 

\medskip 

It remains to prove the quasi-uniqueness.  Assume for the sake of contradiction that there are linearly equivalent $v_0$-reduced divisors $\mathcal D$
and $\mathcal D'$ on $\C$ such that either $D_\Gamma \neq D'_\Gamma$ or $D_v$ and $D'_v$ are not linearly equivalent on $C_v$
for some $v \in V$.  Then there exists a non-constant rational function $\f$ on $\C$ such that $\mathcal D' = \mathcal D+\div(\f)$. 
If $f_\Gamma$ is constant then we obtain an immediate contradiction, so we may assume that $f_\Gamma$ is non-constant.
Without loss of generality, we may assume that $f_\Gamma$ does not take its maximum at $v_0$ (otherwise, we can interchange the role of $\mathcal D$ and $\mathcal D'$). Let $S$ be a connected component of the set of all points where $f_\Gamma$ takes its maximum. Note that $v_0 \notin S$. For all points $x\in \partial S$, the slope of $f_\Gamma$ at any outgoing segment $I_j^x$ emanating from $v$ is at most $-1$. 
Since $\mathcal D$ is $v_0$-reduced, there exists a point $x\in \partial S$ such that either $x \notin V$ and $D_\Gamma(x) < \mathrm{outdeg}_S(x)$, or $x =v$ for some $v\in V$ and the divisor $D_v  - \div(\partial S)$ has negative rank. In the first case, $D'_\Gamma(x) \leq  D_\Gamma(x) - \mathrm{outdeg}_S(x) <0$, contradicting the assumption that the coefficient of $D'_\Gamma$ is non-negative at $x \neq v_0$. In the second case, $D'_v \sim D_v + \sum_{e\in E:\,e\sim v} \mathrm{slp}_e(f_\Gamma) (x^e_v) \leq D_v - \div_v(\partial S)$, which implies that $D'_v$ has negative rank, a contradiction. 
\end{proof}

\begin{remark}
One can presumably give an ``iterated Dhar algorithm'' for algorithmically finding a $v_0$-reduced divisor equivalent to a given divisor $\D$ by combining the ideas
in the proof of Theorem~\ref{thm:reduced.divisors} with the ideas behind \cite[Algorithm 2.12]{Luo}.  
\end{remark}

\subsection{Description of minimal non-special divisors} 
Let $\C$ be a metrized complex of algebraic curves. For each $v \in V$, 
let $\mathfrak N_v$ be the set of all {\em minimal non-special} divisors on $C_v$:
\[\mathfrak N_v = \{D \in \Div(C_v)\,\, :\,\, \deg(D) =g(C_v)-1 \,\,\textrm{and} \,\, |D| =\emptyset\}.\]

\begin{remark}
 Recall from \cite[IV, Example 1.3.4]{HartshorneAG} that a divisor $D$ on a smooth projective curve $C$ of 
genus $g$ over $\k$ is called {\em special} if $r(K_C-D) \geq 0$. 
By Riemann-Roch, every divisor of degree at most $g-2$ on $C$ is special, and if
$\deg(D)= g-1$ then $D$ is special if and only if $r(D)\geq 0$.  
Moreover, every non-special divisor dominates a minimal non-special divisor.
This explains the term ``minimal non-special divisor''.
\end{remark}

Similarly, define the set of minimal non-special divisors on $\C$ as 
\[\mathfrak N = \{\mathcal D \in \Div(\C)\,\, :\,\, \deg(\mathcal D) =g(\C)-1 \,\,\textrm{and} \,\, |\mathcal D| =\emptyset\}.\]
In this section we provide an explicit description of the minimal non-special divisors 
on metrized complexes of algebraic curves, generalizing the corresponding description of minimal non-special divisors 
on graphs and metric graphs from~\cite{BN, MZ}.

\medskip

An {\em acyclic orientation} on $\Gamma$ is an acyclic orientation on some model of $\Gamma$, i.e., a decomposition of $\Gamma$ into closed directed edges with disjoint interiors such that no directed cycles are formed.
Given an acyclic orientation $\pi$ on $\Gamma$, we denote by ${\rm deg}^+_\pi(x)$ the number of tangent directions emanating from $x$ which are compatible with $\pi$. Note that for all but finitely many points of $\Gamma$, $\deg^+_\pi(x)=1$.
For $v \in V$, we denote by $E^+(v)$ the set of edges incident to $v$ which (locally near $v$) are oriented outward from $v$.
A point $x$ with $\deg^+_\pi(x)=0$ is called a {\em sink}.  It is well known and easy to prove that every acyclic orientation contains at least one sink (start at any
point and follow the orientation until you get stuck, which must eventually happen by acyclicity).

\medskip

Given a collection of minimal non-special divisors $D_v \in  \mathfrak N_v$ for each $v\in V$, together with an acyclic orientation $\pi$ of $\Gamma$, define a corresponding divisor $\mathcal D^{\pi, \{D_v\}}$ by the formula
\[
D^{\pi, \{D_v\}} = \sum_{x \in \Gamma \backslash V} (\deg^+_\pi(x) - 1) (x)  + \sum_{v \in V} \left( A^\pi_v + D_v \right),
\]
where $A^\pi_v$ is the sum of all points $x^e_v \in \mathcal A_v$ for which the $\pi$-orientation on $e$ points away from $v$.

The $\Gamma$-part of $D^{\pi, \{D_v\}}$ is 
$$D_\Gamma^\pi := \sum_{x\in \Gamma} (\deg^+_\pi(x)+g_{x}-1) (x)$$
and the $C_v$-part of $D^{\pi, \{D_v\}}$ is $D^\pi_v := A^\pi_v + D_v$.
If $\widetilde G = (\widetilde V, \widetilde E)$ is a (loopless) model of $\Gamma$ such that the orientation $\pi$ of $\Gamma$ is induced by an acyclic orientation $\widetilde G^\pi$ on the edges of $\widetilde G$, then the degree of $\mathcal D^{\pi ,\{D_v\}}$ is given by 
\begin{align*}
\deg(\mathcal D^{\pi ,\{D_v\}}) &= \sum_{v\in \widetilde{V}} \bigl( \deg^+_\pi(v)+g_{v}-1 \bigr) \\
&= \Bigl[ \sum_{v\in \widetilde{V}} \bigl(\deg^+_\pi(v) -1\bigr)\Bigr] +  \sum_{v\in V} g_v  = g(\Gamma)-1+\sum_v g_v = g(\C) -1.
\end{align*} 

A divisor $\mr \in \Div(\C)$ of the form $\D^{\pi,\{ D_v \}}$ is called a {\it moderator} on $\C$.  (This terminology comes from \cite{MZ}.)
Given a moderator $\mr = \D^{\pi,\{D_v\}}$, the {\em dual moderator} $\bar{\mr}$ is defined to be $\D^{\bar{\pi}, \{ K_v - D_v \}}$, where $\bar{\pi}$ is obtained from $\pi$ by reversing the orientation of every oriented segment.  
It is easy to see that $\mr + \bar{\mr}$ belongs to the canonical class on $\C$.
 
\medskip 
 
The following two lemmas are essentially what we need for the proof of Theorem~\ref{thm:RR-metrizedcomplexes}.
 
\begin{lemma} \label{lem:non-special1} For any acyclic orientation $\pi$ of $\Gamma$ and any collection $D_v \in \mathfrak N_v$ of minimal non-special divisors on $C_v$, the  moderator $\mathcal D^{\pi,\{D_v\}}$ is a minimal non-special divisor on $\C$. 
\end{lemma}

\begin{proof} 
Suppose for the sake of contradiction that there is a rational function 
$\f$ on $\C$ such that $\mathcal D^{\pi,\{D_v\}} +\div(\f) \geq 0$. 
Let $\widetilde G = (\widetilde V, \widetilde E)$ be a (loopless) model of $\Gamma$ such that the orientation $\pi$ of $\Gamma$ is induced by an acyclic orientation $\widetilde G^\pi$ on the edges of $\widetilde G$.  
Let $S$ be the closed subset of $\Gamma$ consisting of all the points of $\Gamma$ where the $\Gamma$-part $f_\Gamma$ of $\f$ takes its maximum value. It is easy to see that $\partial S \subseteq \widetilde V$, since at a hypothetical point $x \in \partial S \setminus \widetilde V$ we would necessarily have 
$D^\pi_\Gamma(x)+\div_x(f_\Gamma) <0$, contradicting the assumption on $\f$. 

Note that since $\partial S \subset \widetilde V$, for any edge $e$ of $\widetilde G$ either $S$ contains $e$ entirely or $S$ does not intersect the interior of $e$. Consider now the restriction of the orientation $\pi$ to the induced subgraph $\widetilde G[\widetilde V \cap S]$. 
The resulting directed graph $\widetilde G^\pi[\widetilde V \cap S]$ is acyclic and therefore contains a sink vertex $u$. 

\medskip

{\bf Claim:} $u\in \partial S$. 

\medskip

It suffices to prove that no sink vertex of the directed graph $\widetilde G^\pi$ belongs to $S$, so assume for the sake of contradiction that 
$w \in S$ is a sink vertex of $\widetilde G^\pi$. There are two cases to consider:
\begin{itemize}
\item  If $w \in V$, then $D_{w} +  \div(f_{w}) \geq D^\pi_{w} +  \div_{w}(f_\Gamma) + \div(f_{w}) \geq 0$ (since $f_\Gamma$ achieves its maximum value at $w$ and we supposed that $\mathcal D^{\pi,\{D_v\}} +\div(\f) \geq 0$),  contradicting the assumption that $D_{w}$ is a minimal non-special divisor on $C_{w}$.  
\item If $w \notin V$,  then by definition the coefficient of $w$ in $D^\pi_\Gamma$ is at most  $-1$, from which it follows that the coefficient of $w$ in $\mathcal D^{\pi,\{D_v\}} +\div(\f)$ is negative, a contradiction.
\end{itemize}

This proves the claim.

\medskip

By our choice of $u \in \partial S$ and the definition of $S$, $f_\Gamma$ has strictly negative slope along all outgoing 
edges at $u$,
and it has slope zero along all other edges incident to $u$.
This shows in particular that $\mathrm{ord}_{u}(f_\Gamma) \leq  - \mathrm{outdeg}_S(u) \leq -\deg^+_\pi(u)$. 
There are now two different cases to consider, both of which will lead to a contradiction (and hence complete our proof of the lemma):

\begin{itemize}
\item If $u \in \widetilde V\setminus V$, then $D^\pi_\Gamma(u) + \mathrm{ord}_u(f_\Gamma)\leq  \deg^+_\pi(u)-1-  \deg^+_\pi(u)<0$, contradicting the choice of the rational function $\f$ on $\C$.
\item If $u \in V$, then for each edge $e$ incident to $u$ we have $\mathrm{slp}_e(f_\Gamma) \leq 0$, with strict inequality if 
$e$ is an outgoing edge at $u$ in $\widetilde G^\pi$ (because $u$ is a sink of the oriented graph $\widetilde G^\pi[S\cap \tilde V]$).
By the definition of $D^\pi_u$, we infer that $D^\pi_u + \sum_{e\in \widetilde E: e\sim u} \mathrm{slp}_e(f_\Gamma)(x^e_u) \leq D_u$, which shows that  $D^\pi_u + \sum_{e\in \widetilde E: e\sim u} \mathrm{slp}_e(f_\Gamma)(x^e_u) + \div(f_u)$ 
is not effective, since $D_u$ is minimal non-special. This contradicts our assumptions on $\f$.
 \end{itemize}
\end{proof}

For $v_0 \in \Gamma$, we denote by $\mathcal{AO}_{v_0}(\Gamma)$ the set of all acyclic orientations of $\Gamma$ with a unique sink at $v_0$, i.e., such that $v_0$ has out-degree zero and all other points of $\Gamma$ have out-degree at least one. 

\begin{lemma}\label{lem:reduced}
Let $\mathcal D$ be a $v_0$-reduced divisor on $\C$.  Then $r_\C(\mathcal D) \geq 0$ 
if and only if the following two conditions hold:
\begin{itemize}
\item The coefficient $D_\Gamma(v_0)$ of $D_\Gamma$ at $v_0$ is non-negative.
\item If $v_0 \in V$, the divisor $D_{v_0}$ on $C_{v_0}$ has non-negative rank. 
\end{itemize}
 More precisely, if the above conditions do not both hold, then there exists a moderator $\mr = \D^{\pi,\{D^*_v\}}$ with $\pi\in \mathcal{AO}_{v_0}(\Gamma)$ such that $\mathcal D \leq \mr$. 
\end{lemma}

\begin{proof}

The intuitive idea is that one constructs the moderator $\mr = \D^{\pi,\{D^*_v\}}$ using the burning algorithm
described in Remark~\ref{rmk:burningalg}.  Beginning with a $v_0$-reduced divisor $\mathcal D$, 
burn through $\Gamma$ following the algorithm and keep track
of the direction in which the fire spreads; reversing all the arrows defines an acyclic orientation $\pi\in \mathcal{AO}_{v_0}(\Gamma)$.  
If $\mr$ is the corresponding 
moderator, one checks that $\mathcal D \leq \mr$ outside $v_0$, and if $\mathcal D$ is not effective then $\mathcal D \leq \mr$ everywhere.
A more rigorous version of this argument is as follows.

\medskip

First suppose that $D_\Gamma(v_0) \geq0$, and in the case $v_0 \in V$ that there exists a rational function $f_{v_0}$
on $C_{v_0}$ such that $D_{v_0}+\div(f_{v_0}) \geq 0$. We infer that $D_\Gamma\geq 0$, and since each 
$D_v$ for $v\in V \setminus \{v_0\}$ is of non-negative rank, there exists a rational function $f_v$ on $C_v$ 
such that $D_v + \div(f_v) \geq 0$. Let $\f$ be a rational function on $\C$ consisting of a constant rational 
function on $\Gamma$ and $f_v$ on $C_v$ for  $v \in V$. We obviously have $\mathcal D+\div(\f) \geq 0$, 
which implies that $r_\C(\mathcal D)\geq 0$.

\medskip

For the other direction, assume that 
$D_\Gamma(v_0) <0$ if $v_0 \notin V$, and that $D_{v_0}$ has negative rank if $v_0 \in V$.
To show that $r_\C(\mathcal D)=-1$, by Lemma~\ref{lem:non-special1} it will be enough to show the existence of an acyclic orientation $\pi\in \mathcal{AO}_{v_0}(\Gamma)$ and a set of minimal non-special divisors $D^*_v \in \mathfrak N_v$
such that $\mathcal D \leq \mathcal D^{\pi, \{D^*_v\}}$. 

\medskip 

 Let $\widetilde V$ be the union of $V$ and all the points in the support of $D_\Gamma$. Note that in both cases above, $v_0 \in \widetilde V$. 
Let $\widetilde G=(\widetilde V,\widetilde E)$ be the corresponding model of $\Gamma$. We are going to recursively define an orientation $\pi $ of $\widetilde G$ and the collection $\{D^*_v\}$ by handling at each step the orientation of all the edges incident to a vertex $v \in \widetilde V$ and the minimal non-special divisor $D^*_v$ in the case $v \in V$.  We start by considering the vertex $v_0 \in \widetilde V$. The orientations of all the edges incident to $v_0$ are defined so that $\deg^+_\pi(v_0) = 0$; in other words, all these oriented edges are incoming at $v_0$. Note that in the case $v_0 \notin V$, we have $D_\Gamma(v_0) \leq -1 = D_\Gamma^\pi(v_0)$.
In the case $v_0 \in V$, since $r(D_{v_0})<0$ it follows from the Riemann-Roch theorem for $C_{v_0}$ that there exists a minimal non-special divisor $D^*_{v_0}$ such that $D_{v_0} \leq D^*_{v_0}$.

\medskip

Suppose that the orientation of all edges adjacent to vertices $v_0, \dots, v_i \in \widetilde V$ has been defined, and that
for all $v_j \in V$ with $j\leq i$, a minimal non-special divisor $D^*_{v}$ on $C_{v}$ has been given.    
Let $S_i$ be a connected component of the induced subgraph $\widetilde G[\widetilde V \setminus \{v_0,\dots, v_i\}]$.
Since $\mathcal D$ is $v_0$-reduced and $S_i$ is a cut not containing $v_0$,  there exists a point $v_{i+1}$ on the boundary of $S_i$ which is non-saturated. (Note that $v_{i+1}$ also lies in $\widetilde V$.) This means that either:
\begin{itemize}
\item[(1)] $v_{i+1} \in \Gamma \setminus V$ and $\mathrm{outdeg}_{S_{i}}(v_{i+1}) > D_\Gamma(v_{i+1})$; or
\item[(2)] $v_{i+1} \in V$ and $D_v - \div_v(\partial S_{i})$ has negative rank. 
\end{itemize}

\medskip

All outgoing edges from $S_i$ adjacent to $v_{i+1}$ have already been oriented (and are outgoing from $v_{i+1}$ by the definition of the orientation $\pi$). Orient all other edges incident to $v_{i+1}$ in such a way that they are all incoming at $v_{i+1}$. Note  that 
$\deg^+_\pi(v_{i+1}) =\mathrm{outdeg}_{S_{i}}(v_{i+1}),$ and so in Case (1), 
$$D_{\Gamma}(v_{i+1}) \leq \deg^+_\pi(v_{i+1})-1.$$
 In Case (2),  since $r(D_v -\div_v(\partial S_{i} )) =-1$, it follows from the Riemann-Roch theorem for $C_v$ that there exists a minimal non-special divisor $D^*_v$ such that 
 $$ D_v - \div_v(\partial S_{i}) \leq D^*_v .$$
  This shows that $D_v \leq \div_v(\partial S_{i}) + D^*_v$. By the definition of the orientation $\pi$ we must have $D^{*\,\pi}_{v+1} = \div_{v_{i+1}}(\partial S_{i}) + D^*_v$. Therefore, by the definition of the orientation $\pi$ at $v_{i+1}$ and the choice of $D^*_{v_{i+1}}$, we have $D_{v_{i+1}} \leq D^{*\,\pi}_{v_{i+1}}$. 

\medskip

Let $\pi$ be the orientation of $\Gamma$ just constructed, and let $\{D^*_v\}_{v\in V}$ be the collection of minimal non-special divisors on $C_v$ defined above. By the definition of $\pi$, the only vertex of $\widetilde V$ with out-degree zero is $v_0$ and thus $\pi$ belongs to $\mathcal{AO}_{v_0}(\Gamma)$. We clearly have $\mathcal D \leq \mathcal D^{\pi,\{D^*_v\}}$ as well, which completes the proof.
\end{proof}

As a corollary of the above lemma, we obtain:
 \begin{cor}\label{cor:non-special2}
Any minimal non-special divisor $\mathcal D$ on $\C$ is linearly equivalent to a moderator.  
\end{cor}

\begin{proof} Fix $v_0\in \Gamma$. 
Let $\mathcal D$ be a minimal non-special divisor on $\C$ and let $\mathcal D^{v_0}$ be the $v_0$-reduced divisor linearly equivalent to $\mathcal D$. 
By Lemma~\ref{lem:reduced},  there exists  a moderator $\mr$ such that $\mathcal D^{v_0} \leq \mr$. The effective divisor $\mr - \mathcal D^{v_0}$ has degree zero, and so must be equal to  zero.
\end{proof}

\subsection{Proof of Theorem~\ref{thm:RR-metrizedcomplexes}}
 The proof of Theorem~\ref{thm:RR-metrizedcomplexes} can now be completed using (a slight simplification of) the idea behind the proof of 
 Riemann-Roch theorem for graphs given in \cite{BN}. 
 The key is the following formula, which provides a useful description of the rank of $\mathcal D$ in terms of minimal non-special divisors.
 We denote by $\deg^+(\mathcal D)$ (resp. $\deg^-(\mathcal D)$) 
 the sum of non-negative (resp. non-positive) coefficients in $\mathcal D$. 
 (Note that $\deg^+(\mathcal D) + \deg^-(\mathcal D) = \deg(\mathcal D)$).
 \begin{prop} \label{prop:magicformula}
  For any divisor $\mathcal D$ on $\C$, we have $$r_\C(\mathcal D) = \min_{\mathcal N \in \mathfrak N}\,\,
  \deg^+(\mathcal D -\mathcal N)-1.$$
 \end{prop}
\begin{proof}
Fix $\mathcal N\in \mathfrak N$, and 
let $\mathcal E$ be the non-negative part of $\mathcal D-\mathcal N$, which is of degree $\deg^+(\mathcal D -\mathcal N)$.
Then $r_\C(\mathcal D) \leq \deg^+(\mathcal D -\mathcal N)$, since $\mathcal D -\mathcal E \leq \mathcal N$ and $r(\mathcal N)=-1$
imply that $|\mathcal D-\mathcal E|=\emptyset$. 
This shows that $r_\C(\mathcal D) \leq \min_{\mathcal N \in \mathfrak N} \,\, \deg^+(\mathcal D -\mathcal N)-1 $. 

To prove the opposite inequality, let $\mathcal E$ be an effective divisor of degree  
$r_\C(\mathcal D)+1$ such that $|D- \mathcal E| = \emptyset$. By Lemma~\ref{lem:reduced}, there exists $\mathcal N \in \mathfrak N$ such that 
$\mathcal D-\mathcal E \leq \mathcal N$, or equivalently, $\mathcal D - \mathcal N \leq \mathcal E$. In particular, 
$\deg^+(\mathcal D-\mathcal N)\leq \deg(\mathcal E) = r_\C(\mathcal D)+1$, which proves the proposition. 
\end{proof}

 To finish the proof of Theorem~\ref{thm:RR-metrizedcomplexes}, note that
 $$\deg^+(\mathcal D- \mathcal N) = \deg(\mathcal D -\mathcal N) - \deg^-(\mathcal D-\mathcal N) = 
 \deg(\mathcal D) - g+1 + \deg^+(\mathcal N-\mathcal D)$$
 and observe that $\mathcal N-\mathcal D = \mathcal K -\mathcal D - (\mathcal K -\mathcal N)$.
 If $\mathcal N$ is linearly equivalent to the moderator 
 $\mathcal M$, c.f. Corollary~\ref{cor:non-special2}, then $\mathcal K -\mathcal N$ 
 is linearly equivalent to the dual moderator $\bar \mr$, and so belongs to $\mathfrak N$. 
 Thus, by Proposition~\ref{prop:magicformula}, we have 
 \begin{align*}
  r_\C(\mathcal D)  
  &= \deg(\mathcal D) - g+1 + \min_{\mathcal N \in \mathfrak N}\,\, \deg^+(\mathcal K - \mathcal D - 
  (\mathcal K -\mathcal N))-1\\
  &= \deg(\mathcal D) - g+1 + \min_{\mathcal N' \in \mathfrak N}\,\, \deg^+(\mathcal K - \mathcal D - \mathcal N')-1 = \deg(\mathcal D) - g+1 + 
  r_\C(\mathcal K - \mathcal D).
 \end{align*}

\section{The specialization map and specialization inequality} \label{sec:specialization}

Let $\K$ be a complete and algebraically closed
non-Archimedean field with non-trivial absolute value $|.|$, $R$ the valuation ring of $\mathbb K$, and 
$\k$ its (algebraically closed) residue field. Let ${\mathcal G} = {\rm val}(\mathbb K^\times)$ be the value group of $\mathbb K$.  
Let $X$ be a smooth, proper, connected curve over $\mathbb K$ and 
let $X^{\rm an}$ be the Berkovich analytic space associated to $X$.  (We assume the reader is familiar with the theory of Berkovich
analytic curves, see e.g. \cite[Section 5]{BPR} which contains everything we need.)

Our first goal will be to define a metrized complex $\C\fX$ associated to a strongly semistable $R$-model ${\mathfrak X}$ for $X$.
We then define a specialization homomorphism from $\Div(X)$ to $\Div(\C\fX)$ and prove a specialization inequality which refines 
Lemma 2.8 and Corollary 2.11 from \cite{bakersp}.
Finally, we give some applications to specialization of canonical divisors, Brill-Noether theory, and Weierstrass points as in {\em loc. cit.} 


\subsection{The metrized complex associated to a semistable model}

Recall that a connected reduced algebraic curve over $\kappa$ is called 
{\em semistable} if all of its singularities are ordinary double points, and is called
{\em strongly semistable} if in addition its irreducible components are all smooth.
A {\em (strongly) semistable model} for $X$ is a flat and integral proper relative curve $\fX$ over $R$
whose generic fiber is isomorphic to $X$ and whose special fiber $\bar{\fX}$ is a (strongly) semistable curve. 

\medskip

Given a semistable model $\fX$ for $X$, there is a canonical associated reduction map ${\rm red} : X(\K) \to \bar{\fX}(\kappa)$ which is defined using the natural bijection between $X(\K)$ and $\fX(R)$.
This extends naturally to a map ${\rm red} : X^{\an} \to \bar{\fX}$.

\medskip

We define a metrized complex associated to a strongly semistable model $\fX$ for $X$ as follows.
Let $G$ be the dual graph of $\bar{\fX}$, so that vertices of $G$ correspond to irreducible components of 
$\bar{\fX}$ and edges of $G$ correspond to intersections between irreducible components.
If $x^e$ is the ordinary double point of $\bar{\fX}$ corresponding to an edge $e$ of $G$, 
the {\em formal fiber} $\red^{-1}(x^e)$ is isomorphic to an open annulus ${\mathbf A}$. 
We define the {\em length of the edge $e$} to be the length of the skeleton of ${\mathbf A}$, 
i.e., the {\em modulus} $\log(b)-\log(a)$ of ${\mathbf A} \cong \{ x \in ({\mathbf A}^1)^{\an} \; | \; a < |T|_x < b \}$.
(The modulus of an open annulus is well-defined independent of the choice of such an analytic isomorphism.)
In this way, we have defined a metric graph $\Gamma = \Gamma_{\fX}$ associated to $\fX$ together with a model $G$.
The irreducible components $C_v$ of $\bar{\fX}$ correspond bijectively to the vertices $v$ of $G$, 
and we let ${\mathcal A}_v \subset C_v$ be the finite set of double points of $\bar{\fX}$ contained in $C_v$, so that there is a natural
bijection between ${\mathcal A}_v$ and the edges of $G$ incident to $v$.  In this way we have defined a metrized complex
$\C\fX$ canonically associated to $\fX$.  

\medskip

One can show that essentially every metrized complex of curves comes from this construction.
The following result is proved in \cite[Theorem 3.24]{ABBR}.

\begin{thm} \label{thm:graph.to.curve}
  Let $\C$ be a metrized complex of $\kappa$-curves whose edge lengths are
  contained in the value group of $\K$.  There exists a
  smooth, proper, connected curve $X$ over $\K$ and a semistable
  model $\fX$ for $X$ such that $\C \cong \C\fX$.  
\end{thm}

\subsection{The metrized complex associated to a semistable vertex set}

It is sometimes useful to define skeleta and metrized complexes in terms of semistable vertex sets rather than semistable models.
(By \cite[Theorem 5.38]{BPR}, there is a bijective correspondence between the latter two objects.)
A {\em semistable vertex set} for $X^{\an}$ is a finite set $V$ of type 2 points of $X^{\an}$ such that the complement of $V$ in $X^{\an}$ is isomorphic
(as a $\K$-analytic space) to the disjoint union of a finite number of open annuli and an infinite number of open balls.  (Such a disjoint union is called
the {\em semistable decomposition} of $X^{\an}$ associated to $V$.)
Any finite set of type 2 points of $X^{\an}$ is contained in a semistable vertex set
(\cite[Proposition 5.27]{BPR}).
The {\em skeleton} $\Gamma = \Sigma(X^{\an},V)$ of $X^{\an}$ with respect to a semistable vertex set $V$ is the union (inside $X^{\an}$) of $V$ and the skeletons of
each of the open annuli in the semistable decomposition associated to $V$.
Using the canonical metric on the skeletons of these open annuli, $\Gamma$ can be naturally viewed as a (finite) {metric graph} 
contained in $X^{\an}$ (see \cite[Definition 5.17]{BPR} for more details).
The skeleton $\Gamma$ comes equipped with a natural model $G$ whose vertices are the points of $V$ and whose edges correspond bijectively to the open annuli in the semistable decomposition associated to $V$.  A semistable vertex set $V$ is called {\em strongly semistable} if the graph $G$ has no loop edges.
Every semistable vertex set is contained in a strongly semistable vertex set.

\medskip

Given a strongly semistable vertex set $V$ for $X^{\an}$, 
there is a canonical corresponding metrized complex $\C V$ of $\kappa$-curves.
Indeed, we have already defined a metric graph $\Gamma = \Sigma(X^{\an},V)$ corresponding to $V$, together with a corresponding model $G$.
For $v \in V$, let $C_v$ be the unique smooth projective curve over $\kappa$ with function field 
$\widetilde{{\mathcal H}(v)}$.
(For $x \in X^{\an}$ of type 2, recall that the residue field $\widetilde{{\mathcal H}(x)}$ of the completed residue field ${\mathcal H}(x)$ of $x$ has transcendence degree one 
over $\kappa$.)
It remains to specify, for each $v \in V$, a bijection $\psi_v$ from the edges of $G$ incident to $v$ to a subset 
$\mathcal A_v$ of $C_v(\kappa)$.  Given such an edge $e$, we define $\psi(e)$ to be the point of $C_v(\kappa)$
corresponding to the tangent direction at $v$ defined by $e$.
(Recall from \cite[Paragraph 5.67]{BPR} that if $x$ is of type 2, there is a canonical bijection between $T_x$ and the set of discrete valuations on 
$\widetilde{{\mathcal H}(x)} = \kappa(C_x)$ which are trivial on $\kappa$.)

\begin{remark}
\label{rmk:modelcompatibility}
Passing to a larger semistable vertex set is compatible with linear equivalence of divisors and does not change the rank of divisors. One can thus associate  a canonical group $\Pic \C X^{\an}$ to $X^{\an}$, defined as 
$\Pic(\C V)$ for any semistable vertex set $V$, together with a canonical rank function $r : \Pic \C X^{\an} \to \ZZ$.
\end{remark}

One can define in a similar way semistable vertex sets and skeleta for an affine curve $X'$ (see \cite[Definition 5.19]{BPR}).
In this case, one must also allow a finite number of punctured open balls in the semistable decomposition and the skeleton is a topologically finite 
but not necessarily finite length metric graph -- it will contain a finite number of infinite rays corresponding to the points of 
$X \setminus X'$, where $X$ is the projective completion of $X'$.

\subsection{Specialization of divisors from curves to metrized complexes}
There is a canonical embedding of $\Gamma_{\fX}$ in the
Berkovich analytification $X^{\an}$ of $X$, as well as a canonical retraction map $\tau : X^{\an} \to \Gamma_{\fX}$.
In addition, there is a canonical reduction map ${\rm red} : X^{\an} \to \bar{\fX}$ sending $X(\K)$
surjectively onto the closed points of $\bar{\fX}$.
The retraction map $\tau$ induces by linearity a specialization map $\tau_* : \Div(X) \to \Div(\Gamma_{\fX})$
which is studied in \cite{bakersp}.
We can promote this to a map $\tau_*^{\C\fX}$ whose target is the larger group $\Div(\C\fX)$ as follows. 
If $P \in X^{\an}$ satisfies $\tau(P)=v \in V$ then either $P$ is the unique point of
$X^{\an}$ with $\red(P)$ equal to the generic point of $C_v$,
or else $\red(P)$ is a nonsingular closed point of $C_v$.  The map $\tau^{\C\fX}_* : \Div(X) \to \Div(\C\fX)$ is obtained by linearly extending the map
$\tau^{\C\fX} : X(\K) \to \Div(\C\X)$ defined by
\[
\tau^{\C\fX}(P) = \left\{
\begin{array}{ll}
\tau(P) & \tau(P) \not\in V \\
{\rm red}({P}) & \tau(P) \in V. \\
\end{array}
\right.
\]

\subsection{Reduction of rational functions and specialization of principal divisors}
\label{section:ReductionOfRationalFunctions}

In this section, we show that if $\mathcal D$ is a principal divisor on $X$, then $\tau^{\C\fX}_*(\mathcal D)$ is a principal divisor on $\C\fX$.
In fact, we show a more precise result (Theorem~\ref{thm:PLspecialization} below).

\medskip

Let $x \in X^{\an}$ be a point of type 2.
Given a nonzero rational function $f$ on $X$, choose $c \in \K^\times$ such that $|f(x)|=|c|$.
Define $f_x \in \kappa(C_x)^\times$ to be the image of $c^{-1}f$ in
$\widetilde{{\mathcal H}(x)} \cong \kappa(C_x)$.
Although $f_x$ is only well-defined up to multiplication by an element of $\kappa^\times$,
its divisor $\div(f_x)$ is canonical and the resulting map $\Prin(X) \to \Prin(C_x)$ is a homomorphism. We call $f_x$ the {\em normalized reduction} of $f$.

\medskip

If $H$ is a $\K$-linear subspace of $\K(X)$, the collection of all possible reductions of nonzero elements of $H$, together with $\{ 0 \}$,
forms a $\k$-vector space $H_x$.  For later use we note the following elementary lemma, which says that $\dim H = \dim H_x$:

\begin{lemma}\label{lem:dimensionreduction} Let $X$ be a smooth proper curve over $\K$, and $x\in X^{\mathrm{an}}$  
a point of type 2. The $\k$-vector space $H_x$ defined by the reduction  to $\widetilde{\H(x)}$ 
of an $(r+1)$-dimensional $\K$-subspace 
$H \subset \K(X)$  has dimension $r+1$. 
\end{lemma}
\begin{proof} 
The inequality $\dim_\k(H_x) \leq r+1$ follows from the observation that 
if $f_0, \dots, f_n$ are linearly dependent in $H$, with $|f_i|=1$ for all $i$, then the reductions
$f_{0,x}, \dots, f_{n,x}$ are linearly dependent in $H_x$.

To prove $\dim_\k(H_x) \geq r+1$, choose a $\K$-basis $f_0, \dots, f_r$ for $H$ consisting of elements of norm $1$
and let $f_{0,x}, \dots, f_{r,x}$ in $H_x$ be the corresponding reductions.
It is not true that $f_{0,x}, \dots, f_{r,x}$ necessarily form a basis of $H_x$: for example, if $f$ and $g$ are linearly independent functions of norm $1$ and $c$ is any scalar 
with $|c|<1$, then $f$ and $f+cg$ are linearly independent but they have the same reduction.  
However, by the following procedure one can construct another basis $g_0, \dots, g_r$ of $H$ with $|g_i|=1$ for all $i$
such that the reductions $g_{i,x}$ are linearly independent over $\k$. Define $g_0=f_0$, and recursively define $g_{i+1}$ as follows: 
among all the rational functions of the form $f_{i+1} - \sum_{j\leq i} a_j g_j$ with $|a_j|\leq 1$ for all $j$, choose one, call it $h_{i+1}$, whose norm is minimal. 
Note that $|f_{i+1} - \sum_{j\leq i} a_j g_j|\leq 1$ so such an element exists by compactness.
Let $c_{i+1}$ be a scalar with $|c_{i+1}|=|h_{i+1}|$ and define $g_{i+1} = c_{i+1}^{-1} h_{i+1}$. 
It is clear that the $\{g_i\}$ form a basis of $H$. We claim that the reductions $g_{i,x}$ form a basis of $H_x$. 
Indeed, if not then there exists a minimal index $i$ and scalars $b_j$ with $|b_j| \leq 1$ for $j=0, \ldots, i$ such that 
$|g_{i+1} - \sum_{j\leq i} b_j g_{j}|<1$. But this implies that
$|f_{i+1} - \sum_{j\leq i} (a_j+c_{i+1}b_j)g_j| < c_{i+1}$,
contradicting the choice of $h_{i+1}$.
\end{proof}

We recall the following useful result from \cite[Theorem 5.69]{BPR}.

\begin{thm}[Slope Formula]  \label{thm:PL}
Let $f$ be a nonzero rational function on $X$, let $X'$ be an open affine subset of $X$ on which 
$f$ has no zeros or poles, and let $F = -\log |f|: (X')^{\an}\to\RR$.  
Let $V$ be a semistable vertex set for $X'$, and let $\Sigma = \Sigma(X',V)$ be the corresponding skeleton.  Then:
  \begin{enumerate}
  \item[$(1)$] $F = F\circ\tau_\Sigma$ where $\tau_\Sigma:(X')^{\an}\to\Sigma$ is the retraction.
  \item[$(2)$]  $F$ is piecewise linear with integer slopes, and $F$ is linear on each edge of $\Sigma$. 
  \item[$(3)$]  If $x$ is a type-$2$ point of $X^{\an}$ and $\vec{\nu}\in T_x$ corresponds to the discrete valuation $\ord_v$ on $\kappa(C_x)$, then
    $d_{\vec \nu} F(x) = \ord_v(f_x)$.
  \item[$(4)$]  $F$ is harmonic at all $x \in {\mathbf H}(X^{\an})$. 
  \item[$(5)$]  Let $x\in X \setminus X'$, let $e$ be the ray in $\Sigma$ whose closure in
    $X^{\an}$ contains $x$, let $y\in V$ be the other endpoint of $e$, and
    let $\vec{\nu }\in T_y$ be the tangent direction represented by $e$.  Then $d_{\vec \nu}
    F(y) = \ord_x(f)$.
  \end{enumerate}
\end{thm}

To each nonzero rational function $f$ on $X$ and each strongly semistable model $\fX$ for $X$, one can associate a rational function 
$\f$ on $\C\fX$ whose $\Gamma$-part $f_{\Gamma}$ is the restriction to 
$\Gamma=\Gamma_{\fX}$ of the piecewise linear function $F = \log|f|$ on $X^{\an}$
and whose $C_v$-part is the normalized reduction $f_v$ (which is well-defined up to multiplication by a non-zero scalar).\footnote{We take $F=\log|f|$ instead of $-\log|f|$ because our convention is that $\ord_{u}(F)$ is the sum of the outgoing slopes of $F$ at $u$.
One could equally well take the opposite convention, defining $\ord_{u}(F)$ to be {\em minus} the sum of the 
outgoing slopes of $F$ at $u$, and then defining $F$ to be $-\log|f|$.  Such a modification would also necessitate a change of sign in our definition of $\div_v(F)$.}

As an application of Theorem~\ref{thm:PL}, we obtain the following important formula:

\begin{thm} \label{thm:PLspecialization}
For every nonzero rational function $f$ on $X$,
\[
\tau_*^{\C\fX}(\div(f)) =  \div(\f).
\]
\end{thm}

\begin{proof} 
Let $D$ be the support of $\div(f)$ in $X$. We apply Theorem~\ref{thm:PL} to $X' = X \setminus D$, the invertible rational function $f$ on $X'$, and a semistable vertex set $V'$ for $X'$ which contains $V$. Let $\Sigma =\Sigma(X'^{\an},V')$ and $\Gamma = \Sigma(X,V)$. Since $V$ is a semistable vertex set for $X$, the closure of $\Sigma \setminus \Gamma$ in $\Sigma$ is a disjoint  union of metric trees $\mathcal T_1,\dots, \mathcal T_s$. For each tree $\mathcal T_i$, denote by $x^i_1,\dots, x^i_{n_i}$ all the points in $D \subset X(\K)$ which are in the closure of rays in $\mathcal T_i$, and let $\overline{\mathcal T_i} = \mathcal T_i \cup \{x_1^i,\dots,x^i_{n_i}\}$. Denote by $y_i$ the unique point of $\Gamma\cap \overline{\mathcal T_i}$. 
In addition, if $y_i \in V$, let $z^i_j = \tau_*^{\C\fX}(x^i_j) \in C_{y_i}(\kappa)$. 

\medskip 

For each $i \in \{ 1, \ldots, s \}$, the restriction of $F= \log|f|$ to $\mathcal T_i$ is harmonic at every point of $\overline{\mathcal T_i} \setminus \{y_i, x^i_1,\dots, x^i_{n_i}\}$ by (4). By (5), $\log|f|$ is linear in a sufficiently small neighborhood of $x^i_j$ in $\mathcal T_i$, and has (constant) slope $\ord_{x^i_j}(f)$ along the unique tangent direction in this interval toward the point $y_i$.  
We infer that the order of $F|_{\mathcal T_i}$ at $y_i$ is equal to $-\sum_{1\leq j\leq n_i}\ord_{x^i_j}(f)$.  By (1), $F$ is constant on any connected component of $X^{\an} \setminus \Sigma$. By (4), the restriction $f_\Gamma$ of $F$ to $\Gamma$ satisfies $\ord_{y_i}(f_\Gamma) = - \ord_{y_i}(F|_{\mathcal T_i}) =  \sum_{1\leq j\leq n_i} \ord_{x^i_j}(f)$. Since $f_\Gamma $ is linear on each edge of $\Gamma$ by (2), it follows that $\tau_*(\div(f)) = \div(f_\Gamma)$. 

\medskip

Fix $i \in \{ 1, \ldots, s \}$ and suppose that $y_i \in V$. For any tangent direction $\vec \nu  \in T_{y_i}$, by (3) we have $d_{\vec \nu} F = - \ord_\nu(f_{y_i})$. Since $F$ is locally constant on $X^{\an} \setminus \Sigma$, we infer that $d_{\vec \nu} F = 0$ for any  $\vec \nu \in T_x$ not corresponding to a 
direction in $\Sigma$.  This shows that $\div(f_{y_i})$ is supported on $\{z^i_1,\dots, z^i_{n_i}\}\cup \mathcal A_{y_i}$.  By (3), for any point $\nu$ in $\mathcal A_{y_i}$, corresponding to a tangent direction ${\vec \nu}$, we have $d_{\vec \nu} f_\Gamma + \ord_\nu(f_{y_i}) =0$. This shows that the coefficient of $\nu$ in $\div(\f)$ is zero.  For a point $\nu \in \{z^i_1,\dots, z^i_{n_i}\}$, let $\overline{\mathcal T}_{i,\nu}$ be the closure in $\Sigma$ of the connected component of $\Sigma \setminus y_i$ which contains all the points $x^i_j$ with $z^i_j = \nu$. By (1) and (4), $F$ is harmonic at any point of $\overline{\mathcal T}_{i,\nu} \setminus (D \cup \{y_i\})$, and is linear of slope $\ord_{x^i
 _j}(f)$ in a sufficiently small neighborhood of $x^{i}_j$ in $\mathcal T_{i,\nu}$ along the unique tangent direction in this interval toward the point $y_i$. By (3), we infer that  
 $\ord_\nu(f_{y_i}) = - d_{\vec \nu}(F) = \sum_{x^i_j: \, z^i_j =\nu} \ord_{x^i_j} (f)$. We conclude that $\tau_*^{\C\fX}\div(f) = \div(\f)$.
\end{proof}

In particular, it follows that $\tau_*^{\C\fX}(\Prin(X)) \subseteq \Prin(\C\fX)$.

\subsection{The specialization inequality}

Fix a strongly semistable $R$-model ${\mathfrak X}$ for $X$ and let $\C\fX$ be the metrized complex associated to $\fX$. 
The following is a variant of Lemma 2.8 and Corollary 2.11 from \cite{bakersp}: 

\begin{thm}[Specialization Inequality]
\label{lem:BakSpecLem}
For every divisor $D \in \Div(X)$,
 \[
r_X(D) \leq r_{\mathcal {C}{\fX}}(\tau^{\mathcal C\fX}_*(D)).
 \]
\end{thm}

\begin{proof}
Let $\mathcal D := \tau^{\C\fX}_*(D)$. 
It suffices to prove that if $r_X(D) \geq k$, then $r_{\C{\fX}}(\mathcal D) \geq k$ as well.

\noindent The base case $k = -1$ is obvious.  Now suppose $k = 0$, so that $r_X(D) \geq 0$.
Then there exists an effective divisor $E \in \Div(X)$ with
$D - E \in \Prin(X)$.  Since $\tau^{\C\fX}_*$ is a homomorphism and 
takes principal (resp. effective) divisors on $X$ to principal (resp. effective) 
divisors on $\C{\fX}$, we have $\mathcal D = \tau_*^{\C\fX}(D) \sim \tau_*^{\C\fX}(E) \geq 0$, so that 
$r_{\C{\fX}}(\mathcal D) \geq 0$ as well.

\medskip

We may therefore assume that $k \geq 1$. For each $v\in V$, let  
$\R_v \subset C_v(\k) \setminus \mathcal A_v$ be a subset of size $g_v+1$ and denote by $\R$ the union of the $\R_v$. 
By Theorem~\ref{thm:f-width}, $\R$ is a rank-determining set in $\C$.
Let $\mathcal E$ be an effective divisor of degree $k$ with support in $\R$. For each $x \in {\rm supp}(\mathcal E)$, there exists a point $P \in X(\K)$ whose reduction is $x$,
so there is an effective divisor $E$ of degree $k$ on $X$ with $\tau^{\C\fX}_*(E) = \mathcal E$.
By our assumption on the rank of $D$, $r_X(D - E) \geq 0$, and so $r_{\C{\fX}}(\mathcal D - \mathcal E) \geq 0$. 
Since this is true for any effective divisor of degree $k$ with support in $\R$, and since $\R$ is rank-determining,    we infer that $r_{\C\fX}(\mathcal D) \geq k$ as desired.
\end{proof}

\begin{remark}
There are many examples where the inequality in Theorem~\ref{lem:BakSpecLem} can be strict.
For example, if $g_v = 0$ for all $v \in V$ then 
$r_{\mathcal {C}{\fX}}(\tau^{\mathcal C\fX}_*(D)) = r_{\Gamma}(\tau_*(D))$
and thus the examples of strict inequality from \cite{bakersp} apply.
At the other extreme, if $C = \bar{\fX}$ is smooth (so that $G$ is a point) then the specialization inequality becomes the well-known semicontinuity statement $h^0(D) \leq h^0(\bar{D})$, where $\bar{D} \in \Div(C)$ is the reduction of $D$, and
it is clear that such an inequality can be strict; for example, take $D=(P)-(Q)$ where 
$X$ has genus at least $1$ and $P,Q \in X(\K)$ are distinct points with the same reduction in $C(\kappa)$.
\end{remark}

\subsection{Refined versions of weighted specialization lemma}

In this section we use the specialization lemma for divisors on metrized complexes to provide a framework for obtaining stronger versions of the weighted specialization lemma from~\cite{AC}. 
Our key technical tool for this will be the ``$\eta$-function'' on a metrized complex.  
To help motivate the definition, we first study an analogous $\eta$-function which controls the rank of divisors on connected sums.

\subsubsection{The $\eta$-function and the rank of divisors on connected sums}\label{sec:rank-csum}

Suppose the underlying graph of a metrized complex $\C$  has a bridge edge $e = \{v_1,v_2\}$, 
and denote by $\C_1$ and $\C_2$ the two metrized complexes obtained by removing the edge $e$ from $\C$ ($v_i$ is a vertex of $\C_i$). 
Denote by $x_1$ and $x_2$ the points of $\C_1$ and $\C_2$, respectively, corresponding to the edge $e$. 
We say that $\C$ is a {\it connected sum} of 
$\C_1$ and $\C_2$.

There is an addition map $\Div(\C_1) \oplus \Div(\C_2) \rightarrow \Div(\C)$ 
which associates to any pair of divisors $(\mathcal D_1,\mathcal D_2) \in \Div(\C_1) \oplus \Div(\C_2)$
the divisor $\mathcal D_1+\mathcal D_2$  on $\C$ whose restriction to each $\C_i$ is $\mathcal D_i$.

The following proposition provides a precise description of the $\C$-rank of $\mathcal D_1 + \mathcal D_2$ in terms of the two rank functions 
$r_{\C_1}(\cdot)$ and $r_{\C_2}(\cdot)$. 
In order to state the result, we introduce a function $\eta = \eta_{x_2,\mathcal D_2}: \mathbb N\cup\{0\} \rightarrow \mathbb Z$ 
 defined by the condition that $\eta(k)$ is the smallest integer $n$ such that 
$r_{\C_2}(\mathcal D_2 + n(x_2)) = k$.
(By the Riemann-Roch theorem for $\C_2$, $\eta$ is well-defined.)

\begin{prop} \label{prop:ConnectedSumRankForMetrizedComplexes} 
With notation as above, we have
\begin{equation}\label{eq:wrank3}
 r_{\C}(\mathcal D) = \min_{k\in \mathbb N\cup\{0\}} \Bigl\{\,k + r_{\C_1}(\mathcal D_1-\eta(k)(x_1))\,\Bigr\}.
 \end{equation}
\end{prop}

\begin{proof}
We prove the equality of the two sides of Equation~\eqref{eq:wrank3} by showing that the inequalities 
$\leq$ and $\geq$ both hold. By Theorem~\ref{thm:f-width}, 
we can assume that all the effective divisors we consider below are supported on geometric points (i.e., on points of the curves $C_v$).

\medskip 

\noindent $(\geq)$
Let $r$ be the right-hand term in \eqref{eq:wrank3} and let $\mathcal E$ be an effective divisor of degree $r$ on 
$\C$ supported on geometric points. We need to show the existence of a rational function $\f$ on $\C$ such that $\div(\f)+\mathcal D -\mathcal 
E\geq 0$. There is a (unique) decomposition of $\mathcal E$ as 
$\mathcal E=\mathcal E_1+\mathcal E_2$ with $\mathcal E_1$ and $\mathcal E_2$  effective divisors
of degree $r-k$ and $k$, respectively, having support on geometric points.
By the definition of $\eta$, there exists a rational function $\f_2$ on $\C_2$ such that 
$\div(\f_2) + \eta(k)(x_2)+ \mathcal D_2 -\mathcal E_2 \geq 0$. By the choice of $r$ as the minimum 
in Equation~\eqref{eq:wrank3},
we have $r_{\C_1}\bigl(\mathcal D_1-\eta(k)(x_1)\bigr)\geq r-k$, so  
there exists a rational function  $\f_1$ on $\C_1$ such that $\div(\f_1) +\mathcal D_1-\eta(k)(x_1)-
\mathcal E_1 \geq 0$. 
We can in addition suppose that the $\Gamma_1$-part of $\f_1$ takes value zero at $v_1$. 
Up to adding a constant to the $\Gamma_1$-part of $\f_1$, we can 
define a rational function $\f$ which restricts to $\f_1$ and $\f_2$ 
on $\C_1$ and $\C_2$, respectively, and which is linear of slope $-\eta(k)$ on the oriented edge $(v_1,v_2)$ in $\Gamma$. 
For this rational function, we clearly have 
\begin{align*}
\div(\f) + \mathcal D -\mathcal E = \div(\f_1)+\mathcal D_1 -\eta(k)&(x_1)-\mathcal E_1  \\
&+ \div(\f_2)+\mathcal D_2 + \eta(k)(x_2)-\mathcal E_2 \geq 0.
\end{align*}

\medskip

\noindent $(\leq)$
It will be enough to show that for each $k\in \mathbb N\cup\{0\}$, 
the inequality 
$$r_\C(\mathcal D) \leq k + r_{\C_1}\bigl(\mathcal 
D_1-\eta(k)(x_1)\bigr)$$ 
holds. 
Let $\widetilde r =  r_\C(\mathcal D)$. 
We can obviously restrict to the case $k\leq \widetilde r$.  
For each effective divisor $\mathcal E_2$ of degree $k$ supported on geometric points of $\C_2$, 
let $\f_{\mathcal E_2}$ be a rational function on $\C_2$ with the property that the coefficients of 
$\div(\f_{\mathcal E_2}) + \mathcal D_2-\mathcal E_2$ outside $x_{2}$ are all non-negative,
and such that in addition, the coefficient of $x_2$ in $\div(\f_{\mathcal E_2}) + \mathcal D_2-\mathcal E_2$ is maximized among all 
the rational functions with this property.
Denote by $n_{\mathcal E_2}$ the coefficient of $x_2$ in $\div(\f_{\mathcal E_2}) + \mathcal D_2-\mathcal E_2$ 
and let $\bar{\mathcal E}_2$  be an effective divisor of degree $k$ on $\C_2$, supported on geometric points of $\C_2$,
for which $n := n_{\bar{\mathcal E}_2}$ is minimal, i.e.,
$$n_{\bar{\mathcal E}_2} = \min\{n_{\mathcal E_2}\,|\,\mathcal E_2\geq 0:\,\deg(\mathcal E_2)=k\}.$$
Note that by the choice of $n$, for an effective divisor $\mathcal E_2$ of degree $k$ supported on geometric points of $\C_2$ 
and the rational function $\f_{\mathcal E_2}$, we have  $\div(\f_{\mathcal E_2}) + \mathcal D_2-\mathcal 
E_2 -n(x_2) \geq 0$, which shows that $r_{\C_2}(\mathcal D_2 -n(x_2))\geq k$. By the definition of 
$\eta$, it follows that $\eta(k) \leq -n$.

\medskip

By the definition of $\f_{\bar{\mathcal E}_2}$, for any rational function $\f_2$ on $\C_2$ with the property 
that $\div(\f_2)+\mathcal D_2-\bar{\mathcal E}_2$ 
has non-negative coefficients outside $x_2$, 
the coefficient of $x_2$ in $\div(\f_2)+\mathcal D_2-\bar {\mathcal E}_2$ is at most $n \leq -\eta(k)$. 

Let $\mathcal E_1$ be an effective divisor of degree $\widetilde r - k$ on $\C_1$ supported on geometric points. 
There exists a rational function $\f$ on $\C$ such that $\div(\f)+\mathcal D -\mathcal E_1-\bar{\mathcal E}_2\geq 0$. 
Let $\f_1$ and $\f_2$ be the restrictions of 
$\f$ to $\C_1$ and $\C_2$, respectively. By the preceding paragraph, the coefficient of $x_2$ in 
$\div(\f_2)+\mathcal D_2-\bar{\mathcal E}_2$ is at most $-\eta(k)$. Since the coefficient of $x_2$ in 
$\div(\f)+\mathcal D -\mathcal E_1-\bar{\mathcal E}_2$ is non-negative, the $\Gamma$-part of $\f$ has slope at least 
$\eta(k)$ along the tangent direction emanating from $v_2$ which corresponds to $x_2$. Since the coefficients of 
$\div(\f)$ at the interior points of the edge $\{v_1,v_2\}$ are all non-negative, it follows that the slope of the $\Gamma$-part of $\f$ 
along the tangent direction emanating from $v_1$ on the edge $\{v_1,v_2\}$ is at most $-\eta(k)$, 
which by the assumption  $\div(\f)+\mathcal D -\mathcal E_1-\bar{\mathcal E}_2\geq 0$ implies that the coefficient of 
 $x_1$ in $\div(\f_1)+\mathcal D_1- \mathcal E_1$ must be at least $\eta(k)$.  It follows that $\div(\f_1)+\mathcal D_1 -\eta(k)(x_1)- \mathcal E_1\geq 0$. Since $\mathcal E_1$ was an arbitrary effective 
 divisor of degree $\widetilde r-k$ supported on geometric points of $\C_1$, it follows that $r_{\C_1}\bigl(\mathcal 
D_1-\eta(k)(x_1)\bigr)\geq \widetilde r-k$, and the claim follows.
\end{proof}
\begin{remark}
In the same spirit, consider two metric graphs $\Gamma_1$ and $\Gamma_2$, 
and suppose that two distinguished points $v_1\in \Gamma_1$ and $v_2\in \Gamma_2$ are given.
The {\it wedge} or {\it direct sum} of $(\Gamma_1,v_1)$ and $(\Gamma_2,v_2)$, denoted $\widetilde{\Gamma} = \Gamma_1 \vee \Gamma_2$, is the metric graph obtained by identifying the points $v_1$ and $v_2$ in the disjoint union of 
$\Gamma_1$ and $\Gamma_2$. Denote by $v \in \widetilde{\Gamma}$ the image of $v_1$ and $v_2$ in $\widetilde \Gamma$. 
(By abuse of notation, we will use $v$ to denote both $v_1$ in $\Gamma_1$ and $v_2$ in $\Gamma_2$.) 
We refer to $v\in \widetilde \Gamma$ as a {\it cut-vertex} and to $\widetilde \Gamma=\Gamma_1\vee \Gamma_2$ 
as the decomposition corresponding  to the cut-vertex $v$. Note that there is an addition map $\Div(\Gamma_1) \oplus \Div(\Gamma_2) \rightarrow \Div(\widetilde \Gamma)$ which associates to any pair of divisors $D_1$ and $D_2$ in $\Div(\Gamma_1)$ and $\Div(\Gamma_2)$ the divisor,
$D_1+D_2$  on $\widetilde \Gamma$ defined by pointwise addition of coefficients in $D_1$ and $D_2$. Let $r_1(\cdot)$, $r_2(\cdot)$, and 
$\widetilde{r}(\cdot) = r_{\Gamma_1 \vee \Gamma_2}(\cdot)$ be 
the rank functions in $\Gamma_1, \Gamma_2$, and $\widetilde{\Gamma}$, respectively. 
For any non-negative integer $k$, denote by $\eta_{v,D_2}(k)$, or simply $\eta(k)$, the smallest integer $n$ such that 
$r_2 \bigl(D_2 + n(v) \bigr) = k$. Then for any divisor $D_1$ in $\Div(\Gamma_1)$, we have (by an entirely similar argument)
\begin{equation}
\label{eq:wrank2}
\widetilde{r}(D_1 + D_2) = \min_{k\in \mathbb N \cup \{0\}}\, \Bigl\{\,k+r_1\bigl(D_1 -\eta(k)\,(v)\bigr)\,\Bigr\}.
\end{equation}
\end{remark}

\begin{remark} The above formalism is very handy for studying metric graphs such as the chain of cycles used in~\cite{CDPR}.
In particular, the lingering lattice paths studied in~\cite{CDPR} in connection with a tropical proof of the Brill-Noether theorem 
can be understood in a natural way as values of $\eta$-functions.
The $\eta$-function formalism is also useful for studying limit linear series, see for example the proof of 
Theorem~\ref{thm:limitseries} below.
\end{remark}

We now define a partial analogue of the $\eta$-function introduced above for a general metrized complex $\C$ 
and provide an analogue of 
Proposition~\ref{prop:ConnectedSumRankForMetrizedComplexes}.

\medskip

Consider a divisor $\mathcal D$ on $\C$ with $C_v$-part $D_v$ for $v \in V$. For every vertex $v$ of $G$, define the function 
$\eta_v: \mathbb N\cup\{0\} \rightarrow \mathbb N\cup \{0\}$ as follows. For any $k\geq 0$, $\eta_v(k)$ is the smallest integer $n \geq 0$  such that there exists a divisor 
$\widetilde D_v$ of degree $n-\deg(D_v)$ supported on the points of $\mathcal A_v \subset C_v(\kappa)$
such that $r_{C_v}(D_v + \widetilde D_v) = k$.  For any effective divisor $E=\sum_{v\in V(G)} E(v)(v)$ on $G$, 
define 
\begin{equation*}
\eta(E) := \sum_{v\in V(G)} \eta_v(E(v))\,(v).
\end{equation*}
\noindent Note that the definition of $\eta$ depends on the data of the divisors $D_v \in \Div(C_v)$.

\begin{prop}
\label{prop:wrank3}
Let $\C$ be a metrized complex of algebraic curves over $\kappa$. For any divisor $\mathcal D$ on $\C$ with $\Gamma$-part $D_\Gamma$, we have
\[ r_\C(\mathcal D)\,\, \leq\,\, \min_{E \geq 0}\, \Bigl(\,\deg(E) + r_\Gamma(D_\Gamma - \eta(E))\,\Bigr), \] 
where the minimum is over all effective divisors $E$ on $G$.
\end{prop}

\begin{proof}
Fix an effective divisor $E$ supported on the vertices of $G$. We need to prove that
$$r_\C(\mathcal D) \leq \deg(E) + r_\Gamma(D_\Gamma - \eta(E)).$$
For the sake of contradiction, suppose this is not the case, i.e., that
$\deg(E) + r_\Gamma(D_\Gamma - \eta(E)) < r_{\C}(\mathcal D)$. 
Since $r_{\C}(\mathcal D) - \deg(E) > r_\Gamma(D_\Gamma - \eta(E))$, 
Luo's Theorem (which is a special case of Theorem~\ref{thm:f-width}) implies that 
there exists a divisor $E^0$ of degree 
$r_{\C}(\mathcal D) - \deg(E)$ supported on $\Gamma \setminus V(G)$ such that 
$|D_\Gamma - E^0_\Gamma -\eta(E)| =\emptyset.$

\medskip

Consider all the effective divisors on $\C$ of the form $E^0 + \sum_{v\in V} E_v$, where each $E_v$ is an effective divisor of degree $E(v)$ on $C_v$. 
 Note that $E^0 + \sum_{v\in V} E_v$ has degree $r_{\C}(\mathcal D)$, so that by the definition of the rank function, for any choice of effective divisors $E_v$ on $C_v$ there exists a rational function $\f$ on $\C$ 
such that $\div(\f) + \mathcal D- E^0 -\sum_v E_v \geq 0$.
Denote by $f_\Gamma$ and $f_v$ the different parts of $\f$, so in particular we have 
$\div(f_\Gamma) + D_\Gamma -E^0 -E \geq 0.$ One can assume 
as in the proof of Proposition~\ref{prop:combinrank}
that the slope of $f_\Gamma$ along the edges 
incident to each vertex of $G$ is fixed. This shows the existence of the divisor $\widetilde D_v = \div_v(f_\Gamma)$  supported on the points $\mathcal A_v \subset C_v(\kappa)$ such that $D_v + \widetilde D_v$ has rank at least $E(v)=\deg(E_v)$ on $C_v$.  By the definition of $\eta_v$, 
$\widetilde D_v$ has degree at least $\eta_v(E(v)) -\deg(D_v)$. Since the degree of $\widetilde D_v$ 
is the sum of the slopes of $f_\Gamma$ along the edges adjacent to $v$, which is by definition 
$\mathrm{ord}_v(f_\Gamma)$, for each $v\in V(G)$ we obtain 
$\mathrm{ord}_v(f_\Gamma) \geq \eta_v(E(v)) - \deg(D_v)$, i.e., 
$\mathrm{ord}_v(f_\Gamma) + D_\Gamma(v) - \eta_v(E(v)) \geq 0$.
In addition, for each non-vertex point $u \in \Gamma$ we must have 
$\mathrm{ord}_u(f_\Gamma) + D_\Gamma(u) - E^0(u) \geq 0$. 
This implies that $|D_\Gamma - E^0 -\eta(E)|\neq \emptyset$, a contradiction. 
\end{proof}

\subsubsection{Weighted specialization lemma}\label{section:weightedspecialization}
Let $\mathbb K$ be a complete and algebraically closed
non-Archimedean field with non-trivial absolute value $|.|$, $R$ the valuation ring of $\mathbb K$, and 
and $\kappa$ its residue field. Let ${\mathcal G} = {\rm val}(\mathbb K^\times)$ be the value group of $\mathbb K$.  
Let $X$ be a smooth, proper, and connected $\mathbb K$-curve
and fix a semistable $R$-model ${\mathfrak X}$.
Let $(\Gamma,\omega)$ denote the skeleton of ${\mathfrak X}$ 
together with the weight function $\omega: V \rightarrow \mathbb N$ defined  by 
$\omega(v) = g_v$.  Let $\mathcal W = \sum_{v \in V} g_v(v)$.
Following~\cite{AC}, we define a metric graph $\Gamma^\#$ by attaching $g_v$ loops of arbitrary positive length at each point $v \in \Gamma$
(see \cite{AC} for a discussion of the intuition behind this construction.)

\medskip

Any divisor $D$ in $\Gamma$ defines a divisor on $\Gamma^\#$ with the same support and coefficients (by viewing $\Gamma$ as a subgraph of $\Gamma^\#$). By an abuse of notation, we will also denote this divisor by $D$. 
However, we distinguish the ranks of $D$ in $\Gamma$ and $\Gamma^\#$ by using the notation $r_\Gamma(D)$ for the former and $r_{\Gamma^\#}(D)$ for the latter. 
Following~\cite{AC}, the {\em weighted rank} $r^\#$ of $D$ is defined by $r^\#(D):= r_{\Gamma^\#}(D)$. Applying the formalism of the $\eta$-function from the previous subsection, we obtain the following useful formula for $r^\#$:

\begin{cor}
\label{prop:wrank} 
For every $D \in \Div(\Gamma)$, we have
\[
r^\#(D) = \min_{0\leq E \leq \mathcal W} \bigl(\deg(E)+ r_\Gamma(D-2E)\bigr).
\]
\end{cor}

\begin{proof} Proceeding by induction, we are reduced to the case 
$\deg(\mathcal W)=1$.
This case follows from (\ref{eq:wrank2}) applied to $\Gamma_1=\Gamma$, $\Gamma_2$ a circle, $\widetilde \Gamma =  \Gamma_1\vee \Gamma_2 $, $D_1 = D$, and $D_2=0$, since in this case one has $\eta(0) = 0$, $\eta(1)=2$, and $\eta(k) = k+1$ for $k>1$.
\end{proof}

The following result is a generalization to metric graphs $\Gamma$, and to smooth curves $X$ over not necessarily discrete non-Archimedean fields, of
the corresponding statement for graphs which was obtained by different methods in~\cite{AC}.

\begin{thm}[Weighted specialization inequality]
\label{thm:weightedmetricspecialization}
For every divisor $D \in \Div(X)$, we have $r_X(D) \leq r^\#(\tau_*(D))$.
\end{thm}

\begin{proof}
Let $\mathcal D =\tau_*^{\C\fX}(D)$ 
and note that the $\Gamma$-part $D_\Gamma$ of $\mathcal D$ is equal to $\tau_*(D)$. 
Applying Proposition~\ref{prop:wrank3}, we obtain
\begin{equation}
\label{eq:wrank3cor}
r_{\C\fX}(\tau_*^{\C\fX}(D)) \leq \min_{E \geq 0}\, \Bigl(\,\deg(E) + r_\Gamma\bigl(\tau_*(D) - \eta(E)\bigr)\,\Bigr).
\end{equation}

For each vertex $v$ of $G$, define the function $\bar \eta_v : \mathbb N\cup\{0\} \rightarrow \mathbb N\cup\{0\}$ by $\bar\eta_v(k)=2k$ for $k \leq g_v$ and $\bar \eta_v (k) = k+g_v$ for $k \geq g_v$. 
If $E$ is an effective divisor supported on vertices of $G$, define 
$\bar \eta(E) = \sum_v \bar\eta(E(v))(v)$. 

By the definition of $\eta_v$ and Clifford's theorem for the curve $C_v$, for any integer $k\geq 0$ we have $\eta_v(k) \geq \bar \eta_v(k)$.  
Combining this with \eqref{eq:wrank3cor}, we get the inequality
\begin{equation}
\label{eq:wrank3corbis}
r_{\mathcal{C}{\fX}}(\tau_*^{\mathcal C\fX}(D)) \leq \min_{E \geq 0}\, \Bigl(\,\deg(E) + r_\Gamma\bigl(\tau_*(D) - \bar \eta(E)\bigr)\,\Bigr).
\end{equation}
 Note that for $E(v) > g_v$, we have $\bar \eta_v(E(v)) = g_v+E(v)$ and so, whatever the other coefficients of $E$ are, we have 
\begin{align*}
\deg(E) + r_\Gamma\bigl(\tau_*(D)-\bar \eta(E)\bigr) &= \deg(E-(v))+ 1+r_\Gamma\bigl(\tau_*(D)-\bar\eta(E-(v))-(v)\bigr) \\
&\geq \deg(E-(v)) + r_\Gamma(\tau_*(D)-\bar\eta(E-(v))).
\end{align*}
In other words, in taking the minimum in \eqref{eq:wrank3corbis}, one can restrict to effective divisors $E$ with $E(v) \leq g_v$ for all $v$, i.e., $0\leq E \leq \mathcal W$. 
Finally, applying \eqref{eq:wrank3corbis}, Theorem~\ref{lem:BakSpecLem}, and Corollary~\ref{prop:wrank}, 
we obtain the inequality $r_X(D) \leq r^\#(\tau_*(D)).$
\end{proof}

\subsection{Some applications} We give some direct applications of the results of this section.
\subsubsection{Specialization of canonical divisors}
Let $K_X$ be a canonical divisor on the curve $X/\K$. By Theorem~\ref{lem:BakSpecLem}, we have $r_{\C\fX}\left( \tau_*^{\C\fX}(K_X)\right) \geq r(K_X) = g(X)-1 = g(\C\fX)-1$. By the Riemann-Roch theorem for metrized complexes (Theorem~\ref{thm:RR-metrizedcomplexes}), the divisor $\mathcal \tau_*^{\C\fX}(K_X) - \mathcal K$ has degree zero and non-negative rank, and thus $\mathcal \tau_*^{\C\fX}(K_X) \sim \mathcal K$.  In particular, we have $\tau_*(K_X) \sim K^\#$.

\begin{remark}
As mentioned in the Introduction, for discretely valued $R$ the fact that $\tau_*(K_X) \sim K^\#$ follows from the adjunction formula for arithmetic surfaces.  
But it does not seem straightforward to generalize that argument to the not necessarily discretely valued case.
\end{remark}

\subsubsection{Brill-Noether theory}
The following result is an immediate consequence of Theorem~\ref{thm:graph.to.curve}
combined with Theorem~\ref{lem:BakSpecLem}.

\begin{cor}[Brill-Noether existence theorem]
\label{cor:BNexistence}
Let $\k$ be a field and let $g,r,d$ be nonnegative integers.  
If the Brill-Noether number $\rho^r_d(g) := g - (r+1)(g-d+r)$ is nonnegative,
then for every metrized complex $\C$ of $\k$-curves 
with $g(\C) = g$, there exists a divisor $\D \in \Div^d_+(\C)$ such that $r_\C(\D) \geq r$.
\end{cor}

\begin{remark}
In particular, combining the original Specialization Lemma from \cite{bakersp} with 
Theorem~\ref{thm:graph.to.curve} yields a new proof of \cite[Theorem 3.20]{bakersp}, which says that 
if $\rho^r_d(g) \geq 0$ then  for every metrized graph of genus $g$ there exists 
$D \in \Div^d_+(\Gamma)$ such that $r(D) \geq r$.  The present proof is more direct in that one does not need the approximation arguments given in Lemma 3.17 and Corollary 3.18 of {\em loc.~cit.}
\end{remark}

\begin{remark}
The notion of Brill-Noether rank on metric graphs introduced in~\cite{LPP}  and the corresponding specialization lemma extend to the context of metrized complexes as well. Let $\C$ be a metrized complex of $\k$-curves, and let $W^r_d(\C) \subset \Pic^d(\C)$ be the subset of all divisors of degree $d$ and rank at least $r$. Define $w^r_d(\C)$ to be the largest integer $n$ such that for every effective divisor $\E$ of degree $r+n$ on $\C$, there exists $\D \in W^r_d(\C)$ such that $\D-\E \geq 0$.  If $W^r_d(\C) =\emptyset$, define $w^r_d(\C) =-1$. 
If $\fX$ is a strongly semistable $R$-model of a proper smooth curve $X/\K$ such that $\C\fX = \C$, and $W^r_d(X) \subset \Pic^d(X)$ denotes the Brill-Noether locus of $X$, then we have the inequality $w^r_d(\C) \geq \dim\,W^r_d(X)$.
\end{remark}
 
\subsubsection{Weierstrass points on metrized complexes of $\k$-curves}
Let $\C$ be a metrized complex of $\k$-curves.  
We say that a point $x$ in $\C$ is a {\em Weierstrass point} if $r_\C\bigl(g(\C)(x)\bigr) \geq 1$.
By  Theorem~\ref{lem:BakSpecLem}, the specialization of a Weierstrass point on $X$ is a Weierstrass point on $\C$.
And by Theorem~\ref{thm:graph.to.curve} combined with Theorem~\ref{lem:BakSpecLem}, 
any metrized complex of $\k$-curves of genus at least two has at least one Weierstrass point.

\section{Limit linear series}\label{sec:limitseries}
Our aim in this section is to compare our divisor theory on metrized complexes with the Eisenbud-Harris theory of limit linear series for curves of compact type
\cite{EH86}.
As we will see,  our comparison results lead naturally to a generalization of the notion of limit linear series to nodal curves which are 
not necessarily of compact type. 
Note that since Eisenbud and Harris work over a discrete valuation ring, in order to compare our theory to theirs we will sometimes restrict to this setting; we note, however, that most of our results hold over the valuation ring of an arbitrary non-trivially valued non-Archimedean field.

\subsection{Degeneration of linear series}  
We start by reviewing some basic facts and definitions concerning the degeneration of linear series in families.
 Let $Y$ be a smooth projective  curve over $\k$. Recall that a linear series $L$ of (projective) dimension $r$ and degree $d$, or simply a $\g^r_d$,  over $Y$ consists of a pair $(\mathcal L, H)$ with $\mathcal L$ an invertible sheaf  of degree $d$ on $Y$ and $H$ a subspace of $H^0(Y,\mathcal L)$ of $\kappa$-dimension $r+1$.

 \medskip
 
Let $\phi: \mathcal X \rightarrow B = \Spec R$ be a regular smoothing of a strongly semistable curve $X_0$ over a discrete valuation ring $R$ with fraction field $K$. This means that $\phi$ is proper, $\mathcal X$ is regular, the generic fiber $\mathcal X_\eta$ of $\phi$ is smooth, and the special fiber of $\phi$ is $X_0$. Denote by $G=(V,E)$ the dual graph of $X_0$ and let $\{X_v\}_{v\in V(G)}$ be the set of irreducible components of $X_0$.  

\begin{prop} Any line bundle $\mathcal L_\eta$ on $\mathcal X_\eta$ extends to $\mathcal X$. The restriction of any two extensions $\mathcal L_1$ and $\mathcal L_2$ of $\mathcal L_\eta$ to the special fiber $X_0$ are combinatorially equivalent (c.f. Section~\ref{sec:mcnodal}).
\end{prop}

\begin{proof} The first part is a well-known consequence of the regularity of $\mathcal X$ (see for example \cite[Chapter II, 6.5(a) and 6.11]{HartshorneAG}).  For the second part, note that any two extensions $\mathcal L_1$ and $\L_2$ of $\L_\eta$ differ by a divisor supported on the special fiber $X_0$. 
In other words, there exist integers $n_v$ for each vertex $v$ of the dual graph of $X_0$ such that $\mathcal L_1 \simeq \mathcal L_2(\sum_v n_v X_v)$. It is now easy to see that, in the notation of Section~\ref{sec:mcnodal}, $\pi(\mathcal L_1|_{X_0})$ and $\pi(\mathcal L_2|_{X_0})$ differ by the divisor of the function $f_G:V(G) \rightarrow \mathbb Z$ which sends $v$ to $n_v$. 
\end{proof}

Consider now a $\g^r_d$ $L_\eta=(\mathcal L_\eta, H_\eta)$ on $\mathcal X_\eta$ and choose an extension $\L$ of $\L_\eta$ to $\mathcal X$. 
Since cohomology commutes with flat base change, there is a natural isomorphism $H^0(X_\eta,\mathcal L_\eta) \cong H^0(X,\L) \otimes_R K$ and 
$H_\eta$ corresponds to a free $R$-submodule H of $H^0(\mathcal X, \mathcal L)$ such that $H \otimes_R K \cong H_\eta$. 
The restriction $H_0$ of $H$ to the special fiber $X_0$ is an $(r+1)$-dimensional subspace of $H^0(X_0,\L_0)$. This shows the existence of the limit of a $\g^r_d$ once the extension $\mathcal L$ of $\L_\eta$ is fixed. However, since the extension is not unique, the limit pair $(\L_0, H_0)$ is not unique.  
If one were to restrict to a single such extension, there would inevitably be a loss of important information; this can be
seen, for example, in the study of limits of ramification points and in the study of limits of smooth hyperelliptic, and more generally, smooth $d$-gonal curves, see for example~\cite{HM82,EH86, EM, Ran}.  Some of the difficulties in formulating a good notion of limits for linear series are already manifest in the atypical behavior of linear series on reducible curves, since (as mentioned earlier) they do not satisfy many of the classical theorems which govern the behavior of (global sections of) linear series on irreducible curves. 

\medskip

The Eisenbud-Harris theory of limit linear series on curves of compact type provides a way to keep track of the geometry in the limit by choosing an {\it aspect} of the limit for each irreducible component, with some extra conditions relating these different aspects at nodes (see below for a precise definition). We refer to~\cite{O, O2} for recent refinements of the Eisenbud-Harris theory, and to~\cite{E98, EM} for generalizations of this theory to certain curves not of compact type.

\subsection{Limit linear series for curves of compact type}\label{sec:limitcompact} 
We recall the definition of limit linear series for curves of compact type, following~\cite{EH86}. 

\medskip

We first need to recall the sequence of orders of vanishing of a $\g^r_d$ $L=(\L,H)$ at a point on a smooth projective curve $Y$ over $\kappa$.
For any closed point $p$ on $Y$ and any section $f \in H$, denote by $\ord_p(f)$ the order of vanishing of $f$ at $p$.  The orders of vanishing at $p$ of all the sections of $\L$ in $H$ define a sequence of integers $0 \leq a_0^L( p ) <  a_{1}^L( p ) <  \dots <a_{r-1}^L( p ) < a_r^L ( p )$.  (This sequence is obtained by induction: let $f_r$ be the section in $H$ with the highest order of vanishing, and define $f_{r-i}$ inductively as a section in $H$ with the highest possible order of vanishing at $p$ among
all sections which are linearly independent of the first $i$ sections $f_r,\dots, f_{r-i+1}$. Then $a_i( p ) = \ord_p(f_i)$.)
 
 \medskip

Let $X_0$ be a strongly semistable curve over $\k$, let $G=(V,E)$ be the dual graph of $X_0$, and for any $v\in V$, let $X_v$ be the corresponding irreducible component of $X_0$.  When $G$ is a tree, the curve $X_0$ is called of {\it compact type}: in this case, $\Pic^0(X_0)$ is compact. 

\noindent (Recall that in general, we have an exact sequence of algebraic groups 
$$0 \rightarrow H^1(G, \mathbb \ZZ) \otimes {\mathbb G}_m \rightarrow \Pic^0(X_0) \rightarrow \prod_{v \in V_0} \Pic^0(X_v) \rightarrow 0.)$$

\medskip

Let $X_0$ be a curve of compact type. A {\it crude limit} $\g^r_d$ $L$ over $X_0$ is by definition the data of a $\g^r_d$ $L_v = (\mathcal L_v, H_v)$ over $X_v$ for each vertex $v \in V$, called the $X_v$-{\it aspect} of $L$,  such that the following property holds: If the two components $X_u$ and $X_v$ of $X_0$ meet at a node $p$ corresponding to 
an edge $\{u, v\}$ in $E$, then 
\begin{equation}\label{eq:crude}
 a^{L_v}_{i}( p ) + a^{L_u}_{r-i}( p ) \, \geq \, d
 \end{equation}
for all $ 0\leq i \leq r$.
A crude limit linear series is a {\it refined limit linear series} if all the inequalities in \eqref{eq:crude}
are equalities.  

\medskip

We are now going to associate to $L$ a degree $d$ divisor on the regularization of $X_0$ (c.f. Section~\ref{sec:mcnodal} for the definition).
 Roughly speaking, this amounts to choosing a divisor of degree $d$ on each $X_v$ 
 in the divisor class of $L_v$. 
 However, the global divisor on $\C X_0$ defined by such a collection of divisors does not have degree $d$.  
 To fix this problem, we need to choose a root $\mathfrak r$ for the tree $G$ 
 and modify the local divisors using the root. This is done as follows.

Let $\Gamma$ be the metric graph associated to $G$ and let $\C X_0$ be the regularization of $X_0$.  
For each vertex $v$ of $G$, let $\overline D_v$ be a divisor of degree $d$ on $X_v$ in the divisor class defining $\mathcal L_v$ (i.e., $\mathcal L(\overline D_v) \simeq \mathcal L_v$). The subspace $H_v\subseteq H^0(X_v, \mathcal L(\overline D_v))$ can be naturally regarded as a subspace of $\k(X_v)$.
  Let $\mathfrak r$  be a fixed vertex of $G$ and let $G_\mathfrak r$ denote $G$ considered as a rooted tree 
 with root $\mathfrak r$. 
 For any vertex $u \neq \mathfrak r$ in $G$, consider the unique path from $u$ to $\mathfrak r$ in $G$ and let $e_u$ be the edge adjacent to $u$ along this path. Let $x_u$ be the node of $X_u$ corresponding to the edge $e_u$. Consider the divisor $\mathcal D$ of degree $d$ on $\C X_0$ defined by 
 
\begin{equation}\label{eq:D}
D_\Gamma := \overline D_\mathfrak r + \sum_{u \neq \mathfrak r} D_u,  \; \textrm{where} \; D_u := \overline D_u - d( x_u ). 
\end{equation}
 
We will see below that $\mathcal D$ has rank at least $r$ in $\C X_0$.
In order to derive a more precise characterization of crude limits in terms of a suitable rank function, 
we need to introduce a refined notion of rank on a metrized complex which 
takes into account the spaces $H_v$.
 
\subsubsection{A refined notion of rank for divisors on a metrized complex}
Let $\C$ be a metrized complex of algebraic curves, $\Gamma$ the underlying metric graph with model $G$, and $\{ C_v \}$ the corresponding family of smooth projective curves over $\k$.  Suppose we are given, for each $v \in V$, a non-empty $\k$-linear subspace $F_v$ of  $\k(C_v)$.  We denote by $\mathcal F$ the collection of all $F_v$.  
Define  the {\em $\mathcal F$-rank} of a divisor $\mathcal D$ on $\C$, denoted $r_{\C, \mathcal F}(\mathcal D)$, to be the maximum integer $r$ such that 
for any effective divisor $\mathcal E$ of degree $r$, there exists a nonzero rational function $\f$ on $\C$, with $C_v$-parts $f_v \in F_v$ for all $v \in V$, such that $\mathcal D + \div(\f) - \mathcal E \geq 0$.

\medskip

The following proposition provides an upper bound for $r_{\C,\mathcal F}$ in terms of the maximal dimension of $F_v$
for $v\in V(G)$.

\begin{prop}\label{prop:basic} Let $s = \max\left( 0, \min_{v\in V(G)} (\dim_\k (F_v)-1) \right).$  
Then for any divisor $\mathcal D$ on $\C$, we have $r_{{\C},\mathcal F}(\mathcal D) \leq s$. 
\end{prop}
\begin{proof}
If $\dim_\k (F_v) =0$ for some $v$, then either $r_{\C,\mathcal F}(\mathcal D) = 0$ (if $\mathcal D \geq 0$) or $r_{\C,\mathcal F}(\mathcal D) = -1$ (otherwise), and the proposition holds. For the sake of contradiction, suppose now that $r_{\C,\mathcal F}(\mathcal D) \geq s+1 \geq 1$. Let $v$ be a vertex with $s = \dim_\k (F_v)-1$. For any effective divisor $E_v$ of degree $s+1$ on $C_v$, there exists a rational function $\f$ on $\C$ with $f_v \in F_v$ such that $\mathcal D - E_v+ \div(\f) \geq 0$.  Taking $\Gamma$-parts shows that $D_\Gamma -(s+1)(v) + \div(f_\Gamma) \geq 0$. 
For a generic choice of $E_v$, we can assume 
as in the proof of Proposition~\ref{prop:combinrank}
that the slopes of $f_\Gamma$ on the edges adjacent to $v$ are fixed. 
For the divisor $D'_v=D_v +\div_v(f_\Gamma)$, this shows that for any choice of $E_v$ 
 on $C_v$, there exists an element $f_v \in F_v$ with $D'_v -E_v +\div(f_v) \geq 0$. This contradicts 
the assumption that $\dim_\k(F_v) = s+1$, since $F_v$ can define a linear system of projective dimension at most $s$ on $C_v$.
\end{proof}

In Section~\ref{sec:rank-csum},
we provided a formula for the rank of a divisor on a metrized complex $\C$ which is a connected sum of two 
metrized complexes $\C_1$ and $\C_2$. We now state 
a straightforward generalization of Proposition~\ref{prop:ConnectedSumRankForMetrizedComplexes} 
to restricted ranks, which will be used in the next section.

Suppose that the underlying graph of $\C$  has a bridge edge $e=\{v_1,v_2\}$ and let
$\C_1$ and $\C_2$ be the two metrized complexes obtained by removing the edge $e$. 
Let $\Gamma_1$, 
$\Gamma_2$, and $\Gamma$ be the underlying metric graphs of $\C_1$, $\C_2$, and $\C$ respectively (with $v_i$ a vertex of $\Gamma_i$).
Denote by $x_1$ and $x_2$ the points of $\C_1$ and $\C_2$, respectively, corresponding to the edge $e$.  
Let $\mathcal F$ be a family of spaces of rational functions $F_v$ for vertices $v$ of $\Gamma$, 
and denote by $\mathcal F_1$ (resp. $\mathcal F_2$) the collection of those $F_v$ for $v$ 
a vertex of $\Gamma_1$ (resp. $\Gamma_2$).

\begin{prop}\label{prop:ConnectedSumRankForMetrizedComplexes2}
Let $\mathcal D = \mathcal D_1 + \mathcal D_2$ be the sum of
divisors $\mathcal D_1$ and $\mathcal D_2$ on $\C_1$ and $\C_2$, respectively. 
For a non-negative integer $k$, define $\eta(k)$ to be the smallest integer $n$ such that 
$r_{\C_1,\mathcal F_1}\bigl(\mathcal D_1 + n(x_1)\bigr) = k$.  Then 
\begin{equation} 
r_{\C,\mathcal F}(\mathcal D) = \min_{k\geq 0} \bigl\{\,k + r_{\C_2,\mathcal F_2}(\mathcal D_2-\eta(k)(x_2))\,\bigr\}. 
\end{equation} 
\end{prop}

\subsubsection{Characterization of limit linear series in terms of the refined rank function} We retain the terminology from the previous sections. Let $X_0$ be a curve of compact type, $G=(V,E)$ the dual graph, $\C X_0$ the regularization of $X_0$, and let $L$ be a collection of $\g^r_d$'s $L_v = (\mathcal L_v,H_v)$ on $X_v$, one for each $v \in V$. Fix a root vertex $\mathfrak r \in V$ and divisors $\overline D_v \in \Div(X_v)$ such that $\mathcal L(\overline D_v) \sim \mathcal L_v$, and let $\mathcal D$ be the divisor defined by~\eqref{eq:D}. Denote by $\mathcal H$ the family of all $H_v \subseteq H^0(X_v, \mathcal L(D_v)) \subset \k(X_v)$.

\begin{thm}\label{thm:limitseries} The following two assertions are equivalent:
\begin{itemize}
\item[(i)] $L$ is a crude limit $\g^r_d$ on $X_0$.
\item[(ii)] $r_{\C X_0, \mathcal H} (\mathcal D) = r$. 
\end{itemize}
\end{thm}

\begin{proof}
 Denote by $\mathfrak r_1, \dots, \mathfrak r_h$ all the children of $\mathfrak r$ in the rooted tree $G_\mathfrak r$. 
 For each $1\leq i\leq h$, denote by $x_i = x_{\mathfrak r_i}$ and $y_i$ the $\k$-points of 
$X_{\mathfrak r_i}$ and $X_\mathfrak r$, respectively, 
which correspond to the edge 
$\{\mathfrak r,\mathfrak r_i\}$.

\medskip

Proof of ${\rm (i)} \Rightarrow {\rm (ii)}$: 
Let $L$ be a crude limit $\g^r_d$ and let $s \leq r$ be a non-negative integer. 
Let $F_\mathfrak r \subset H_\mathfrak r$ be a subspace of dimension $s+1$, and denote by 
$\mathcal F$ the collection of $F_\mathfrak r$ and $H_v$ for $v\neq \mathfrak r$ in $V$. We will prove
that $r_{\C X_0, \mathcal F}(\mathcal D) = s$. The result then follows by taking $\mathcal F = \mathcal H$.

By Proposition~\ref{prop:basic}, we have $r_{\C X_0, \mathcal F} (\mathcal D)\leq s$, 
so it will be enough to prove the reverse inequality $r_{\C X_0, \mathcal F} (\mathcal D) \geq s$. 
The proof proceeds by induction on $|V|$, and by applying Proposition~\ref{prop:ConnectedSumRankForMetrizedComplexes2}. 
For the base case of our induction, we have $|V|=1$, and thus $\C X_0 = X_\mathfrak r$ and $r_{X_\mathfrak r, 
\{F_\mathfrak r\}}(D_\mathfrak r) =s$. 

Now suppose that the claim holds for all metrized complexes on at most $|V|-1$ vertices. For each $1\leq i\leq h$, denote by 
$G_i$ the rooted tree at $\mathfrak r_i$ obtained by removing the edge $\{\mathfrak r, \mathfrak r_i\}$ from $G$,
and let $\Gamma_i$ be the metric graph obtained from $G_i$ by taking all edge lengths to be $1$. 
Let $\C_i$ be the sub-metrized complex of 
$\C X_0$ with underlying metric graph $\Gamma_i$ and curves 
$X_v$ for $v$ a vertex of $G_i$, and denote by $\mathcal D_i$ the restriction of $\mathcal D$ to $\C_i$, i.e., 
$D_{i,v} = D_v$ for all vertices $v$ of $G_i$. Denote by $\mathcal H_i$ the family of all $H_v$ for $v$ a vertex of $G_i$. 
Let $\eta_i$ be the $\eta$-function associated to the bridge edge $\{\mathfrak r, \mathfrak r_i\}$, so that
 $\eta_i(k)$ is the smallest integer $n$ such that 
$r_{\C_i,\mathcal H_i}(\mathcal D_i + n(x_{i})) = k$.
Repeatedly applying Proposition~\ref{prop:ConnectedSumRankForMetrizedComplexes2} reduces the calculation of $r_{\C X_0, \mathcal F}(\mathcal D)$ 
to the calculation of the restricted ranks of certain divisors 
on the metrized complex $X_\mathfrak r$ (with a single vertex $\mathfrak r$ and a unique curve $X_\mathfrak r$) 
with respect to the subspace $F_\mathfrak r \subset \k(X_\mathfrak r)$; more precisely,  
\[r_{\C X_0, \mathcal F} (\mathcal D) = 
\min_{s_1,\dots, s_h \geq 0} \bigl\{\,s_1+\dots+s_h + r_{X_\mathfrak r, \{F_\mathfrak r\}}
\bigl(D_\mathfrak r - \sum_{i=1}^h \, \eta_i(s_i) (y_{i})\bigr)\,\bigr\}.\]
We are thus led to prove that each term $s_1+\dots+s_h + r_{X_\mathfrak r, \{F_\mathfrak r\}}
\bigl(D_\mathfrak r - \sum_{i=1}^h \, \eta_i(s_i) (y_{i})\bigr)$ in the above formula, for non-negative integers
$s_i$, is bounded below by $s$. 
We may clearly assume that $s_1+\dots+s_h \leq s+1$. 
\begin{claim}
 For each $1\leq i\leq h$, $\eta_i(s_i) \leq d- a_{r-s_i}^{L_{\mathfrak r_i}}(x_{i})$.
\end{claim}
\noindent{\it Proof of the claim.}
By definition, the subspace $F_{\mathfrak r_i}$ of $H_{\mathfrak r_i}$ consisting of all global sections of 
$\mathcal L_{\mathfrak r_i}$ having an order of vanishing at least 
$a_{r-s_i}^{L_{\mathfrak r_i}}(x_{i})$ at $x_{i}$ has dimension $s_{i}+1$. 
Let $\mathcal D'_i$ be the divisor defined  on $\C_i$ 
by the data of the linear series $L_v=(\mathcal L(\overline D_v), H_v)$, for $v$ a vertex of $G_i$, 
and the root $\mathfrak r_i$, as in~\eqref{eq:D}. Note that 
$\mathcal D'_i = \mathcal D_i + d(x_i)$. 
Let $\mathcal F_i$ be the collection of 
$F_{\mathfrak r_i}$ and $H_v$, for $v\neq \mathfrak r_i$ a vertex of $G_i$. 
By the induction hypothesis applied to the subcurve of $X_0$ consisting of all curves $X_v$ for $v$ a vertex of $G_i$ and the crude limit linear series 
$\{L_v\}_{v\in V(G_i)}$, we have $r_{\C_i, \mathcal F_i}(\mathcal D'_i) = s_i$.
On the other hand, by the definition of $F_{\mathfrak r_i}$, we have 
\[
r_{\C_i, \mathcal F_i}(\mathcal D'_i) = r_{\C_i,\mathcal F_i}(\mathcal D'_i - a_{r-s_i}^{L_{\mathfrak r_i}}(x_{i})(x_{i})) 
 \]
and thus
\begin{align*}
 r_{\C_i,\mathcal F_i}(\mathcal D_i + (d- a_{r-s_i}^{L_{\mathfrak r_i}}(x_i))
 (x_{i})) =r_{\C_i,\mathcal F_i}(\mathcal D'_i - a_{r-s_i}^{L_{\mathfrak r_i}}(x_{i})
(x_{i}))  = s_i.
\end{align*}
By the definition of $\eta_i(s_i)$, we infer that $\eta_i(s_i) \leq d- a_{r-s_i}^{L_{\mathfrak r_i}}(x_{i})$, 
which proves the claim.

\medskip

To finish the proof of ${\rm (i)} \Rightarrow {\rm (ii)}$, note that since $L$ is a crude limit $\g^r_d$, we have 
\[a_{s_i}^{L_\mathfrak r}(y_{i}) \geq d- a_{r-s_i}^{L_{\mathfrak r_i}}(x_{i}).\] 
This shows that the space of all global sections of $\mathcal L_{\mathfrak r}$ with an order of vanishing at least 
$\eta(s_i)$ at $y_{i} $ has codimension at most $s_i$. 
The intersection of all these spaces with $F_\mathfrak r$ has dimension at least  
$s+1 - \sum_{i=1}^h s_i$, which shows that $r_{X_\mathfrak r, \{F_\mathfrak r\}}
\bigl(D_\mathfrak r - \sum_{i=1}^h \, \eta_i(s_i) (y_{i})\bigr) \geq s - \sum_{i=1}^h s_i$. 

\medskip

\noindent Proof of ${\rm (ii)} \Rightarrow {\rm (i)}$: 
By Corollary~\ref{cor:f-widthZ}, in calculating $r_{\C X_0, \mathcal H}(\mathcal D)$ 
we can restrict to effective divisors in $\Div(\C X_0)_\mathbb Z$. 
In this case, since $G$ is a tree, we can suppose that all the rational functions on $\Gamma$ which appear 
in the calculation of $r_{\C X_0, \mathcal H}$ 
arise via linear interpolation from integral functions 
$f_G: V \rightarrow \mathbb Z$. Since $G$ is a tree, the data of an integral function $f_G$ on $V$ is equivalent, 
up to an additive constant, to an assignment of labels $a_{uv}\in \mathbb Z$ 
to each oriented edge $uv$ of $G$ such that $a_{uv} = -a_{vu}$ for any $\{u,v\} \in E$. (Set $a_{uv} = f_G(u) -f_G(v)$.)  
In what follows, we denote by $p_v$ the unique parent of a vertex $v \neq \mathfrak r$ in the rooted tree 
$G_\mathfrak r$.

Suppose now that $r_{\C X_0,\mathcal H}(\mathcal D) = r$. 
We first observe that, for any other choice of a root vertex $\mathfrak r'$, 
if $\mathcal D'$ denotes the corresponding divisor on $\C X_0$ defined as in~\eqref{eq:D} then
the two divisors $\mathcal D$ and $\mathcal D'$ differ by the principal divisor $\div(\f)$, where 
$\f$ is a rational function with $\Gamma$-part $f_G$ and constant $X_v$-part for all $v \in V$.
Thus $r_{\C X_0,\mathcal H}(\mathcal D') = r$, so
it will be enough to prove the validity of Condition~\eqref{eq:crude} for the root $\mathfrak r$ and a child $u$
of $\mathfrak r$ among $\mathfrak r_1,\dots, \mathfrak r_h$. 

Denote by $x_u$ and $y_u$ the $\k$-points of $X_u$ and 
$X_\mathfrak r$, respectively, corresponding to the edge $\{\mathfrak r, u\}$, and
let $ 0 \leq i\leq d$. For any effective divisor $\mathcal E = E_\mathfrak r +E_u$, where $E_\mathfrak r$ and $E_u$ are effective divisors of degree $r-i$ and $i$ 
on $X_\mathfrak r$ and $X_u$, respectively, there must exist a rational function 
$\f$ with $f_w\in H_w$ for all $w$ 
such that $\mathcal D - \mathcal E + \div(\f) \geq 0$. 
For any two adjacent vertices $w$ and $z$ in $G$, let $a_{wz} = f(w) - f(z)$ as above. 
Since for any vertex $w\neq \mathfrak r$, $D_w$ has degree zero, 
a simple induction starting from the leaves and going towards the root shows that $a_{wp_w} \geq 0$.
 This  in particular implies that $D_\mathfrak r + \div(f_\mathfrak r) - E_\mathfrak r - a_{u\mathfrak r}(y_u) \geq 0$
 and $D_u + \div(f_u) - E_u - (d-a_{u\mathfrak r})(x_u) \geq 0$. 
   For a generic choice of $E_\mathfrak r $ and $E_u$, we can assume as in the proof of Proposition~\ref{prop:combinrank}
   that $a_{u\mathfrak r} = a $ is a constant. In other words, 
   there exists an integer $0\leq a\leq d$ such that the sublinear 
   system of $H_\mathfrak r$ (resp., $H_u$) 
   which consists of those sections with an order of vanishing at least $a$ (resp. $d- a$) at $y_u$ (resp. $x_u$) 
   has projective dimension at least $r-i$ (resp. $i$). By the definition of the sequences $a^{L_\mathfrak r}(\cdot)$ 
   and $a^{L_u}(\cdot)$, this simply means that 
   $a^{L_\mathfrak r}_i(y_u) \geq a$ and  $a^{L_u}_{r-i}(x_u) \geq d-a$.  It follows that
   $a^{L_\mathfrak r}_{i}( y_u ) + a^{L_u}_{r-i}( x_u ) \, \geq \, d,$ which is Condition~\eqref{eq:crude} 
   for the two irreducible components $X_\mathfrak r$ and $X_u$ of $X_0$.  
\end{proof}

\begin{cor} Let $\k$ be of characteristic zero. Then $L$ is a refined limit series iff $r_{\C, \mathcal H}(\overline D) =r$ and all the ramification points of $L_v$ are smooth points of  $X_v$ for all $v\in V(G)$. 
\end{cor}
\begin{proof}
This is a direct consequence of Theorem~\ref{thm:limitseries} above and the Pl\"ucker formula, c.f. \cite[Proposition 1.1]{EH86}.
\end{proof}

We note that the {\em a priori} dependence of the family $\mathcal H$ on the $X_v$-aspects $L_v$ can be removed, and recovered from the condition on the rank, in the following sense. For any crude  limit $\g^r_d$ $L$ on $X_0$ with $X_v$-aspect $L_v = (\mathcal L_v, H_v)$, choose a divisor $D_v \in \Div^d(X_v)$ with $\mathcal L_v \simeq \L(D_v)$. Two divisors $\D$ and $\D'$ in $\Div(\C X_0)_\mathbb Z$ are called {\em combinatorially equivalent} if they differ by the divisor of a rational function $\f$ on $\C X_0$ with all $f_v$ constant.

The following result is essentially a reformulation of Theorem~\ref{thm:introLLS} from the Introduction and easily implies the result stated there.

\begin{thm}
Let  $X_0$ be a curve of compact type and $\C X_0$ the regularization of $X_0$. Then there is a bijective correspondence between the following:

\begin{itemize}
\item Pairs $(L,\{D_v\})$ consisting of a crude limit $\g^r_d$ $L$ on $X_0$ and a collection of divisors $D_v$ on $X_v$ with $L_v\simeq \mathcal L(D_v)$ for any irreducible component $X_v$ of $X_0$; and
\item Pairs $(\H,[\D])$, where $\H = \{ H_v \}$, $H_v$ is an $(r+1)$-dimensional subspace of $\kappa(X_v)$ for each $v \in V$, and $[\D]$ is the combinatorial linear equivalence class of a divisor $\D \in \Div(\C X_0)_\mathbb Z$ of degree $d$ on $\C X_0$ with 
$r_{\C X_0,\H}(\D) = r$.
\end{itemize}

In particular, the data of a crude limit $\g^r_d$ on $X_0$ is equivalent to the data of a pair $(\mathcal H, [\mathcal D])$ with $r_{\C X_0, \mathcal H}(\mathcal D) = r$.
\end{thm}

\begin{proof}
We already know by Theorem~\ref{thm:limitseries} that any crude limit $\g^r_d$ gives a pair $(\mathcal H, [\mathcal D])$ with $r_{\C X_0, \mathcal H}(\mathcal D) = r$.

\medskip

Let $\mathcal H$ be a collection  of $(r+1)$-dimensional subspaces $H_v \subset \k(X_v)$ for $v \in V$,
and let  $\mathcal D$ be a divisor of degree $d$ in $\Div(\C X_0)_\mathbb Z$ with $r_{\C X_0, \mathcal H}(\mathcal D) = r.$ We show the existence of a $\g_d^r$ $\mathcal L_v = (L_v, H_v )$ on $X_v$ for each $v \in V$ such that the collection $L = \{ L_v \}$ defines a crude limit $\g^r_d$ on $X_0$, and such that $\mathcal D$ is (combinatorially) linearly equivalent to the divisor $\mathcal D^L$ associated to $L$ by~\eqref{eq:D}.  

\medskip

Let $D_\Gamma$ (resp. $D_v$) denote the $\Gamma$-part (resp. $X_v$-part) of $\D$.
Since $\Gamma$ has genus zero, for any vertex $v$ in $V(G)$ there exists a rational function $f_{\Gamma}$ on $\Gamma$ such that $D_\Gamma +\div(f_{\Gamma}) = d(v)$. Define $\widetilde D_v := D_v + \div_v(f_{\Gamma})$, where $\div_v(f_\Gamma)$ is defined as in
(\ref{eq:divf2}). 
We claim that $H_v \subset H^0(X_v, \mathcal L_v)$; this amounts to showing that $\div(f) + \widetilde D_v \geq 0$ for any $f \in H_v$.  If we denote by $H'_v$ the subspace of $H_v$ consisting of all rational functions in $H_v$ with this property, it will be enough show that $\dim_\k (H'_v) = r+1$.

\medskip

First, we note that for any rational function $f'_\Gamma$ on $\Gamma$ with $D_\Gamma +\div(f'_\Gamma) \geq 0$, we have $\div_v(f_{\Gamma}) - \div_v(f'_\Gamma) \geq0$. Indeed,  $D_\Gamma +\div(f'_\Gamma) \geq 0$ implies that for any edge $e$ adjacent to $v$, the slope of $f_{\Gamma}$ in the tangent direction corresponding to $e$ is  bounded above by the sum of the coefficients $D_\Gamma(w)$ of the points $w$ lying on the connected component of $\Gamma \setminus \{v\}$ which contains the edge $e$. This sum is precisely the slope of $f_{\Gamma}$ in the tangent direction corresponding to $e$.

\medskip

For any effective divisor $E$ of degree $r$ on $X_v$, let  $\mathcal E$ denote the corresponding divisor on $\C X_0$. Since $r_{\C X_0, \mathcal H}(\mathcal D) = r$, there exists a rational function $\f^E$ such that $\div(\f^E) + \mathcal D - \mathcal E \geq 0$.
Since $\div(f^{E}_\Gamma) + D_\Gamma \geq 0$, the preceding remark implies that $\widetilde D_v =  D_v+ \div_v(f_{\Gamma}) \geq D_v + \div_v(f^E_\Gamma)$. By the definition of $\div(\f)$,
we conclude that $\div(f_v^E) + \widetilde D_v -E \geq 0$, and in particular $f_v^E$ belongs to $H'_v$. Since this holds for any $E$, we infer that the linear series on $X_v$ defined by $H'_v$ has (projective) dimension at least $r$ and so  $H'_v$ has dimension at least $r+1$, thus $ H'_v= H_v $. We conclude that $L_v = (\mathcal L_v, H_v)$ is a $\g^r_d$ on $X_v$, and thus the collection of $L_v$
for $v \in V$ defines a crude limit $\g^r_d$ $L$ on $X_0$.
If $\mathcal D^L$ is the divisor associated to $L$ by~\eqref{eq:D}, it is straightforward to check that $\mathcal D^L$ and  $\mathcal D$ differ by the divisor of a rational function $\f$ on $\C X_0$
whose $X_v$-parts are all constant and whose $\Gamma$-part is
the linear interpolation of an integer-valued function $f: V \rightarrow \mathbb Z$. 
\end{proof}

\subsection{Limit linear series for metrized complexes.}

Let $\C$ be a metrized complex of algebraic curves over $\k$, $\Gamma$ the underlying metric graph with model $G=(V,E)$, and $\{ C_v \}$ the corresponding collection of smooth projective curves over $\k$.

\medskip

We define a {\em (crude) limit $\g^r_d$} on $\C$ to be an equivalence class of pairs $(\mathcal H, \mathcal D)$ consisting of a  divisor $\mathcal D$ of degree $d$ on $\C$ and a collection $\mathcal H$ of $(r+1)$-dimensional subspaces $H_v \subset \k(C_v)$, for $v \in V(G)$, such that $r_{\C, \mathcal H}(\mathcal D) = r$.  Two pairs $(\H,\D)$ and $(\H', \D')$ are considered equivalent if there is a rational function $\f$
on $\C$ such that $D' = D + \div(\f)$ and $H_v = H'_v \cdot f_v$ for all $v \in V$, where $f_v$ denotes the $C_v$-part of $\f$.

\begin{remark}
Since $r_{\C, \H}(\D) \leq r_\C (\D)$ for all $\mathcal H$ as above and all $\D \in \Div(\C)$, it follows immediately from
Theorem~\ref{thm:RR-metrizedcomplexes} that for any limit $\g^r_d$ on $\C$ we have if $d\geq 2g$ or $r\geq g$, then $r+g \leq d$  (Riemann's inequality), and if $\D$ is {\em special} then $r \leq d/2$ (Clifford's inequality).  For curves of compact type, these inequalities are established (by a completely different proof) in \cite{EH86}.
\end{remark}

\begin{thm}\label{thm:limitseriesgeneral}
Let $X$ be a smooth proper curve over $\K$, $\fX$ a strongly semistable model for $X$, 
and $\C \fX$ the metrized complex associated to $\fX$. 
Let $D$ be a divisor on $X$ and let $L_\eta = (\mathcal L(D), H_\eta)$, 
for $H_\eta \subseteq H^0(X, \mathcal L(D)) \subset \K(X)$, be a $\g^r_d$ on $X$. 
For any vertex $v \in V$, define $H_v$ as the $\k$-vector space defined by the 
reduction to $\k(C_v)$ of all the rational functions in $H_\eta$ (c.f. Section~\ref{section:ReductionOfRationalFunctions}), and let $\mathcal H = \{H_v\}_{v\in V}$. 
Then the pair $(\tau^{\C \fX}_*(D), \mathcal H)$ defines a limit $\g^r_d$ on $\C \fX$. 
\end{thm}

\begin{proof} 
By Lemma~\ref{lem:dimensionreduction}, all the subspaces $H_v$ have dimension $r+1$ over $\k$.
By Proposition~\ref{prop:basic}, we only have to show that $r_{\C\fX,\mathcal H} \geq r$.  
By Theorem~\ref{thm:f-width},
it suffices to show that for any effective divisor $\mathcal E \in \Div(\C \fX)$ of degree $r$ supported on a subset $\R = \bigcup_v \R_v$ of $\C\fX$ with $\R_v \subset C_v(\k)\setminus \mathcal A_v$ of size $g_v+1$, 
there exists a rational function $\f$ on $\C \fX$ such that  $f_v\in H_v$ for all $v \in V$ and $\div(\f) + \tau^{\C\fX}_*(D) - \mathcal E \geq 0$. For any such $\mathcal E$, there exists an effective divisor of degree $r$ on $X$ such 
that $\tau^{\C\fX}_*(E) = \mathcal E$. Since $L_\eta$ is a $\g^r_d$ on $X$, there exists a rational function $f \in H_\eta$ such that $D -E +\div(f) \geq 0$. Let $\f$ be the corresponding rational function on $\C\fX$ (as in the paragraph preceding 
Theorem~\ref{thm:PLspecialization}).
  We conclude that $\tau^{\C\fX}_*(D-E+\div(f)) = \tau^{\C\fX}_*(D) -\mathcal E +\tau^{\C\fX}_*(\div(f)) = \tau^{\C\fX}_*(D) -\mathcal E +\div(\f) \geq 0$. Since $f_v \in H_v$, this shows that $r_{\C\fX,\mathcal 
 H} \geq r$ as desired.
\end{proof}

\medskip

For the regularization $\C X_0$ of a strongly semistable curve $X_0$ over $\k$, Theorem~\ref{thm:limitseriesgeneral} can be reformulated as follows.
Let $\phi: \mathcal X\rightarrow \Spec R$ be a regular smoothing of $X_0$ over  a discrete valuation ring $R$ and let $L_\eta=(\mathcal L_\eta,H_\eta)$ be a $\g^r_d$ over the generic fiber $\mathcal X_\eta$ of $\phi$. Consider an extension $\mathcal L$ of $L_\eta$ to $\mathcal X$, and let $\mathcal L_0$ be the restriction of $\mathcal L$ to $X_0$. Let $\mathcal D$ be a divisor on $\C X_0$ in the linear equivalence class of $\pi(\mathcal L_0)$, where $\pi:\Pic(X_0) \rightarrow \oplus_{v\in V(G)} \Pic(X_v)$ is as in Section~\ref{sec:basics}. There exists a family $\mathcal H=\{H_v\}_{v\in V(G)}$ of $(r+1)$-dimensional subspaces $H_v \subset \k(X_v)$  such that $r_{\C X_0, \mathcal H}(\mathcal D)=r$. In addition, $H_v$ can be explicitly constructed as follows. For any vertex $v$, choose a rational function $\f_{v}$ on $\C X_0$ such that 
the divisor $\mathcal D + \div(f_{G,v})$ is $v$-reduced and let $f_{G,v}: V(G) \rightarrow \mathbb Z$ be the $\Gamma$-part of $\f_v$.
Consider the extension $\mathcal L\,\bigl(\,\sum_v f_{G,v}(u) X_u\,\bigr) $ of $\mathcal L_\eta$ to $\mathcal X$ and let $H_{f_{G,v}}$ be the closure of $H_\eta$ in $H^0\bigl(\,\mathcal X, \mathcal L\,\bigl(\,\sum_v f_{G,v}(u) X_u\,\bigr)\,\bigr)$. Define $H_v:=H_{f_{G,v}}|_{X_v}$ and $\tau_*(L_\eta):=(\D, \mathcal H)$. 
Finally, for any $v \in V$ let $\mathcal D^v$ be the $v$-reduced divisor linearly equivalent to $\mathcal D$, and define $d_v:= \mathcal D_\Gamma(v) = \deg(D^v_v)$.

\begin{thm} 
\label{thm:regularlimitseries}
Let $X_0$ be a strongly semistable curve over $\k$ and $\C X_0$ the regularization of $X_0$.
For any regular smoothing $\pi: X\rightarrow \Spec R$ of $X_0$ over a discrete valuation ring $R$, $\tau_*(L_\eta)$ as 
defined above is a limit $\g^r_d$ on $X_0$. In addition, for each vertex $v\in V(G)$ we have $H_v \subseteq H^0(X_v, \mathcal L(D^v_v))$, so that $(\mathcal L(D^v_v), H_v)$ defines a $\g^r_{d_v}$ on the irreducible component $X_v$ of $X_0$. 
\end{thm}

We conclude this section with the following strengthening of Theorem~\ref{thm:limitseriesgeneral} which will be used in \S\ref{ChabautySection}.

\begin{thm}\label{thm:limitseriesgeneral2}
Let $X$ be a smooth proper curve over $\K$, $\fX$ a strongly semistable model for $X$, 
and let $\C \fX$ be the metrized complex associated to $\fX$. 
Let $D$ be a divisor on $X$ and set $\D = \tau^{\C \fX}_*(D)$.
Let $L_\eta = (\mathcal L(D), H_\eta)$ be a $\g^r_d$ on $X$
and let $\mathcal H = \{H_v\}_{v\in V}$ be as in the statement of Theorem~\ref{thm:limitseriesgeneral}.
Let $E^\circ$ be an effective divisor of degree $e$ supported on the smooth locus of 
the special fiber of $\fX$, and define $\E^\circ = \sum_{v \in V} E^\circ_v \in \Div(\C\fX)$,
where $E^\circ_v$ is the restriction of $E^\circ$ to $C_v$.
Suppose that $H_v \subseteq L(D_v - E^\circ_v)$ for all $v$.
Then the pair $(\D - \E^\circ, \mathcal H)$ is a limit $\g^r_{d-e}$ on $\C \fX$. 
\end{thm}

\begin{proof}
In the view of Theorem~\ref{thm:f-width}, the proof is similar to the proof of 
Theorem~\ref{thm:limitseriesgeneral}, by requiring each subset 
$\R_v \subset C_v(\k) \setminus \mathcal A_v$ to be disjoint from the support of $E^\circ_v$. 
For any effective divisor $\E$ of degree $r$ supported on $\R = \bigcup \R_v$, 
the proof of Theorem~\ref{thm:limitseriesgeneral} shows that there exists a rational function 
$\f$ such that $\mathcal D -\E +\div(\f) \geq 0$. 
We claim that in fact $\D - \E +\div(\f) - \E^\circ \geq 0$, which implies that 
$(\D - \E^\circ, \H)$ is a limit $\g^r_{d-e}$ on $\C\fX$ as desired.
To see this, let $E^\circ_\Gamma$ be the $\Gamma$-part of $\mathcal E^\circ$ and 
note that since $E^\circ_\Gamma = \sum_{v \in V} \deg(E^\circ_v) (v)$ and $\div(f_\Gamma) \geq E_\Gamma - D_\Gamma$
by hypothesis, it suffices to prove that the $C_v$-part of $\div(\f)$ is at least the $C_v$-part of 
$\E + \E^\circ - \D$ for all $v \in V$, i.e.,
\begin{equation}
\label{eq:limitseriesgeneral2}
\div(\f)_v := \div(f_v) + \div_v(f_\Gamma) \geq E_v + E^\circ_v - D_v.
\end{equation}

If $z \in \mathrm{supp}(E^\circ_v)$, then since $\mathrm{supp}(E^\circ_v)$ is disjoint from $\mathcal A_v \supset \mathrm{supp}(\div_v(f_\Gamma))$ we have 
$\div(\f)_v(z) = \div(f_v)(z)$, and since $\mathrm{supp}(E^\circ_v)$ is disjoint from $\R_v \supset \mathrm{supp}(E_v)$ we have $E_v(z) = 0$.
Thus (\ref{eq:limitseriesgeneral2}) holds in this case since $H_v \subset L(D_v -E^\circ_v)$ implies that 
$\div(f_v)(z) \geq E^\circ_v(z) - D_v(z)$.
On the other hand, if $z \not\in \mathrm{supp}(E^\circ_v)$, then $E^\circ_v(z)=0$ and thus (\ref{eq:limitseriesgeneral2}) holds because by assumption we have
$\div(\f) \geq \E - \mathcal D$ and in particular $\div(\f)_v(z) \geq E_v(z) - D_v(z)$.

Thus (\ref{eq:limitseriesgeneral2}) is valid for all $z \in C_v(\k)$ as desired.
\end{proof}

\subsection{Remarks on completion and smoothing}

Given a divisor $\D$ of degree $d$ and rank $r$ on $\C$, one may ask if there exists a collection $\H = \{ H_i \}$, with $H_i$ an $(r+1)$-dimensional subspace of $\kappa(C_i)$
for all $i$, such that $r_{\C,\H}(\D) = r$.  (Recall that we always have $r_{\C,\H}(\D) \leq r_{\C}(\D)$.)  If so, we say that $\D$ can be {\em completed to a limit $\g^r_d$} on $\C$.
Moreover, Theorem~\ref{thm:limitseriesgeneral} shows that if $X$ is a smooth proper curve over $\K$, $(D,L_\eta)$ is a $\g^r_d$ on $X$,
and $\C \fX$ is the metrized complex associated to a strongly semistable model $\fX$ for $X$, then the associated pair $(\tau^{\C \fX}_*(D), \mathcal H)$ 
(which we call the {\em specialization} of $(\D,\H)$) is a limit $\g^r_d$ on $\C \fX$. 
We say that a limit $\g^r_d$ $(\D,\H)$ on a metrized complex $\C$ over $\kappa$ is {\em smoothable} if there exists a smooth proper curve $X$ over $\K$ such that
$\C=\C\fX$ for some strongly semistable model $\fX$ of $X$ and $(\D,\H)$ arises the specialization of a limit $\g^r_d$ on $X$.
Note that by \cite[Theorem 3.24]{ABBR}, every ${\rm val}(\K^*)$-rational metrized complex $\C$ over $\kappa$ is of the form $\C=\C\fX$ for some smooth proper curve $X$ over $\K$ and some strongly semistable model $\fX$ of $X$, where {\em ${\rm val}(\K^*)$-rational} means that the edge lengths in $G$ all belong to the value group of $\K$.

\medskip

In this section, we show by example that not every divisor $\D$ of degree $d$ and rank $r$ on a metrized complex $\C$ can be completed to a limit $\g^r_d$,
and not every limit $\g^r_d$ on a ${\rm val}(\K^*)$-rational metrized complex $\C$ is smoothable.

\medskip

\begin{example}[A degree $d$ rank $r$ divisor which cannot be completed to a limit $\g^r_d$]
\label{ex:NotCompletable}  
Consider the weighted graph $\Gamma$ depicted in 
Figure~\ref{fig:nog12} (with arbitrary edge lengths), 
and let $\C$ be a metrized complex whose underlying weighted metric graph is $\Gamma$ 
(so each $C_{v_i}$ has genus $1$ and $C_v$ has genus $0$).
For any point $p$ in $C_v \simeq \mathbb P^1$, consider the degree $2$ divisor $\mathcal D = 2(p)$.
We claim that $\mathcal D$ has rank $1$ on $\C$ but cannot be completed to a limit $\g^1_2$.
The fact that $\mathcal D$ has rank $1$ is straightforward and left to the reader.
Assume for the sake of contradiction that there exists a limit $\g^1_2$ 
$(\H, \mathcal D)$ on $\C$. Let $x_1, x_2, x_3$ (resp. $y_1,y_2,y_3$) be the points of 
$C_v = \mathbb P^1$ (resp. $C_i := C_{v_i}$) corresponding to the three edges of $\Gamma$. Fix $i \in \{ 1,2,3 \}$, choose a point $q_i \neq y_i$ in $C_i$ and let $\mathcal E = (q_i)$. For any rational function $\f$ on $\C$ such that $\mathcal D - \mathcal E + \div(\f) \geq 0$, the fact that $C_i$ has genus one
implies that $\div_{v_i}(f_\Gamma) \geq 2(y_i)$, since there is no rational function on $C_i$ whose divisor is $(q_i)-(y_i)$.
Therefore $\div(f_\Gamma) = 2(v_i) - 2(v)$.  Restricting $\f$ to $C_v$ gives a rational function $f_i$ on $C_v$ with
$\div(f_i) = 2(x_i) - 2(p)$.  Since $(\H, \mathcal D)$ is by assumption a limit $\g^1_2$, we must have $f_1,f_2,f_3 \in H_v$.
But a simple calculation shows that the $f_i$ are linearly independent, implying that $\dim_\k (H_v) \geq 3$, a contradiction.
\end{example}

\begin{figure}[!tb]
\scalebox{.48}{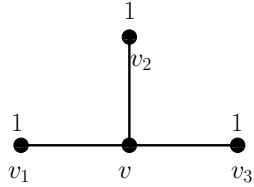}
\caption{A weighted metric graph underlying a metrized complex possessing a degree $2$ rank $1$ divisor which
cannot be completed to a limit $\g^1_2$.}
\label{fig:nog12}
\end{figure}

\begin{remark}
Suppose the residue field $\kappa$ of $\K$ has characteristic different from $2$. 
Then there exists a metrized complex $\C$ over $\kappa$ (with all curves $C_v$ isomorphic to $\mathbb P^1$) 
and a divisor $\D$ of degree 2 and rank 1 on $\C$ which cannot be completed to a limit $\g^1_2$.
Indeed, any hyperelliptic metric graph which cannot be realized as the skeleton of a hyperelliptic curve over $\K$ provides an example of this phenomenon
(and one can construct many examples of such metric graphs, using the characterization of those hyperelliptic metric graphs which can be lifted to a hyperelliptic curve over $\K$ given in~\cite{caporaso-gonality} and~\cite{ABBR}.)
Here is a sketch of the argument.  It is not difficult to show that
a limit $\g^1_2$ on a metrized complex $\C$ gives rise, in a natural way,
to a finite harmonic morphism of degree two (in the terminology of~\cite[Definition 2.22]{ABBR}) 
from $\C$ to a metrized complex of genus zero .  By the lifting result \cite[Corollary 1.5]{ABBR},
such a finite harmonic morphism can be lifted to a
degree two map from a smooth proper curve $X/\K$ with associated metrized complex $\C$
to $\mathbb P^1$, which implies that the underlying metric graph of $\C$ is the skeleton of a hyperelliptic curve over $\K$. 
\end{remark}

\begin{example}[A limit $\g^r_d$ which is not smoothable]
\label{ex:NotSmoothable}
This phenomenon occurs already in the compact type case: the original paper of Eisenbud-Harris on limit linear series \cite{EH86} 
provides an example (Example 3.2 of {\it loc.~cit.}) of a limit $\g^2_4$ which cannot be smoothed; we recall their example here.

The metrized complex $\C$ in question is associated to a semistable curve $X_0$ with two components meeting transversely at a point $p$; one of the
components is a hyperelliptic curve $Y$ of genus at least $4$ and the other component $Z$ has genus zero.  The point $p$ is a ramification point of the 
$\g^1_2$ on $Y$, and without loss of generality we may assume that $p$ corresponds to the point $0$ on $\PP^1 \cong Z$.
The dual graph $G$ of $X_0$ has two vertices $v_Y$ and $v_Z$ connected by a single edge, which may take to be of length one in $\Gamma$.
Now define a divisor $\D$ on $\C$ by setting $\mathcal D = 4(\infty)$ for the point 
$\infty\in \mathbb P^1 \simeq Z$. 
If we define $\H = \{ H_Y,H_Z \}$, where $H_Y = H^0(Y,4(p))$ and $H_Z$ is the span of $1,t^2+t^3,t^4$ (with $t$ a local parameter at $0$ on $\PP^1$),
then by \cite[Example 3.2]{EH86} and Theorem~\ref{thm:limitseries} above, the pair 
$(\D,\H)$ defines a limit $\g^2_4$ on $\C$.  (It is also not hard to show this directly using our definitions.)
This limit $\g^2_4$ cannot be smoothed.
\end{example}

\subsection{Application to bounding the number of rational points on curves}
\label{ChabautySection}

In this section, we explain how limit linear series on 
metrized complexes of curves can be used to illuminate the proof of a recent theorem due 
to E. Katz and D. Zureick-Brown \cite{KZB} and put it into a broader context.

\medskip

Let $K$ be a number field and suppose $X$ is a smooth, proper, 
geometrically integral curve over $K$ of genus $g \geq 2$.
Let $J$ be the Jacobian of $X$, which is an abelian variety of dimension $g$ defined over $K$.
If the Mordell-Weil rank $r$ of $J(K)$ is less than $g$, 
Coleman \cite{Co} adapted an old method of Chabauty to prove that if
$p>2g$ is a prime which is unramified in $K$ and $\mathfrak{p}$ is a prime of 
good reduction for $X$ lying over $p$, then $\# X(K) \leq \# \bar{X}({\mathbf F}_{\mathfrak p}) + 2g - 2$.
Here $\bar{X}$ denotes the special fiber of a smooth proper model for $X$ over the completion ${\mathcal O}_{\mathfrak p}$ of ${\mathcal O}_K$ at $\mathfrak{p}$ and ${\mathbf F}_{\mathfrak p} = {\mathcal O}_K / \mathfrak{p}$.
Stoll \cite{St} improved this bound by replacing $2g-2$ with $2r$.  Lorenzini and Tucker \cite{LT} (see also \cite{MP})
proved the same bound as Coleman without
assuming that $X$ has good reduction at $\mathfrak{p}$; in their bound, $\bar{X}({\mathbf F}_{\mathfrak p})$ is replaced by $\bar{\mathfrak X}^{\rm sm}({\mathbf F}_{\mathfrak p})$
where ${\mathfrak X}$ is a proper regular model for $X$ over ${\mathcal O}_{\mathfrak p}$ and $\bar{\mathfrak X}^{\rm sm}$ is the smooth locus of the special fiber of ${\mathfrak X}$.
Katz and Zureick-Brown combine the improvements of Stoll and Lorenzini-Tucker by proving:

\begin{thm}
\label{thm:Katz-ZB}
Let $K$ be a number field and suppose $X$ is a smooth, proper, geometrically integral curve over $K$ of genus $g \geq 2$.
Suppose the Mordell-Weil rank $r$ of $J(K)$ is less than $g$, and that $p>2g$ is a prime which is unramified in $K$.
Let $\mathfrak{p}$ be a prime of ${\mathcal O}_K$ lying over $p$ and let ${\mathfrak X}$ be a proper regular model for $X$ over ${\mathcal O}_{\mathfrak p}$.
Then 
\[
\# X(K) \leq \# \bar{\mathfrak X}^{\rm sm}({\mathbf F}_{\mathfrak p}) + 2r.
\]
\end{thm}

In order to explain the main new idea in the paper of Katz and Zureick-Brown, we first quickly recall the basic arguments used by Coleman, Stoll, and Lorenzini-Tucker.
(See \cite{MP} for a highly readable and more detailed overview.)
Assume first that we are in the setting of Coleman's paper, so that $r < g$, $p>2g$ is a prime which is unramified in $K$, and $X$ has good reduction at the prime ${\mathfrak p}$ lying over $p$.
Fix a rational point $P \in X(K)$ (if there is no such point, we are already done!).  
Coleman associates to each regular differential $\omega$ on $X$ over $K_{{\mathfrak p}}$ (the ${\mathfrak p}$-adic completion of $K$) 
a ``definite $p$-adic integral'' $\int_P^Q \omega \in K_{{\mathfrak p}}$.  If $V_{\rm chab}$ denotes the vector space of all $\omega$ such that 
$\int_P^Q \omega = 0$ for all $Q \in X(K)$, Coleman shows that ${\rm dim} V_{\rm chab} \geq g-r > 0$.  
Locally, $p$-adic integrals are obtained by formally integrating a power series expansion for $\omega$ with respect to a local parameter.
Using this observation and an elementary Newton polygon argument, Coleman proves that
\[
\# X(K) \leq \sum_{\widetilde{Q} \in \bar{X}({\mathbf F}_{\mathfrak p})} \left( 1 + n_{\widetilde{Q}} \right),
\]
where $n_{\widetilde{Q}}$ is the minimum over all nonzero $\omega$ in $V_{\rm chab}$ of 
${\rm ord}_{\widetilde{Q}} \widetilde{\omega}$; here $\widetilde{\omega}$ denotes the reduction of a suitable rescaling $c \omega$ of $\omega$ to $\bar{X}$,
where the scaling factor is chosen so that $c \omega$ is regular and non-vanishing along the special fiber $\bar{X}$.
If  we choose any nonzero $\omega \in V_{\rm chab}$,
then the fact that the canonical divisor class on $\bar{X}$ has degree $2g-2$ gives 
\[
\sum_{\widetilde{Q} \in \bar{X}({\mathbf F}_{\mathfrak p})} n_{\widetilde{Q}} \leq \sum_{\widetilde{Q} \in \bar{X}({\mathbf F}_{\mathfrak p})} {\rm ord}_{\widetilde{Q}} \widetilde{\omega}
\leq 2g - 2,
\]
which yields Coleman's bound.

\medskip

Stoll observed that one could do better than this by adapting the differential $\omega$ to the point $\widetilde{Q}$ rather than using the
same differential $\omega$ on all residue classes.
Define the {\it Chabauty divisor}
\[
D_{\rm chab} = \sum_{\widetilde{Q} \in \bar{X}({\mathbf F}_{\mathfrak p})} n_{\widetilde{Q}} (\widetilde{Q}).
\]
Then $D_{\rm chab}$ and $K_{\bar{X}} - D_{\rm chab}$ are both equivalent to effective divisors, so by Clifford's inequality (applied to the smooth proper curve $\bar{X}$) 
we have $r(D_{\rm chab}) := h^0(D_{\rm chab}) - 1 \leq \frac{1}{2} {\rm deg} (D_{\rm chab})$.
On the other hand, by the semicontinuity of $h^0$ under specialization we have $h^0(D_{\rm chab}) \geq {\rm dim} V_{\rm chab} \geq g-r$.
Combining these inequalities gives 
\[
\sum_{\widetilde{Q} \in \bar{X}({\mathbf F}_{\mathfrak p})} n_{\widetilde{Q}} \leq 2r 
\]
which leads to Stoll's refinement of Coleman's bound.

\medskip

Lorenzini and Tucker observed that one can generalize Coleman's bound to the case of bad reduction as follows.
Let ${\mathfrak X}$ be a proper regular model for $X$ over ${\mathcal O}_{\mathfrak p}$ and note that points of $X(K)$ specialize to 
$\bar{\mathfrak X}^{\rm sm}({\mathbf F}_{\mathfrak p})$.
One obtains by a similar argument the bound
\begin{equation}
\label{eq:LTinequality}
\# X(K) \leq \sum_{\widetilde{Q} \in \bar{\mathfrak X}^{\rm sm}({\mathbf F}_{\mathfrak p})} \left( 1 + n_{\widetilde{Q}} \right),
\end{equation}
where $n_{\widetilde{Q}}$ is the minimum over all nonzero $\omega$ in $V_{\rm chab}$ of 
${\rm ord}_{\widetilde{Q}} \widetilde{\omega}$; here $\widetilde{\omega}$ denotes the reduction of (a suitable rescaling of) $\omega$ to the unique irreducible component of the special
fiber of ${\mathfrak X}$ containing $\widetilde{Q}$ and ${\rm dim} V_{\rm chab} \geq g-r > 0$ as before.
Choosing a nonzero $\omega \in V_{\rm chab}$ as in Coleman's bound, the fact that the relative dualizing sheaf for ${\mathfrak X}$ has degree $2g-2$
gives the Lorenzini-Tucker bound.

\medskip

In order to combine the bounds of Stoll and Lorenzini-Tucker, we see that it is natural to form the Chabauty divisor
\[
D_{\rm chab} = \sum_{\widetilde{Q} \in \bar{\mathfrak X}^{\rm sm}({\mathbf F}_{\mathfrak p})} n_{\widetilde{Q}} (\widetilde{Q})
\]
and try to prove, using some version of semicontinuity of $h^0$ and Clifford's inequality, that its degree is at most $2r$.
This is the main technical innovation of Katz and Zureick-Brown, so we state it as a theorem:

\begin{thm}[Katz--Zureick-Brown]
\label{thm:Katz-ZB_Degree_Bound}
The degree of $D_{\rm chab}$ is at most $2r$.
\end{thm}

Combining Theorem~\ref{thm:Katz-ZB_Degree_Bound} with \eqref{eq:LTinequality} yields Theorem~\ref{thm:Katz-ZB}.
As noted by Katz and Zureick-Brown, if one makes a base change from $K_{{\mathfrak p}}$ to an extension field $K'$ over which there is a regular semistable model ${\mathfrak X'}$ for $X$ dominating the base change of ${\mathfrak X}$, 
then the corresponding Chabauty divisors satisfy $D'_{\rm chab} \geq D_{\rm chab}$.
(Here $D'_{\rm chab}$ is defined relative to the $K'$-vector space $V'_{\rm chab} = V_{\rm chab} \otimes_K K'$; one does not want to look at the Mordell-Weil group of $J$ over extensions of $K$.)
In order to prove Theorem~\ref{thm:Katz-ZB_Degree_Bound}, we may therefore assume that $\mathfrak X$ is a regular semistable model for $X$ (and also that the residue field of $K'$ is algebraically closed).

\medskip

Let $d={\rm deg}(D_{\rm chab})$.  
We now explain how to prove that $d \leq 2r$ when $\mathfrak X$ is a semistable regular model using limit linear series on metrized complexes of curves. 
Our proof is different from (and arguably streamlined than) the one given in \cite{KZB}.

\begin{proof}[Proof of Theorem~\ref{thm:Katz-ZB_Degree_Bound}]
Let $s = \dim_{K'} V'_{\rm chab} - 1 \geq g-r-1 \geq 0$.  We can identify $V'_{\rm chab}$ with an $(s+1)$-dimensional space $W$ of rational functions on $X$ in the usual way by identifying $H^0(X,\Omega^1_X)$ with 
$L(K_X) = \{ f \; : \; {\rm div}(f) + K_X \geq 0 \}$ for a canonical divisor $K_X$ on $X$.
The divisor $D_{\rm chab}$ on $\bar{\mathfrak X}^{\rm sm}$ defines in a natural way a divisor $\D = \sum_v D_v$ of degree $d$ on the metrized complex $\C\mathfrak X$ associated to $\mathfrak X$, where $D_v$ is the restriction of $D_{\rm chab}$ to $C_v$.
We promote the divisor $K_{\C\mathfrak X} - \D$ to a limit linear series 
$\L = (K_{\C\mathfrak X} - \D, \{ H_v \})$ on $\C\mathfrak X$ by defining $H_v$ to be the reduction of $W$ to $C_v$ for each $v \in V(G)$.  
By the definition of $D_{\rm chab}$, each element of $H_v$ vanishes to order at least $n_{\widetilde{Q}}$ at each point $\widetilde{Q}$ in ${\rm supp}(D_v)$.
By Corollary~\ref{cor:introcanonspec} and Theorem~\ref{thm:limitseriesgeneral2}, $\L$ is a limit $\g^{s}_{2g-2-d}$ on $\C\mathfrak X$.
In particular, we must have $r_{\C\mathfrak X}(K_{\C\mathfrak X} - \D) \geq g-r-1$.
On the other hand, Clifford's inequality for metrized complexes (Theorem~\ref{thm:MCClifford}) implies that
$r_{\C\mathfrak X}(K_{\C\mathfrak X} - \D) \leq \frac{1}{2}(2g-2-d)$.
Combining these inequalities gives $d \leq 2r$ as desired.
\end{proof}

\appendix
\section{Rank-determining sets for metrized complexes}\label{sec:rank-determining}
We retain the terminology from Section~\ref{sec:basics}. Let $\C$ be a metrized complex of algebraic curves, $\Gamma$ the underlying metric graph, 
$G=(V,E)$ a model of $\Gamma$ and $\{C_v\}_{v\in V}$ 
the collection of smooth projective curves over $\k$ corresponding to $\C$.
In this section, we generalize some basic results concerning rank-determining 
sets~\cite{Luo, HKN} from metric graphs to metrized complexes
by following and providing complements to the arguments of~\cite{amini} 
(to which we refer for a more detailed exposition).

\medskip

Let $\R$ be a set of geometric points of $\C$ (i.e., a subset of $\bigcup_{v\in V} C_v(\k)$). 
The set $\R$ is called {\it rank-determining} if for any divisor $\mathcal D$ on $\C$, $r_\C(\mathcal D)$ coincides with $r_{\C}^\R(\mathcal D)$, defined as the largest integer $k$ such that $\mathcal D - \mathcal E$ 
is linearly equivalent to an effective divisor
for all degree $k$ effective divisors $\mathcal E$ on $\C$ with support in $\R$.  In other words, $\R$ is rank-determining
if in the definition of rank given in Section~\ref{sec:basics}, one can restrict to effective divisors $\mathcal E$ with support in $\R$.

\medskip

The following theorem is a common generalization of (a) Luo's theorem~\cite{Luo} (see also~\cite{HKN}) that
$V$ is a rank-determining set for any loopless model $G=(V,E)$ of a metric graph $\Gamma$ and (b) the classical fact
(also reproved in \cite{Luo}) that for any smooth projective curve $C$ of genus $g$ over $\k$, every subset of $C(\k)$ of size $g+1$ is rank-determining. 

\begin{thm}\label{thm:f-width} Let $\C$ be a metrized complex of algebraic curves, and suppose that the given model $G$
of $\Gamma$ is loopless.  Let $\R_v \subset C_v(\k)$ be a subset of size $g_v+1$ and let $\R = \cup_{v\in V} \R_v$. Then $\R$ is a rank-determining subset of $\C$.
\end{thm}

Let $\mathcal D$ be a divisor on $\C$. For any point $P \in \Gamma$, 
let $\D^P$ be the quasi-unique $P$-reduced divisor on $\C$ linearly equivalent to $\mathcal D$,
and denote by $D^P_\Gamma$ (resp. $D^P_v$) the $\Gamma$-part (resp. $C_v$-part) of $\D^P$.

\begin{lemma}\label{lem:rankone}
A divisor $\mathcal D$ on $\C$ has rank at least one if and only if
\begin{itemize}
\item[(1)] For any point $P$ of $\Gamma$, $D^P_\Gamma( P )\geq 1$, and 
\item[(2)] For any vertex $v\in V(G)$, the divisor $D^v_v$ has rank at least one on $C_v$. 
\end{itemize}
\end{lemma}
\begin{proof}
The condition $r_\C(\mathcal D) \geq 1$ is equivalent to requiring that $r_\C(\mathcal D - \mathcal E ) \geq 0$ for every effective divisor $\mathcal E$ of degree $1$ on $\C$. 
For $P \in \Gamma\setminus V$, the divisor $\mathcal D -( P )$ has non-negative rank if and only if $D^P_\Gamma( P )\geq 1$ (by Lemma~\ref{lem:reduced}). Similarly, for $v \in V$ and $x\in C_v(\k)$, 
the divisor $ \mathcal D - (x)$ has non-negative rank in $\C$ if and only if $D^v_\Gamma(v) \geq 1$ and $D^v_v - (x)$ has non-negative rank on $C_v$ (by Lemma~\ref{lem:reduced}). These are clearly equivalent to (1) and (2).  
\end{proof}

\begin{lemma}\label{lem:rank-determ}
A subset  $\R \subseteq \bigcup_{v\in V} C_v(\k)$ which has non-empty intersection with each $C_v(\k)$ is rank-determining  if and only if for every divisor $\mathcal D$ of non-negative rank on $\C$, the following two assertions are equivalent:
\begin{itemize}
\item[(i)] $r_\C(\mathcal D) \geq 1$.
\item[(ii)] For any vertex $u  \in V$, and for any point $z \in \R\, \cap\, C_u(\k)$, $D_u^u -(z)$ has non-negative rank on $C_u$.
\end{itemize}
\end{lemma}
\begin{proof}  In view of Lemma~\ref{lem:rankone}, for a rank-determining set the two conditions (i) and (ii) are equivalent. Suppose now that  (i) and (ii) are equivalent for any divisor $\mathcal D$ on $\C$. By induction on $r$, 
we prove that $r_\C(\mathcal D)\geq r$ if and only if for every effective divisor $\mathcal E$ of  degree $r$ with support in $\R$,  
$r_\C(\mathcal D - \mathcal E) \geq 0$. This will prove that $\R$ is rank-determining.

The case $r=1$ follows by the hypothesis and Lemma~\ref{lem:rankone}. 
Supposing now that the statement holds for some integer $r \geq 1$, we prove that it also holds for $r+1$.  

Let $\mathcal D$ be a divisor with the property that $r_\C(\mathcal D -\mathcal E) \geq 0$ 
for every effective divisor $\mathcal E$ of  degree $r+1$ with support in $\R$. 
Fix an effective divisor $\mathcal E$ of degree $r$ with support in $\R$.  
By the base case $r=1$, the divisor $\mathcal D - \mathcal E$ has rank at least $1$ on $\C$ because
$r_\C(\mathcal D -\mathcal E -(x))\geq 0$ for any $x\in \R$. Thus $r_\C(\mathcal D - (x) - \mathcal E) \geq 0$ 
for any point of $|\C|$. This holds for any effective divisor $\mathcal E$ of degree $r$ with support in $\R$, and so
from the inductive hypothesis we infer that  $\mathcal D - (x)$ has rank at least $r$ on $\C$. 
Since this holds for any $x \in |\C|$, we conclude that $\mathcal D$ has rank at least $r+1$.
\end{proof}

Let $\mathcal D$ be a divisor of degree $d$ and non-negative rank on $\C$, and let $D_\Gamma$ and $D_v$ be the $\Gamma$ and $C_v$-parts of $\D$,
respectively.
Define
\[|D_\Gamma| :=\{E\geq 0 \,|\,\, E\in \Div(\Gamma) \textrm{ and } E\sim D_\Gamma \}.\]
Note that $|D_\Gamma|$ is a non-empty subset of the symmetric product $\Gamma^{(d)}$ of $d$ copies of $\Gamma$. 

\medskip

Consider the {\it reduced divisor} map $\Red_\mathcal D : \Gamma \rightarrow \Gamma^{(d)}$ which sends a point $P\in \Gamma$ to 
$D^P_\Gamma$, the $\Gamma$-part of the $P$-reduced divisor $\mathcal D^P$. 
The following theorem extends~\cite[Theorem 3]{amini} to divisors on metrized complexes. 

\begin{thm}
  For any divisor $\mathcal D$ of degree $d$ and non-negative rank on $\C$, the reduced divisor map $\Red_\mathcal D : \Gamma \rightarrow \Gamma^{(d)}$ is continuous. 
\end{thm}

\begin{proof}
This is based on an explicit description of the reduced divisor map 
$\Red_\mathcal D$ in a small neighborhood around any point of $\Gamma$, 
similar to the description provided in~\cite[Theorem 3]{amini} in the context of metric graphs.
We merely give the description by providing appropriate modifications to~\cite[Theorem 3]{amini}, 
referring to {\it loc.~cit.} for more details.

\medskip

Let $P$ be a point of $\Gamma$ and let $\vec \mu$ be a (unit)
tangent direction in $\Gamma$ emanating from $P$. For $\epsilon>0$ sufficiently small, 
we denote by $P+\epsilon \vec \mu$ 
the point of $\Gamma$ at distance $\epsilon$ from $P$ in the direction of $\vec \mu$.  We will describe the restriction of $\Red_\mathcal D$ to the segment 
$[P, P+\epsilon \vec \mu]$ for sufficiently small $\epsilon>0$. 
One of the two following cases can happen:
\begin{itemize}
\item[(1)] For all sufficiently small $\epsilon>0$, the $P$-reduced divisor $\mathcal D^P$ is also $(P+\epsilon\vec \mu)$-reduced.
\end{itemize}
\noindent In this case, the map $\Red_\mathcal D$ is constant (and so obviously continuous) 
on a small segment $[P, P+\epsilon_0 \vec \mu]$ with $\epsilon_0>0$.

\begin{itemize}
\item[(2)] There exists a cut $S$ in $\Gamma$ which is saturated with respect to $\mathcal D^P$ 
such that $P \in \partial S$ and $P+\epsilon \vec \mu \not\in S$ for all sufficiently small $\epsilon > 0$.
\end{itemize}

We note that there is a maximum saturated cut $S$ (i.e., containing any other saturated cut) 
with the property described in (2) (see the proof of~\cite[Theorem 3]{amini} for details). In the following $S$ denotes the maximum saturated cut with property (2).
In this case, there exists an $\epsilon_0>0$ such that for any $0<\epsilon<\epsilon_0$, the reduced divisor 
$\mathcal D^{P+\epsilon \vec \mu}$ has the following description (the proof mimics that of~\cite[Theorem 3]{amini} and is omitted).

\medskip

Let  $\mu,\vec \mu_1,\dots,\vec \mu_s$ be all the different 
tangent vectors in $\Gamma$ (based at the boundary points $P, x_1,\dots,x_s \in \partial S$, respectively)
which are outgoing from $S$.
(It might be the case that $x_i=x_j$ for two different indices $i$ and $j$). 
Let $\gamma_0>0$  be small enough so that for any point $x \in \partial S$ and any tangent vector $\vec \nu$ to $\Gamma$ at $x$ which is outgoing from $S$, the entire segment 
$(x,x+\gamma_0\vec \nu\,]$ 
lies outside $S$ and does not contain any point of the support of $D_\Gamma$.

\medskip

For any  $0<\gamma < \gamma_0$ and any positive integer $\alpha$, 
we will define below a rational function $f_{\Gamma}^{(\gamma,\alpha)}$ on $\Gamma$. 
Appropriate choices of $\gamma = \gamma(\epsilon)$ and $\alpha$ will 
then give the ($P+\epsilon \vec \mu$)-reduced divisor 
$\mathcal D^{P+\epsilon \vec \mu} = \mathcal D^P + \div(\f^{\gamma,\alpha})$, for any 
$\epsilon < \epsilon_0 := \frac{\gamma_0}\alpha$,
where $\f^{\gamma,\alpha}$ is the rational function on $\C$ given by $f_\Gamma^{\gamma,\alpha}$ on $\Gamma$ and $f_v =1$ on each $C_v$. 

\medskip

For $0<\gamma<\gamma_0$ and integer $\alpha\geq 1$, define $f_\Gamma^{(\gamma,\alpha)}$ as follows:

\begin{itemize}
\item $f_\Gamma^{(\gamma,\alpha)}$ takes value zero at any point of $S$;
\item On any outgoing  interval $[x_i,x_i+ \gamma\vec \mu_i]$ from $S$, $f_\Gamma^{(\gamma,\alpha)}$ is linear of slope $-1$;
\item The restriction of $f_\Gamma^{(\gamma,\alpha)}$ 
to the interval $[P,P+ (\frac \gamma \alpha) \vec \mu \,]$ is linear of slope $-\alpha$; 
\item $f_\Gamma^{(\gamma,\alpha)}$ takes value $-\gamma$ at any other point of $\Gamma.$
\end{itemize}

Note that the values of $f_\Gamma^{(\gamma,\alpha)}$ at the points $(x_i+ \gamma\vec \mu_i)$ and $P+ (\frac \gamma \alpha) \vec \mu$ are all equal to $-\gamma$,
so $f_\Gamma^{(\gamma,\alpha)}$ is well-defined.

\medskip 

It remains to determine the values of $\alpha$ and $\gamma$. 
Once the value of $\alpha$ is determined, $\gamma$ will be defined as $\alpha\epsilon$
so that the point $P+(\frac \gamma \alpha)\vec \mu$ coincides with the
point $P+\epsilon \vec \mu$.  We consider the following two cases, depending on whether or not $P$ is a vertex of $G$:

\begin{itemize}
\item If $P \in \Gamma \setminus V$, then $\alpha = D^P_\Gamma( P ) -  \mathrm{outdeg}_S ( P ) +1$. 
(Note that since $S$ is saturated with respect to $D_\Gamma^P$, we have $D^P_\Gamma( P ) \geq  \mathrm{outdeg}_S ( P )$ and thus 
$\alpha \geq 1$.)
\item If $P=v$ for a vertex $v\in V(G)$, let $e_0, e_1, \dots, e_l$ be the outgoing edges at $v$ with respect to $S$, and consider the points $x^{e_0}_v, x^{e_1}_v, \dots, x^{e_l}_v$ in $C_v(\k)$ indexed by these edges.  Suppose in addition that $e_0$ is the edge which corresponds to the 
tangent direction $\vec\mu$. 
Since $S$ is a saturated cut with respect to $D^v_\Gamma$, 
the divisor $D_v - \div_v(\partial S) = D_v - \sum_{i=0}^l (x^{e_i}_v)$ has non-negative rank in $C_v$. 
Define $\alpha$ to be the largest integer $n\geq 1$ such that 
$D_v - n (x_v^{e_0}) - \sum_{i=1}^l (x^{e_i}_v)$ has non-negative rank. 
\end{itemize}

Now for any $0\leq \epsilon<\epsilon_0= \frac {\gamma_0}\alpha$, the divisor 
$\mathcal D^{P+\epsilon \vec \mu}$ is $(P+\epsilon \vec \mu)$-reduced.
(The argument is similar to \cite[Proof of Theorem 3]{amini}.) 
It follows immediately that the reduced divisor map is continuous on the interval $[P, P+\epsilon_0 \vec \mu)$,
and the result follows.
\end{proof}

We are now ready to give the proof of Theorem~\ref{thm:f-width}.
\begin{proof}[Proof of Theorem~\ref{thm:f-width}] By Lemma~\ref{lem:rank-determ}, it is enough to check the equivalence of the following two properties for any divisor $\mathcal D$ on $\C$:
\begin{itemize}
\item[(i)] $r_\C(\mathcal D) \geq 1$.
\item[(ii)] For any $u  \in V$ and any point $z \in \R_u = \R\, \cap\, C_u(\k)$, the divisor $D_u^u -(z)$ has non-negative rank on $C_u$.
\end{itemize}

It is clear that (i) implies (ii). So we only need to prove that (ii) implies (i).
In addition, by Lemma~\ref{lem:rankone}, Property (i) is equivalent to:
\begin{itemize}
\item[(1)] for any point $P$ of $\Gamma$, $D^P_\Gamma( P )\geq 1$; and 
\item[(2)] for any vertex $v\in V(G)$, the divisor $D^v_v$ has rank at least one on $C_v$. 
\end{itemize}

\medskip

So it suffices to prove that ${\rm (ii)}\Rightarrow (1)$ and $(2)$. Since cardinality of $\R_v$ is $g_v+1$, $\R_v$ is rank-determining in $C_v$. Therefore, (ii) implies $(2)$. We now show that $(2)$ implies $(1)$. 
Let $\Gamma_0$ be the set of all $P\in\Gamma$ such that $D^P_\Gamma ( P )\geq 1$. By the continuity of the map $\Red_\mathcal D$, $\Gamma_0$ is a closed subset of $\Gamma$. In addition, since $D^v_v$ has rank at least 
one on $C_v$ and $D_\Gamma^v( v ) = \deg(D_v^v)$ for every vertex $v\in C$, we have $V \subset \Gamma_0$. This shows that $\Gamma \setminus \Gamma_0$ is a disjoint union of open segments contained in edges of $G$. 
Suppose for the sake of contradiction that $\Gamma_0 \subsetneq \Gamma$, and let $I = (P,Q) $ be a non-empty segment contained in the edge $\{u,v\}$ of $G$ such that 
$I \cap \Gamma_0 = \emptyset$. 

\medskip

{\bf Claim:} $\Red_\mathcal D$ is constant on the closed interval $[P,Q]$.  

\medskip

To see this, note that for any point $Z \in [P,Q]$ and any tangent direction $\vec \mu$ for which $Z+\epsilon \vec \nu \in [P,Q]$ for all sufficiently small $\epsilon>0$, 
we are always in case (1) in the description of $\Red_\mathcal D$.  Otherwise, there would be an integer $\alpha>0$ such that $\mathcal D^{Z+\epsilon \vec \mu} = \mathcal D^{Z} +  \div(\f^{(\epsilon \alpha, \alpha)})$ 
for all sufficiently small $\epsilon > 0$. In particular, this would imply (by the definition of $f^{(\eta,a)}$) that 
$D_\Gamma^{Z+\epsilon\vec \mu}( Z+\epsilon \vec \mu ) = \alpha \geq 1$, which implies 
that $Z+\epsilon \vec \nu \in \Gamma_0$, a contradiction.  This proves the claim.

\medskip

A case analysis (depending on whether $P$ and $Q$ are vertices or not) shows that for a point $Z \in (P,Q)$, the cut $S = \Gamma \setminus (P,Q)$ is saturated for $\mathcal D^P =\mathcal D^Q$. 
Since $\mathcal D^Z = \mathcal D^P=\mathcal D^Q$, and $S$ does not contain $Z$, this contradicts the assumption that $\mathcal D^Z$ is $Z$-reduced.
 \end{proof}

\medskip

Theorem~\ref{thm:f-width} has the following direct corollaries. 

\begin{cor} 
\label{cor:f-widthG}
Let $\mathcal G$ be a subgroup of $\mathbb R$ which contains all the edge lengths in $G$. For any divisor $\mathcal D \in \Div(\C)_\mathcal G$, we have 
$$r_{\C, \mathcal G}(\mathcal D) = r_\C(\mathcal D).$$
\end{cor}
\begin{proof}
Fix a rank-determining set $\R \subset \cup_{v\in V} C_v(\k)$ as in Theorem~\ref{thm:f-width}. Since $\R$ is rank-determining and any effective divisor $\mathcal E$ with support in $\R$ obviously belongs to $\Div(\C)_\mathcal G$, 
to prove the equality of $r_{\C, \mathcal G}(\mathcal D)$ and $r_\C(\mathcal D)$ it will be enough to show that the two statements $r_{\C, \mathcal G}(\mathcal D)\geq 0$  and  $r_\C(\mathcal D)\geq 0$ are equivalent. 
Obviously, the former implies the latter, so we only need to show that if $r_\C(\mathcal D)\geq 0$ then $r_{\C, \mathcal G}(\mathcal D)\geq 0$. Let $v$ be a vertex of $G$ and $\mathcal D^{v}$ the $v$-reduced divisor linearly equivalent to 
$\mathcal D$. By Lemma~\ref{lem:reduced}, $r_\C(\mathcal D)\geq 0$ is equivalent to $r_{C_v}(D^{v}_{v})\geq 0$. Now let $\mathcal D$ be an element of 
$\Div(\C)_\mathcal G$ with $r_{C_v}(D^v_v)\geq 0$. Since $v\in V$ and $\mathcal G$ contains all the edge-lengths in $G$, it is easy to see that $\mathcal D$ and $\mathcal D^v$ differ by the divisor of a rational function $\f$ with support in 
$\Div(\C)_{\mathcal G}$. In other words, $\mathcal D \sim \mathcal D^v$ in $\Div(\C)_\mathcal G$. 
Since $\mathcal D^v$ is linearly equivalent to an effective divisor in $\Div(\C)_\mathcal G$ (with constant rational function on $\Gamma$), we conclude that $r_{\C,\mathcal G}(D)\geq 0$. \end{proof}

\begin{cor} 
\label{cor:f-widthZ}
Let $\C X_0$ be the regularization of a strongly semistable curve $X_0$ over $\k$. Let $\L$ be a line bundle on 
$X_0$ corresponding to a divisor $\mathcal D \in \Div(\C)$.  Then
$r_{c} (\L) = r_{\C X_0}(\mathcal D)$.
\end{cor}

\begin{proof} 
This follows from the previous corollary with $\mathcal G = \mathbb Z$. 
\end{proof}

\end{document}